\documentclass[
  reqno,          
  tbtags,         
  intlimits,      
  oneside,
  letterpaper,
  12pt
]{amsart}

\usepackage[T1]{fontenc}     
\usepackage[english]{babel}  
\usepackage{lmodern}         
\usepackage{microtype}       
\usepackage{setspace} 
\usepackage[margin=1in]{geometry} 
\usepackage[indent, skip=5pt plus 2pt minus 1pt]{parskip} 

\usepackage{mathtools}  
\usepackage{amssymb, amsfonts, amsthm, amscd} 

\usepackage[mathscr]{eucal}   

\usepackage{graphicx}
\usepackage{subcaption}   
\usepackage{wrapfig}
\graphicspath{{./images/}} 

\usepackage{tikz}
\usepackage{tikz-cd}      

\usepackage{enumitem}     
\usepackage[title]{appendix} 

\usepackage{xcolor}
\definecolor{shamrockgreen}{RGB}{0,158,97}
\definecolor{forestgreen}{RGB}{34,139,34}
\definecolor{darkblue}{RGB}{0,0,139}

\definecolor{authorJ}{RGB}{247,127,0}   
\definecolor{authorL}{RGB}{35,35,255}   
\definecolor{authorP}{RGB}{254,33,139}  

\numberwithin{equation}{section} 

\theoremstyle{plain} 
\newtheorem{theorem}{Theorem}[section]
\newtheorem{proposition}[theorem]{Proposition}
\newtheorem{lemma}[theorem]{Lemma}
\newtheorem{corollary}[theorem]{Corollary}

\theoremstyle{definition} 
\newtheorem{definition}[theorem]{Definition}

\theoremstyle{remark} 
\newtheorem{remark}[theorem]{Remark}

\usepackage{csquotes} 
\usepackage[
  backend=biber,
  style=alphabetic,
  sorting=ynt
]{biblatex}
\usepackage{hyperref}
\hypersetup{
  colorlinks=true,
  linkcolor=forestgreen,
  citecolor=blue,
  urlcolor=darkblue,
  pdftitle={Working paper},
  pdfpagemode=UseNone 
}

\usepackage[nameinlink,noabbrev]{cleveref} 

\usepackage{setspace}

\newcommand{\bbR}{\mathbb{R}}

\newcommand{\bbZ}{\mathbb{Z}}
\newcommand{\bbN}{\mathbb{N}}

\newcommand{\cL}{\mathcal{L}}

\newcommand{\mrm}[1]{\mathrm{#1}}

\newcommand{\mcl}[1]{\mathcal{#1}}

\newcommand{\rd}{\mathrm{d}}

\DeclareMathOperator{\Av}{Av}
\DeclareMathOperator{\Lip}{Lip}

\newcommand{\tX}{\widetilde{X}}
\newcommand{\tx}{\widetilde{x}}
\newcommand{\tv}{\widetilde{v}}

\newcommand{\LtX}{\mathcal{L}_{\tX}}
\newcommand{\hhat}[1]{\widehat{#1}}

\usepackage{xparse}
\NewDocumentCommand{\Iup}{ s o m m }{%
  \ensuremath{%
    I^{\uparrow}\IfNoValueTF{#2}{}{_{#2}}%
    \IfBooleanTF{#1}{\!\left(#3,\,#4\right)}{(#3,\,#4)}%
  }%
}
\NewDocumentCommand{\Idown}{ s o m m }{%
  \ensuremath{%
    I^{\downarrow}\IfNoValueTF{#2}{}{_{#2}}%
    \IfBooleanTF{#1}{\!\left(#3,\,#4\right)}{(#3,\,#4)}%
  }%
}

\newcommand{\Teichmuller}{Teichm\"uller}

\usepackage{xcolor}
\definecolor{CommentBlue}{HTML}{5DADE2}      
\definecolor{CommentBlueBg}{HTML}{EAF4FF}    

\AtBeginDocument{\setcounter{page}{1}}

\usepackage[foot]{amsaddr} 

\newcommand{\commonaffil}{Department of Mathematics, University of Chicago, USA}
\address{\commonaffil}

\author{Polina Baron}
\email{polina.baron.22.10@gmail.com}

\author{Elizaveta Shuvaeva}
\email{elizshuv@gmail.com}

\title[Unique ergodicity of branched covers of translation surfaces]{Unique ergodicity of branched covers\\ of translation surfaces}

\begin{document}

\begin{abstract}
Let $X$ be a finite-area translation surface whose vertical flow is uniquely ergodic. Given a slit joining two nonsingular points of $X$, one can form a branched cyclic cover by gluing $\mathrm{N}$ copies of $X$ crosswise along the slit. We study when the vertical flow on the resulting cover is uniquely ergodic.

We first prove a geometric criterion for unique ergodicity of the branched cover. We show that if, for a sequence of times along the Teichm\"uller geodesic, one endpoint of the slit is contained in embedded Euclidean disks of uniformly positive radius that avoid the other endpoint, then the branched cover is uniquely ergodic. The proof uses the special symmetry of the cover together with an analysis of forward and backward generic points for the vertical flow. 

We then show that this criterion applies for Lebesgue-almost every choice of slit endpoint under a natural geometric hypothesis on the Teichm\"uller orbit of $X$, namely a uniform lower bound for the embedded radius along a subsequence. Finally, we give sufficient conditions for such a lower bound in terms of the cylinder geometry of $g_tX$, introducing the notion of pipe cylinders and proving that embedded disks of definite size must exist. 

As a consequence, for the class of uniquely ergodic translation surfaces, almost every slit produces a uniquely ergodic branched $\mathrm{N}$-cover.
\end{abstract}

\maketitle

{\small 
\tableofcontents
}


\section{Introduction}
A \emph{flat surface} \cite{Zorich2006,Wright2014,Matheus2018,
Massart2021,Filip2024}  is a Riemann surface with a fixed holomorphic 1-form. It can be obtained by identifying the sides of a polygon in the Euclidean plane by translations, so it can be called a \emph{translation surface} with a \emph{vertical flow}. It is also called an \emph{Abelian differential}. The moduli spaces of translation surfaces provide a rich interface between low-dimensional dynamics, \Teichmuller~theory, and algebraic geometry.

A translation surface can have one or more ergodic measures for the straight-line flow in the vertical direction. Since the straight-line flow preserves the natural Lebesgue measure, unique ergodicity (the case when the number of ergodic measures is $1$) of the flow implies that the unique invariant probability measure is the normalized Lebesgue measure. Masur \cite{Masur1982} and Veech \cite{Veech1982} independently proved that almost all  (with respect to the Euclidean Lebesgue measure in period coordinates, restricted to the unit-area hypersurface $\mathcal{H}_1(\alpha)$ and normalized to be finite) translation surfaces are uniquely ergodic. Later, Masur \cite{Masur1992HD} came up with a geometric criterion that guaranteed unique ergodicity of a translation surface. This criterion was improved by Cheung--Eskin \cite{CheungEskin2007FIC} and Trevino \cite{Trevino2014FiniteArea}. Cheung--Masur \cite{CheungMasur2006Divergent} construct a quadratic differential whose vertical foliation is uniquely ergodic but whose Teichmüller geodesic diverges, proving that the Masur criterion is not necessary, only sufficient.

For any quadratic differential, the set of non-ergodic directions and the set of non-uniquely ergodic directions each have Hausdorff dimension at most $\tfrac12$. In particular, if $c(\alpha)$ denotes the Hausdorff dimension of the set of non-uniquely ergodic directions for a Masur--Veech generic surface in the stratum  $\mathcal H(\alpha)$ (called the Masur--Smillie constant), then
$$
c(\alpha)\le \tfrac12
$$
for every stratum $\mathcal H(\alpha)$ \cite{Masur1992HD,MasurSmillie1991HD}. Cheung \cite{Cheung2003} and Cheung--Hubert--Masur \cite{CheungHubertMasur2008} constructed examples where the set of    non-ergodic directions has Hausdorff dimension $\tfrac12$, showing that Masur's   upper bound is sharp. Athreya--Chaika \cite{AthreyaChaika2015} later proved that in $\mathcal H(2)$
    the Masur--Smillie constant is $c(2)=\tfrac12$.  Chaika--Masur \cite{ChaikaMasur2020} showed
    that the set of non-uniquely ergodic $d$-IETs has Hausdorff codimension
    $\tfrac12$.

For the vertical flow on a genus-$g$ translation surface, if the flow is minimal, then it has at most $g$ ergodic invariant probability measures \cite{Katok1973}. Note that, for a fixed direction, a translation surface decomposes into minimal components and periodic cylinders. In particular, if the directional flow has a periodic cylinder, then it carries uncountably many ergodic invariant probability measures, one on each periodic orbit. This is why the finiteness theorem of Katok applies not to arbitrary directional flows, but to minimal translation flows. This bound is sharp \cite{Sataev1975}. Early examples of minimal but non-uniquely ergodic straight-line flows (and interval exchange transformations) are due to Veech \cite{Veech1969} (not phrased as an IET), Keane \cite{Keane1977},  Sataev \cite{Sataev1975}, and Keynes–Newton \cite{KeynesNewton1976} (noticed that Veech's example is an IET). The example of Veech (1969) is of particular interest to us, as it inspired the constructions used in the present paper.     
\begin{figure}[h]
\centering
\begin{subfigure}{.49\textwidth}
  \centering
  \includegraphics[width=.7\linewidth]{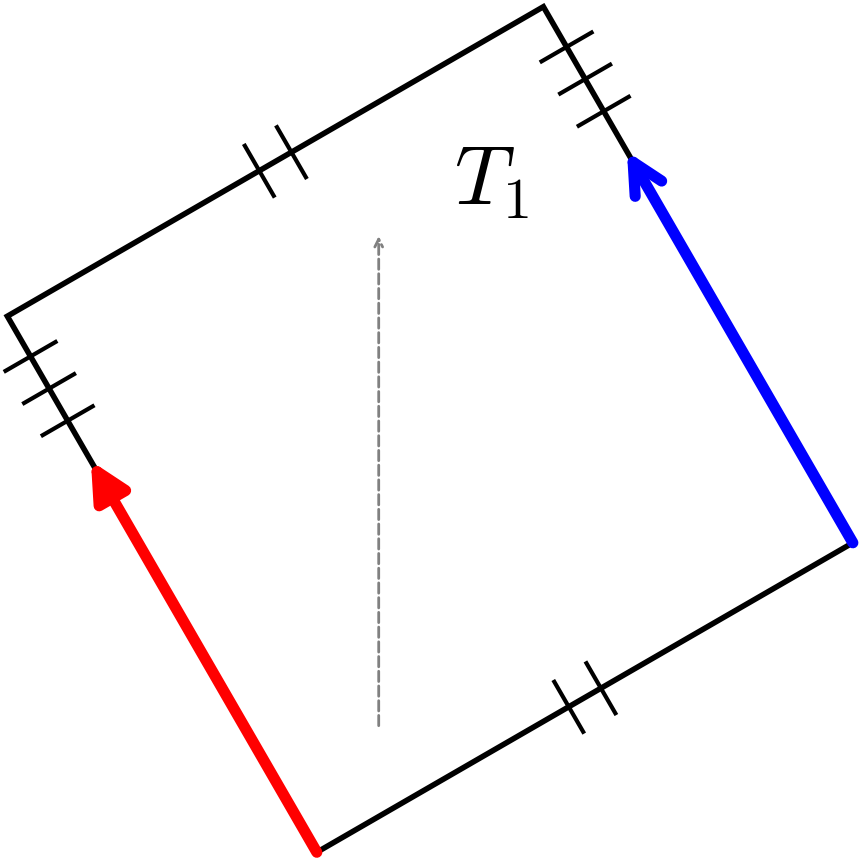}
  \label{}
\end{subfigure}%
\begin{subfigure}{.49\textwidth}
  \centering
  \includegraphics[width=.7\linewidth]{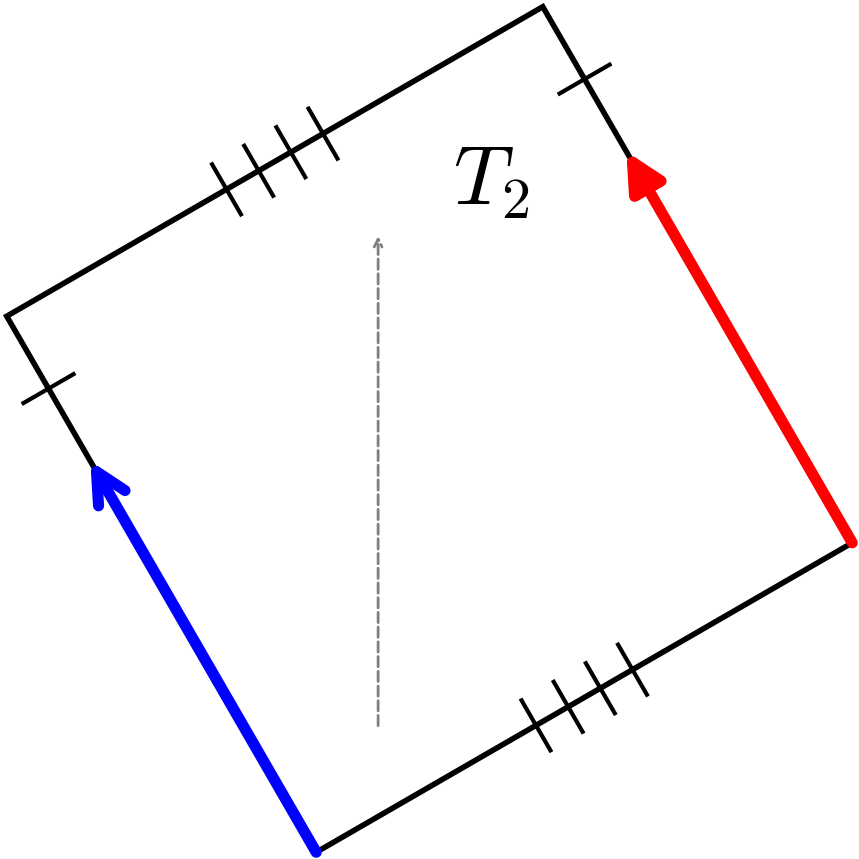}
  \label{}
\end{subfigure}
\caption{Veech's example of a NUE flat surface construction.}
\label{fig__neighorhood}
\end{figure}
 Consider a torus $[0,\,1]\times [0,\,1]$ and take two copies (sheets). Cut a \emph{slit} from $(0,\,0)$ to $(0,\,\beta)$, where $0<\beta<1$. Identify the slit sides across two surfaces. On each sheet, consider the flow generated by the vector $(\alpha,\,1)$, where $\alpha$ is \emph{well-approximable}: $$\liminf _{n \rightarrow \infty} n\cdot d(n \alpha,\, \mathbb{Z})=0.$$ 
Veech gives a condition on $\beta$ (given a fixed $\alpha$) such that the resulting translation surface is minimal but not uniquely ergodic, with exactly $2$ ergodic invariant measures $\mu^{-}, \mu^{+}$ and where Lebesgue measure $\cL$ satisfies $\cL=\frac{1}{2}\left(\mu^{-}+\mu^{+}\right)$. This construction was later generalized by Ferenczi and Hubert \cite{FerencziHubert2024Veech1969}, who give a condition for unique ergodicity of $\bbZ/n\bbZ$  extensions of a rotation of angle $\alpha$  when the initial interval exchange transformation is linearly recurrent, and there are one or two marked points.

To cut a given surface along a slit and reglue several copies cyclically along the cut is a standard way to build new translation surfaces. This construction is simple to describe, but its dynamical effect is subtle: even when the original surface has uniquely ergodic vertical flow, the resulting branched cover may in principle acquire additional invariant measures, like in the example of Veech. The goal of this paper is to understand when this does not happen.

To be more precise, let $X$ be a finite-area translation surface with uniquely ergodic vertical flow,
and let $s$ be a straight non-vertical slit joining two nonsingular points
$P,Q\in X$. Gluing $\mathrm{N}$ copies of $(X,s)$ cyclically along the slit produces a
branched cyclic $\mathrm{N}$-cover $\tX$ of $X$, branched at $P$ and $Q$. 

\vspace{.15cm}
\textbf{Conjecture (Jon Chaika).} 
For Lebesgue-almost every slit on the surface $X$, the branched cyclic $\mathrm{N}$-cover $\tX$
is uniquely ergodic.
\vspace{.15cm}

This conjecture is proven when $n=2$ and $X$ is a torus: Cheung and Eskin \cite{CheungEskin2007SlowDiv} give a complete characterisation of the set of nonergodic
directions in this case. We prove Chaika's conjecture for any $\mathrm{N}$ and any surface genus. Here, the deck transformation exchanging
the sheets strongly constrains the possible ergodic invariant measures. 

\vspace{.15cm}
\textbf{Main result.} 
Let $X$ be a finite-area translation surface whose vertical flow is uniquely ergodic. For almost every pair of points $(P,\,Q)$ on $X$ and for every slit from $P$ to $Q$, the resulting branched $\mathrm{N}$-cover construction is uniquely ergodic. 
\vspace{.15cm}

Unique ergodicity of slit-induced
branched covers is governed by a local geometric condition near one branch point
along the \Teichmuller\ orbit, while the existence of that local geometry can be
verified globally using embedded radius, Delaunay triangulations, and the
distribution of pipe cylinders.

If
the cover is not uniquely ergodic, then the deck transformation exchanges the ergodic measures and, correspondingly, exchanges the sets of generic points for those measures (see Lemma \ref{lemma__n-measures} for details). This observation turns the problem into a quantitative comparison of ergodic averages along carefully chosen orbit segments. To carry this out, in Section \ref{UE-sect2.1}, we develop a framework of forward, backward, and total genericity, together with an ``almost generic'' variant that is uniform over all $1$-Lipschitz test
functions.

The main dynamical result of the paper is a geometric criterion for unique
ergodicity of the branched cover (Theorem \ref{thm__The-CIRCLE_CRITERION}). 

\vspace{.15cm}
\textbf{Theorem A.} 
Suppose that for only one of the slit endpoints $Q\in X$, there exists a $R>0$ and a sequence of moments in time $\{t_k\}_{k\in\bbN_0}$, $\lim_{k\to\infty}t_k=+\infty$, such that the $R$-neighborhood $\mathrm{U}_{t_k}\subset X_{t_k}$ of $Q_{t_k}= g_{t_{k}}Q$ is an embedded disk not containing $P_{t_k}=g_{t_{k}}P$. 
Then the surface $\widetilde{X}$ is uniquely ergodic.
\vspace{.15cm}

The
proof uses what we call the circle argument: one lifts the embedded disk to the
branched cover, constructs symmetric points and long vertical segments inside the
lifted neighborhood, and shows that the corresponding time averages must have
the same limit. If several distinct ergodic measures existed, the deck symmetry
would force these same points to be generic for different measures, producing a
contradiction with Lemma \ref{lemma__n-measures}.


We next show that such disks arise for almost every slit endpoint under a natural
subsequential embedded-radius hypothesis (see Theorem \ref{thm__embedded-radius-ae-points}). 

\vspace{.15cm}
\textbf{Theorem B.} 
Let $X$ be a finite-area translation surface whose vertical flow is uniquely ergodic, and fix a marked point $P\in X$. Suppose that there exist $r>0$ and a sequence $t_k\to\infty$ such that
$$
r_{\mathrm{emb}}(X_{t_k})\ge r,
\qquad X_{t_k}:=g_{t_k}X,
$$
where $r_{\mathrm{emb}}(X_{t_k})$ is the lowest upper bound on the radii of disks that can be embedded into $X_{t_k}$. Then for Lebesgue-almost every point $Q\in X$ there exist a number $R(Q)>0$ and an infinite subsequence $t_{k_j}\to\infty$ such that, for every $j$, the surface $X_{t_{k_j}}$ contains an embedded Euclidean disk of radius at least $R(Q)$ centered at $g_{t_{k_j}}(Q)$, and this disk does not contain $g_{t_{k_j}}(P)$.
\vspace{.15cm}

The remaining task is to understand when the embedded-radius hypothesis holds, which we do in Section \ref{UE-sect4.2}.
For this, we introduce \emph{pipe cylinders}, maximal flat cylinders whose height
is greater than their circumference. These cylinders are exactly the geometric
features that may force long Delaunay edges and obstruct a direct construction of
embedded disks. We analyze three complementary regimes. If no pipe cylinders are
present (Lemma \ref{lemma__inscribed-triangles}), Delaunay triangulations immediately provide embedded disks of uniform
radius. If pipe cylinders occupy uniformly positive total area along a sequence (Lemma \ref{lemma__big-area-pipes}),
one can flow one of them forward until it becomes a nearly round cylinder, which
again contains a large embedded disk. If pipe cylinders persist but their total
area tends to zero (Lemma \ref{lemma__small-area-pipes}), then the complement of the singular set and the pipe region
still contains points with uniformly controlled embedded neighborhoods. Together,
these cases yield a practical geometric route to the hypotheses of our almost-every-slit theorem. 

\vspace{.15cm}
\textbf{Theorem C.} 
Let $X$ be a finite-area translation surface, and fix a marked point $P\in X$. Then there exist $r>0$ and a sequence $t_k\to\infty$ such that
$$
r_{\mathrm{emb}}(X_{t_k})\ge r,
\qquad X_{t_k}:=g_{t_k}X.
$$
\vspace{.15cm}

This paper is organized as follows. Section~\ref{UE-sect1} introduces the slit-induced branched $\mathrm{N}$-cover construction and
the associated invariant-measure dichotomy on the $\mathrm{N}$ cover. Section~\ref{UE-sect2}
develops the theory of generic and almost generic points and records the symmetry
properties imposed by the permutation. Section~\ref{UE-sect3} contains the circle argument and
the proof of the main unique ergodicity criterion. Section~\ref{UE-sect4} proves the almost-every-slit theorem from a subsequential lower bound on embedded radius and then studies
pipe cylinders as a mechanism for producing such lower bounds.

\vspace{.25cm}
\textbf{Acknowledgements:} We are immensely thankful to Jon Chaika, who suggested this problem and spent long hours discussing it with us. We are grateful to Howard Masur for helpful discussions about the second part of this paper.

\section{The setup}\label{UE-sect1}

\begin{figure}[h]
\centering
\includegraphics[width=.5\linewidth]{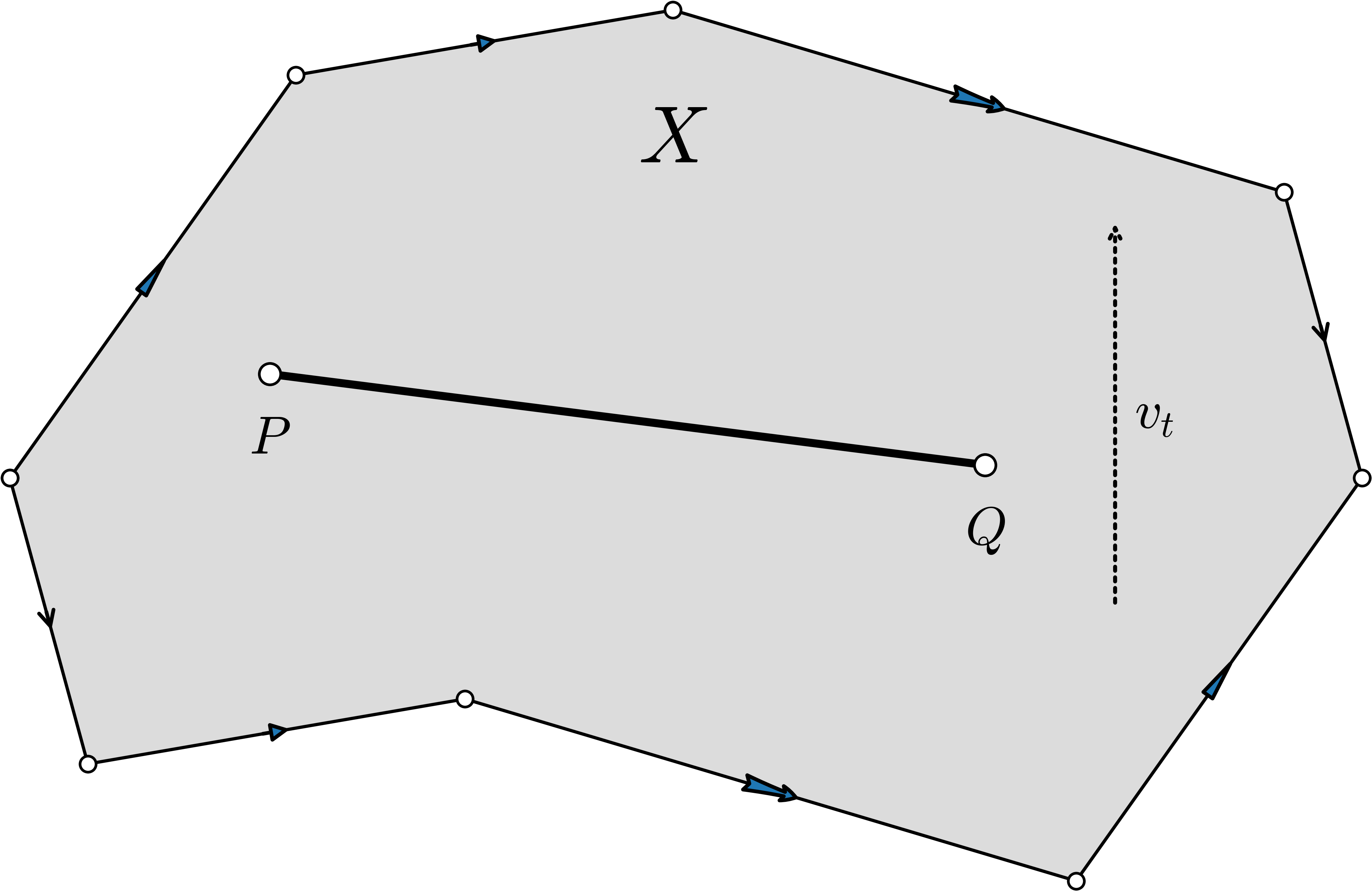}
\caption{An example of a translation surface with a slit.}
\label{fig__X}
\end{figure}

For the rest of the paper, let $X$ be a translation surface of finite area
with the set of singularities denoted by $\Sigma$ and with a uniquely
ergodic vertical flow $v_t$. Fix an integer $N\ge 2$, and let $s$ be an embedded straight-line segment
on $X$ with distinct endpoints $P,Q\in X\setminus\Sigma$, whose interior is
disjoint from $\Sigma$ (see Figure
\ref{fig__X}). The slit has to be a straight line, but can have any slope
except for the vertical one, and is allowed to be of any length.

Take $\mathrm{N}$ copies
$$
(X_0,s_0),\ (X_1,s_1),\ \ldots,\ (X_{\mathrm{N}-1},s_{\mathrm{N}-1})
$$
of the surface $(X,s)$, with indices understood modulo $\mathrm{N}$. We glue them
cyclically along the slit. More precisely, after choosing an orientation
of the slit, we glue one side of $s_i$ to the opposite side of $s_{i+1}$
for each $i\in \mathbb Z/\mathrm{N}\mathbb Z$. Reversing the orientation of the
slit replaces this cyclic gluing by its inverse, and therefore only
changes the choice of generator of the deck group.

The resulting surface will be denoted by $\tX$. A point on $\tX$ away from
the slit can be written as $(S,\,i)$, where $S$ is a position on the initial
surface $X$ and $i\in \mathbb Z/\mathrm{N}\mathbb Z$ is the index of the copy.
Under the forward vertical flow on $\tX$, the $X$-coordinate moves as usual,
and the sheet index changes cyclically whenever the trajectory crosses the
slit. With the above convention, the forward vertical flow changes
$$
i \longmapsto i+1\quad (\mathrm{mod}\;\mathrm{N})
$$
each time it crosses the slit. In other words, the surface $\tX$ is a
cyclic $\mathrm{N}$-cover of $X$, totally ramified over the two points $P$ and $Q$.
We will denote by $\tX_i$ the natural image of the $i$-th copy $X_i$ inside
$\tX$.

See Figure \ref{fig__X-tilde-construction} for a visual explanation of
the construction.

\begin{figure}[h]
\centering
\begin{subfigure}{.49\textwidth}
  \centering
  \includegraphics[width=.9\linewidth]{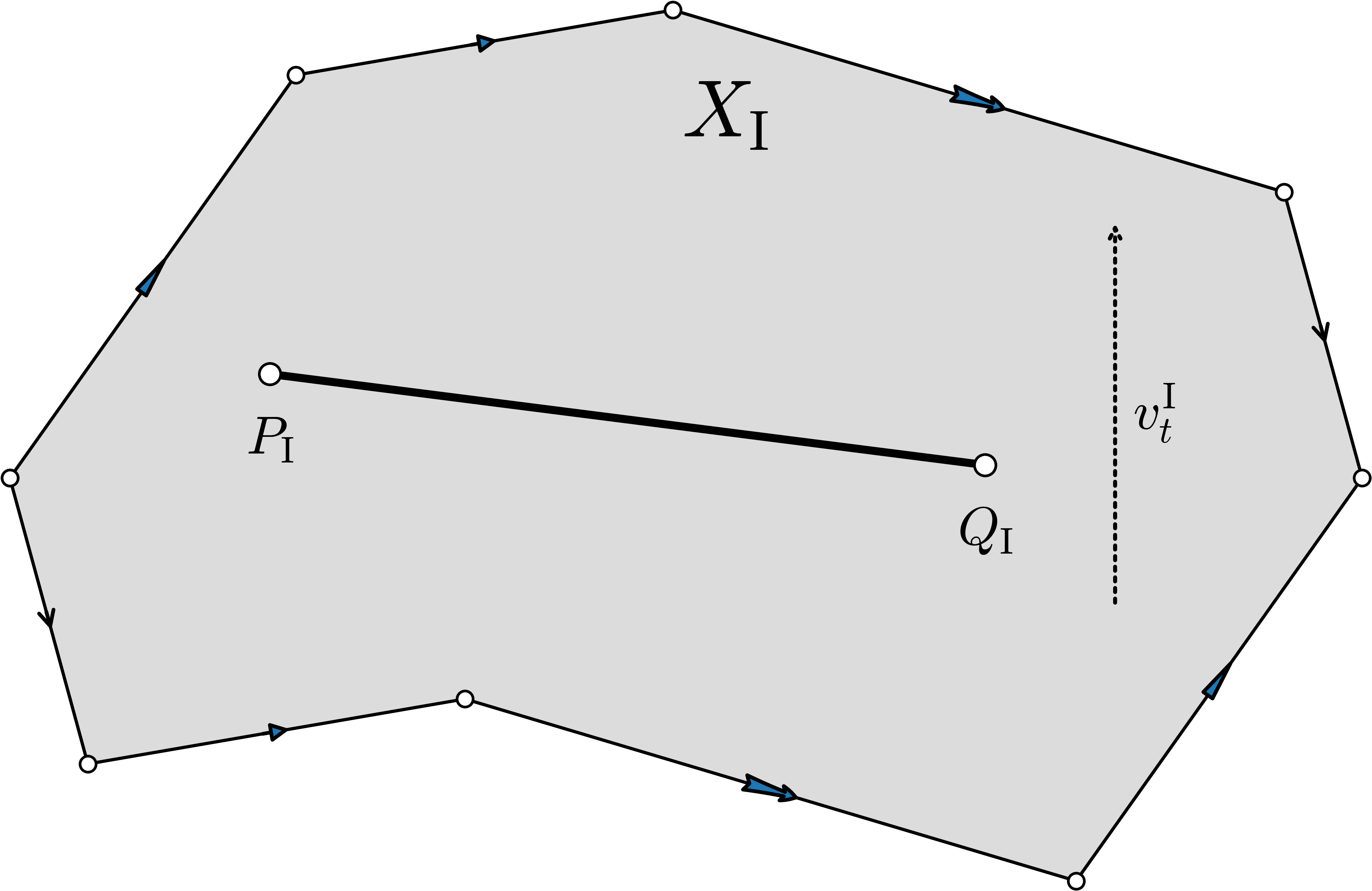}
  \label{}
\end{subfigure}%
\begin{subfigure}{.49\textwidth}
  \centering
  \includegraphics[width=.9\linewidth]{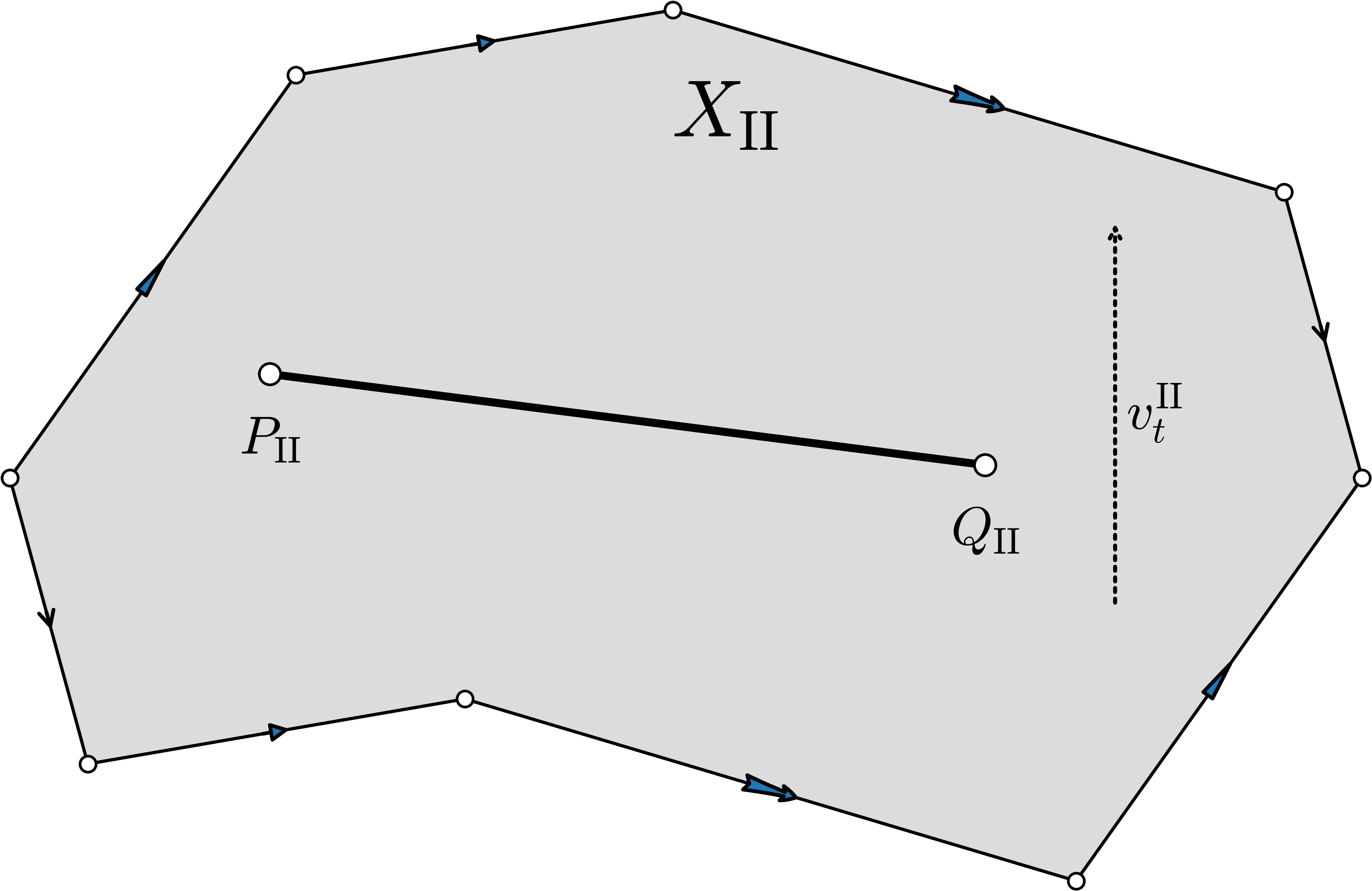}
  \label{}
\end{subfigure}
\caption{Copies used in the cyclic slit construction for $\mathrm{N}=2$.}
\label{fig__n-copies}
\end{figure}

\begin{figure}[h]
\centering
\begin{subfigure}{.49\textwidth}
  \centering
  \includegraphics[width=.9\linewidth]{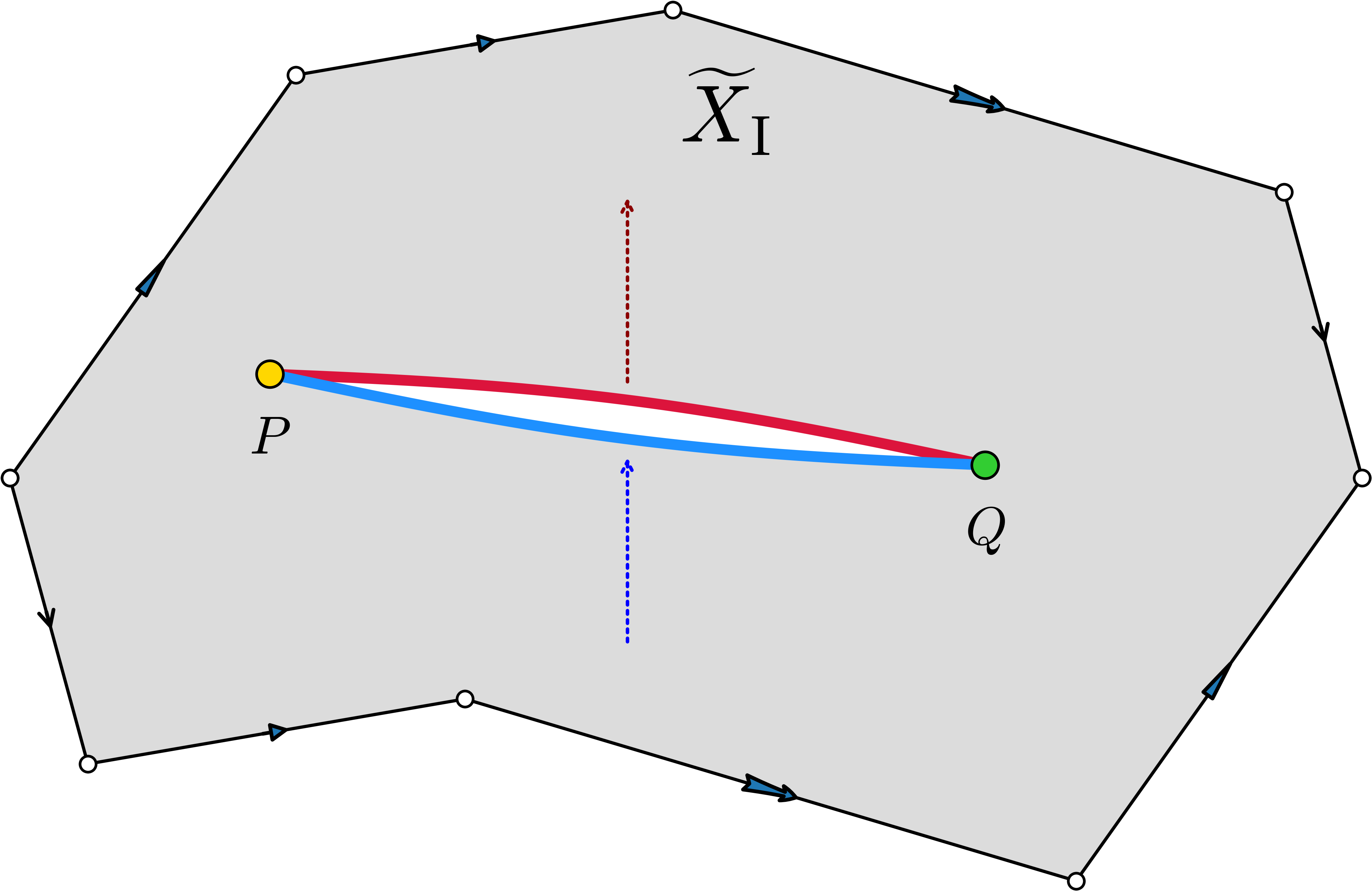}
  \label{}
\end{subfigure}%
\begin{subfigure}{.49\textwidth}
  \centering
  \includegraphics[width=.9\linewidth]{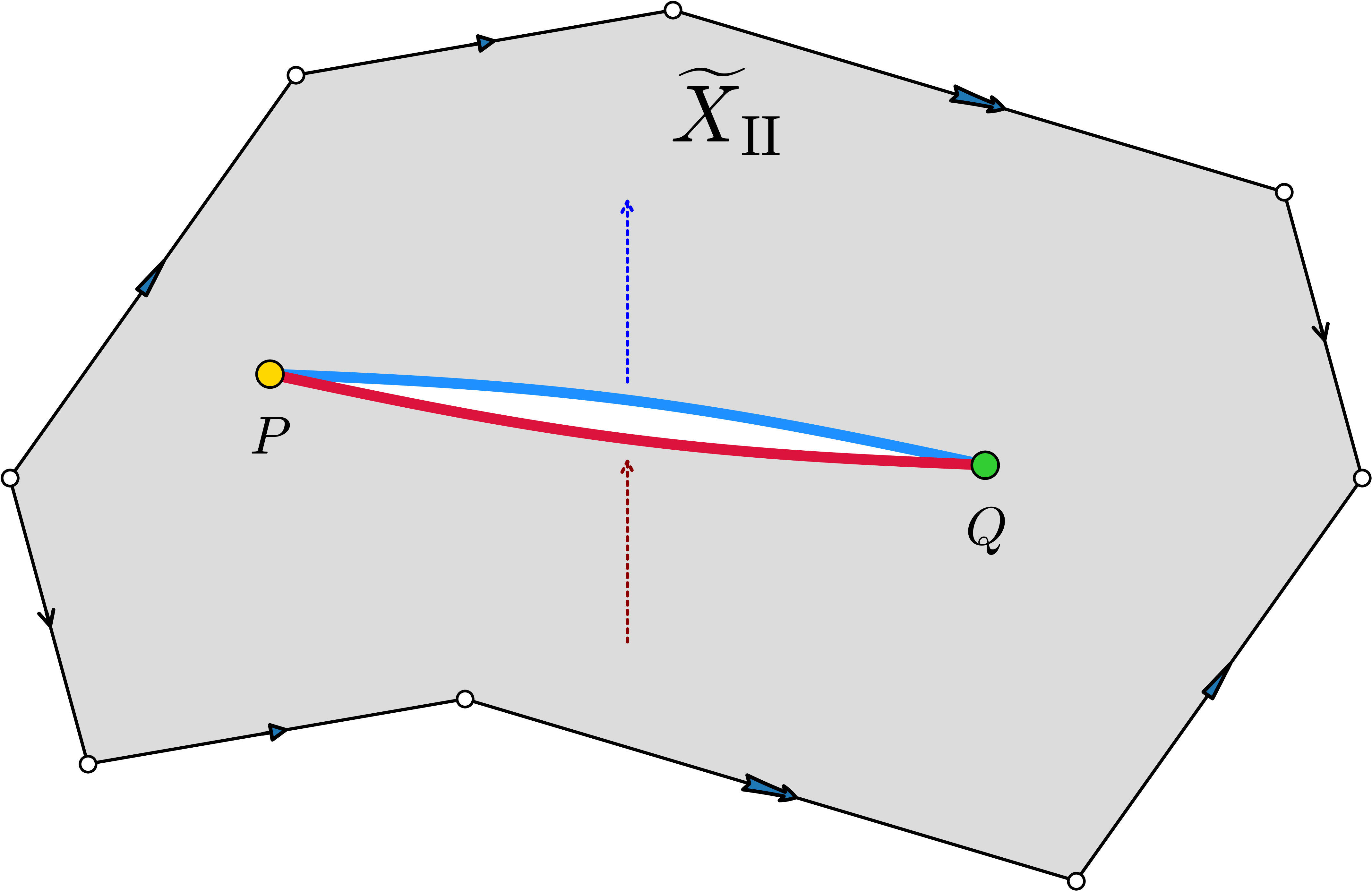}
  \label{}
\end{subfigure}
\caption{Construction of the cyclic $2$-cover surface $\tX$ and visualization of its sheets.}
\label{fig__X-tilde-construction}
\end{figure}

\begin{remark}
The condition $\operatorname{int}(s)\cap\Sigma=\varnothing$ excludes only a measure-zero family of slits. Indeed, fix a
singularity \(R\in\Sigma\). A slit whose interior contains \(R\) breaks into
two straight segments from \(R\) to the two regular endpoints. Such a
choice is determined by a cone direction at \(R\), together with the two
lengths from \(R\) to the endpoints, up to finitely many continuation
choices coming from the cone angle at \(R\). Hence, for fixed \(R\), this
exceptional family has at most three real parameters. On the other hand, the choice of two regular endpoints \(P,\,Q\in X\setminus
\Sigma\) has four real parameters. Therefore, for fixed \(R\), the family
of slits whose interior contains \(R\) is lower-dimensional and has zero
measure. Since \(\Sigma\) is finite, the union over all \(R\in\Sigma\) still
has measure zero.
\end{remark}

In the future, we will denote by
$$
\pi:\tX\to X
$$
the natural projection from $\tX$ to $X$. Thus $\pi(\tX_i)=X$ for every
$i\in \mathbb Z/\mathrm{N}\mathbb Z$. We will also denote by
$$
\varsigma:\tX\to\tX
$$
the generator of the cyclic deck group which leaves the $X$-coordinate
intact and advances the sheet index:
$$
\varsigma(S,i)=(S,i+1).
$$
Thus $\varsigma^\mathrm{N}=\operatorname{id}$. The map $\varsigma$ fixes precisely
the two ramification points lying over the endpoints of the slit, which we
denote by
$$
\widetilde P=\pi^{-1}(P),\qquad \widetilde Q=\pi^{-1}(Q).
$$
We will denote the forward vertical flow on $\tX$ by $\tv_t$ and the
backward vertical flow on $\tX$ by $\tv_t^-=\tv_{-t}$.

In local flat coordinates $z=x+iy$ coming from $\omega$, the
\emph{\Teichmuller~flow} $(g_t)$ is defined by post-composing all charts
with the linear map
$$
(x,y)\longmapsto (e^t x,\,e^{-t}y).
$$
Equivalently, $g_t$ stretches the horizontal direction by $e^t$ and
contracts the vertical direction by $e^{-t}$, preserving area. Note that
ergodic properties of translation surfaces are invariant under the action
of $g_t$. We set
$$
X_t=g_t(X),\qquad \tX_t=g_t(\tX).
$$

\begin{lemma}
For a fixed translation surface $(X,\omega)$, fixed points $P,Q\in X$, and
a fixed integer $\mathrm{N}\geq 2$, the cyclic $\mathrm{N}$-slit construction can produce
only finitely many points in the corresponding moduli space of translation
surfaces.
\end{lemma}

\begin{proof}
Let
$$
X^\circ=X\setminus\{P,Q\}.
$$
Choose a base point $x_0\notin\{P,Q\}$. By the classification theorem for
covering spaces \cite[Theorem 1.38, p. 67]{Hatcher2002AlgebraicTopology},
connected unbranched covers of $X^\circ$ are classified by conjugacy classes of subgroups of
$\pi_1(X^\circ,x_0)$ of the corresponding index.

Since we are considering cyclic degree-$\mathrm{N}$ covers, such covers are
described, after choosing a generator of the deck group, by surjective
homomorphisms
$$
\pi_1(X^\circ,x_0)\longrightarrow \mathbb Z/\mathrm{N}\mathbb Z.
$$
Equivalently, they are described by kernels of such homomorphisms; changing
the generator of $\mathbb Z/\mathrm{N}\mathbb Z$ gives the same cover. Since
$\pi_1(X^\circ,x_0)$ is finitely generated, there are only finitely many
homomorphisms to the finite group $\mathbb Z/\mathrm{N}\mathbb Z$, and therefore
only finitely many surjective ones. Hence there are only finitely many
connected cyclic unbranched degree-$\mathrm{N}$ covers of $X^\circ$.

Given a cyclic branched $\mathrm{N}$-cover
$$
\pi:\widetilde X\to X
$$
branched at $P$ and $Q$, we obtain an unbranched cyclic degree-$\mathrm{N}$ cover of
$X^\circ$ by restricting $\pi$ to the complement of the preimages of $P$
and $Q$. This restriction defines a map from connected cyclic branched
$\mathrm{N}$-covers of $X$ branched at $P$ and $Q$, up to isomorphism over $X$, to
connected cyclic unbranched degree-$\mathrm{N}$ covers of $X^\circ$, up to
isomorphism over $X^\circ$.

Two such branched covers with isomorphic restrictions to $X^\circ$ are
isomorphic over $X$ by uniqueness of extensions across the punctures.
Therefore this map is injective. Hence there are only finitely many
connected cyclic branched $\mathrm{N}$-covers of $X$ branched at $P$ and $Q$. In
particular, the total number of cyclic $\mathrm{N}$-covers obtained through the slit
construction is finite.

Finally, isomorphic cyclic $\mathrm{N}$-covers of $X$ produce the same point in the
corresponding moduli space of translation surfaces. Indeed, suppose
$$
p_A:\widetilde X_A\to X,\qquad p_B:\widetilde X_B\to X
$$
are two branched covers whose restrictions to $X^\circ$ define the same
unbranched cover. By uniqueness of extensions to branched covers, there is
a biholomorphism
$$
F:\widetilde X_A\to \widetilde X_B
$$
such that
$$
p_B\circ F=p_A.
$$
The corresponding abelian differentials upstairs are obtained by pullback,
so the resulting translation surfaces are
$$
(\widetilde X_A,p_A^*\omega)
\qquad\text{and}\qquad
(\widetilde X_B,p_B^*\omega).
$$
Since $p_B\circ F=p_A$, we have
$$
F^*p_B^*\omega=p_A^*\omega.
$$
Thus the two differentials agree under $F$, and the two translation
surfaces are equivalent in moduli space.
\end{proof}

We will now describe a natural measure on the set of all cyclic slit
$\mathrm{N}$-covers of a translation surface $(X,\omega)$. Since $(X,\omega)$ is a
point in the moduli space, it is defined up to the corresponding
equivalence. Let us choose a fixed representative, which we will also call
$(X,\omega)$, of this equivalence class.

Any cyclic $\mathrm{N}$-slit construction over $(X,\omega)$ is determined by an
unordered pair of distinct regular points
$$
\{P,Q\}\subset X^\circ
$$
together with one of the finitely many admissible cyclic $\mathrm{N}$-covers arising
from this choice of endpoints. Thus, if $\mathcal H_\mathrm{N}(\{P,Q\})$ denotes
the finite set of translation surfaces in moduli obtained as cyclic slit
$\mathrm{N}$-covers of $(X,\omega)$ with slit endpoints $P$ and $Q$, we define
$$
\mathcal S_\mathrm{N}(X,\omega)
:=
\bigl\{
(\{P,Q\},S):
\{P,Q\}\subset X^\circ,\ P\neq Q,\ S\in \mathcal H_\mathrm{N}(\{P,Q\})
\bigr\}.
$$

There is a natural projection
$$
\operatorname{pr}:\mathcal S_\mathrm{N}(X,\omega)
\to
\bigl\{\{P,Q\}:P,Q\in X^\circ,\ P\neq Q\bigr\},
\qquad
\operatorname{pr}(\{P,Q\},\,S)=\{P,Q\}.
$$
By construction, the fibers of $\operatorname{pr}$ are finite. We equip
the base with the measure induced from the product flat area measure on
$X^\circ\times X^\circ$, and each fiber with counting measure. The
resulting natural measure on $\mathcal S_\mathrm{N}(X,\omega)$ is obtained by
integrating the fiber cardinalities over the space of slit endpoints.

Let us consider the map
$$
F_{(X,\omega)}:\mathcal S_\mathrm{N}(X,\omega)\to \mathcal H,
$$
where $\mathcal H$ denotes the appropriate moduli space of translation
surfaces. In general, the map $F_{(X,\omega)}$ need not be injective. Thus
the natural measure on the family of resulting cyclic slit $\mathrm{N}$-covers is
the pushforward of the measure defined above along the map
$F_{(X,\omega)}$.

This measure does not depend on the choice of representative of the
initial translation surface. Indeed, suppose that
$$
(X_A,\omega_A)
\qquad\text{and}\qquad
(X_B,\omega_B)
$$
are two representatives of the same moduli-space point, and let
$$
f:X_B\to X_A
$$
be a translation equivalence between them. Then $f$ is a biholomorphism
with $f_*\omega_B=\omega_A$. In particular, $f$ takes distinct regular
points to distinct regular points and sends straight-line segments
connecting them to straight-line segments connecting their images. Since
$f$ is bijective, it induces a one-to-one correspondence between
$\mathcal S_\mathrm{N}(X_A,\omega_A)$ and $\mathcal S_\mathrm{N}(X_B,\omega_B)$, and since
$f_*\omega_B=\omega_A$, the resulting cyclic slit $\mathrm{N}$-covers represent the
same points in $\mathcal H$.

\vspace{.25cm}

We continue with the following simple observation.

\begin{lemma}\label{lemma__n-measures}
Let $\mathrm{N}\geq 2$. The cyclic $\mathrm{N}$-cover surface $\tX$ has exactly $d$ ergodic
invariant probability measures for some divisor $d$ of $\mathrm{N}$. More precisely,
there exists a divisor $d\mid \mathrm{N}$ and ergodic invariant probability measures
$
\mu_0,\mu_1,\ldots,\mu_{d-1}
$ 
on $\tX$ such that
\begin{enumerate}
      \item each $\mu_j$ is absolutely continuous with respect to
    $\mathcal L_{\tX}$;

    \item the deck transformation $\varsigma$ cyclically permutes them:
    $$
    \varsigma_*\mu_j=\mu_{j+1},
    \qquad j\in\mathbb Z/d\mathbb Z;
    $$

    \item their sum is the normalized Lebesgue measure with total weight
    $d$:
    $$
    \mu_0+\mu_1+\cdots+\mu_{d-1}
    =
    d\,\mathcal L_{\tX}.
    $$
\end{enumerate}
In particular, the number of ergodic invariant probability measures on
$\tX$ is at most $\mathrm{N}$.
\end{lemma}

\begin{proof}
First, we show that any normalized ergodic invariant probability measure $\mu$ on
$\tX$ projects to the normalized Lebesgue measure on $X$.

Let $B\subset X$ be invariant under the vertical flow $v_t$. Then
$\pi^{-1}(B)\subset \tX$ is invariant under the vertical flow $\tv_t$,
because
$$
\tv_{-t}\bigl(\pi^{-1}(B)\bigr)
=
\pi^{-1}(v_{-t}B)
=
\pi^{-1}(B).
$$
Since $\mu$ is ergodic on $\tX$, the measure $\mu(\pi^{-1}(B))$ is either
$0$ or $1$. By the definition of pushforward,
$$
\pi_*\mu(B)=\mu(\pi^{-1}(B)).
$$
Thus $\pi_*\mu$ is an ergodic invariant probability measure for the
vertical flow on $X$. Since the vertical flow on $X$ is uniquely ergodic,
we have
$$
\pi_*\mu=\mathcal L_X.
$$

Now we prove that $\mu$ is absolutely continuous with respect to
$\mathcal L_{\tX}$. Let $A\subset \tX$ be a Borel set with
$\mathcal L_{\tX}(A)=0$. Since $\pi$ is a finite-to-one local translation
away from the finitely many ramification points, the set $\pi(A)$ has
$\mathcal L_X$-measure zero. Therefore
$$
\mu(A)
\leq
\mu\bigl(\pi^{-1}(\pi(A))\bigr)
=
\pi_*\mu(\pi(A))
=
\mathcal L_X(\pi(A))
=
0.
$$
Hence
$$
\mu\ll \mathcal L_{\tX}.
$$

Next, we show that there can be at most $\mathrm{N}$ ergodic invariant probability
measures on $\tX$. Suppose, for contradiction, that there are $\mathrm{N}+1$
distinct ergodic invariant probability measures
$$
\mu_0,\mu_1,\ldots,\mu_\mathrm{N}.
$$
Distinct ergodic probability measures are mutually singular. Hence there
exist pairwise disjoint Borel sets
$$
E_0,E_1,\ldots,E_\mathrm{N}\subset \tX
$$
such that
$$
\mu_j(E_j)=1
\qquad\text{for each }j=0,\,1,\,\ldots,\,\mathrm{N}.
$$
Since each $\mu_j$ projects to $\mathcal L_X$, each projection $\pi(E_j)$
has full Lebesgue measure in $X$:
$$
1
=
\mu_j(E_j)
\leq
\mu_j\bigl(\pi^{-1}(\pi(E_j))\bigr)
=
\pi_*\mu_j(\pi(E_j))
=
\mathcal L_X(\pi(E_j)).
$$
Therefore
$$
\mathcal L_X\left(\bigcap_{j=0}^\mathrm{N} \pi(E_j)\right)=1.
$$
In particular, we can choose a point
$$
x\in \bigcap_{j=0}^\mathrm{N} \pi(E_j)
$$
which is not equal to $P$ or $Q$. For each $j$, choose a point
$$
y_j\in E_j\cap \pi^{-1}(x).
$$
Since the sets $E_j$ are pairwise disjoint, the points
$y_0,y_1,\ldots,y_\mathrm{N}$ are distinct. But $x\notin\{P,Q\}$, so the fiber
$\pi^{-1}(x)$ contains exactly $\mathrm{N}$ points. This is impossible. Hence there
are at most $\mathrm{N}$ ergodic invariant probability measures on $\tX$.

Now consider the action of the deck transformation $\varsigma$ on the
finite set of ergodic invariant probability measures. Since $\varsigma$
commutes with the vertical flow, the pushforward $\varsigma_*\mu$ of an
ergodic invariant probability measure is again an ergodic invariant
probability measure.

We first claim that if an invariant probability measure $\eta$ on $\tX$ is
both $\varsigma$-invariant and satisfies
$$
\pi_*\eta=\mathcal L_X,
$$
then
$$
\eta=\mathcal L_{\tX}.
$$
Indeed, let $U\subset X\setminus\{P,Q\}$ be a sufficiently small evenly
covered open set. Then
$$
\pi^{-1}(U)=U_0\sqcup U_1\sqcup\cdots\sqcup U_{\mathrm{N}-1},
$$
where $\varsigma(U_j)=U_{j+1}$. Since $\eta$ is $\varsigma$-invariant, all
the numbers $\eta(U_j)$ are equal. Since $\pi_*\eta=\mathcal L_X$, their
sum is
$$
\eta(\pi^{-1}(U))=\mathcal L_X(U).
$$
Therefore
$$
\eta(U_j)=\frac{1}{\mathrm{N}}\mathcal L_X(U)
=
\mathcal L_{\tX}(U_j).
$$
Such sets generate the Borel $\sigma$-algebra away from the ramification
points, and the ramification points have zero measure for both $\eta$ and
$\mathcal L_{\tX}$. Hence
$$
\eta=\mathcal L_{\tX}.
$$

Now let $\mu$ be one ergodic invariant probability measure on $\tX$, and
let its orbit under the deck transformation be
$$
\mu,\ \varsigma_*\mu,\ \ldots,\ \varsigma_*^{d-1}\mu,
$$
where $d$ is the smallest positive integer such that
$$
\varsigma_*^d\mu=\mu.
$$
Since $\varsigma^\mathrm{N}=\operatorname{id}$, we have $d\mid \mathrm{N}$.

Consider the average over this orbit:
$$
\eta
=
\frac{1}{d}\sum_{j=0}^{d-1}\varsigma_*^j\mu.
$$
This measure is invariant under the vertical flow, invariant under
$\varsigma$, and projects to $\mathcal L_X$. Therefore, by the claim above,
$$
\eta=\mathcal L_{\tX}.
$$

We now show that this orbit contains all ergodic invariant probability
measures on $\tX$. Suppose, for contradiction, that $\nu$ is another
ergodic invariant probability measure not belonging to the orbit of $\mu$.
Then $\nu$ is distinct from each of
$$
\mu,\ \varsigma_*\mu,\ \ldots,\ \varsigma_*^{d-1}\mu.
$$
Distinct ergodic probability measures are mutually singular, so $\nu$ is
singular with respect to each measure in this finite orbit. Hence $\nu$ is
singular with respect to their average $\eta$. But $\eta=\mathcal L_{\tX}$,
while we have already proved that every ergodic invariant probability
measure on $\tX$ is absolutely continuous with respect to
$\mathcal L_{\tX}$. This is impossible. Therefore the orbit of $\mu$
contains all ergodic invariant probability measures.

Thus the ergodic invariant probability measures on $\tX$ are precisely
$$
\mu_0,\mu_1,\ldots,\mu_{d-1},
$$
where $d\mid \mathrm{N}$, and after relabeling we have
$$
\varsigma_*\mu_j=\mu_{j+1},
\qquad j\in\mathbb Z/d\mathbb Z.
$$
Finally, since
$$
\frac{1}{d}\sum_{j=0}^{d-1}\mu_j
=
\mathcal L_{\tX},
$$
we obtain
$$
\mu_0+\mu_1+\cdots+\mu_{d-1}
=
d\,\mathcal L_{\tX}.
$$
This completes the proof.
\end{proof}

This leads to a natural question:

\textbf{When does $\tX$ have one ergodic measure and when does it have more?}

\section{Measure-generic and almost measure-generic points on $\tX$}\label{UE-sect2}

\begin{lemma}\label{lemma__reverse-flow}
    Consider the inverse vertical flow $\widetilde{v}^{-}_{t} = \widetilde{v}_{-t}$ on the translation surface $\tX$. Any measure on $\tX$ that is ergodic and invariant under the action of the flow $\widetilde{v}^{-}_{t}$ is ergodic and invariant under the action of the flow $\widetilde{v}_{t}$, and vice versa.
\end{lemma}

\begin{proof}
The argument in both directions is essentially identical, so it suffices to give a proof of only one of them. Suppose that a measure $\mu$ is ergodic and invariant on $\tX$ under the action of the flow $\widetilde{v}_{t}$. We will show that it is also invariant and ergodic under the action of the flow $\widetilde{v}^{-}_{t}$. Let $A$ be a Borel subset of $\tX$, and let us fix a time $t > 0$. To prove invariance, we note that 
\begin{equation}
    \mu(\widetilde{v}^{-}_{t}A) = \mu(\widetilde{v}_{t}(\widetilde{v}^{-}_{t}A)) = \mu(A).
\end{equation}
To prove ergodicity, consider a Borel $A \subset \tX$ that is invariant under the inverse vertical flow $\widetilde{v}^{-}_{t}$. This means that $\widetilde{v}^{-}_{t}A = A$, that is, $\widetilde{v}_{-t}(A) = A$. Applying $\widetilde{v}_{t}$ to both sides, we get $A = \widetilde{v}_{t}A$, so $A$ is also invariant under the original vertical flow. Since $\mu$ is ergodic for $\widetilde{v}_{t}$, the measure $\mu(A) \in \{0, 1\}$. Hence $\mu$ is ergodic for $\widetilde{v}^{-}_{t}$ as well.
\end{proof}

\subsection{Measure-generic points on $\tX$}\label{UE-sect2.1}

Let $\widetilde{x}$ be a non-singular point on $\tX$, let $f$ be a measurable function on $\tX$, and let $t>0$. We introduce the following notation:
\begin{equation}\label{eq__int-def}
\Iup[t]{f}{\widetilde{x}}:=\frac{1}{t}\int_{0}^{t} f(\widetilde{v}_{s}\widetilde{x})\,\rd s
\quad\quad\text{and}\quad\quad
\Idown[t]{f}{\widetilde{x}}:=\frac{1}{t}\int_{0}^{t} f(\widetilde{v}_{-s}\widetilde{x})\,\rd s.
\end{equation}

Let $m$ be an ergodic $\widetilde{v}$-invariant probability measure on
$\tX$. Suppose that the function $f$ is $m$-integrable on $\tX$. Denote
\begin{equation}\label{eq__Av-def}
    \Av_{m}(f):= \int_{\tX} f(\widetilde{x})\,\rd m(\widetilde{x}).
\end{equation}

Let $\mathcal{C}(\tX)$ be the space of all real-valued continuous functions
on $\tX$. Since $\tX$ is compact, any $f \in \mathcal{C}(\tX)$ is bounded,
hence integrable with respect to any finite regular Borel measure $m$ on
$\tX$. Thus, any such measure produces a well-defined linear functional
\begin{equation}
\begin{aligned}
    \Lambda_m:\mathcal{C}(\tX)&\longrightarrow \mathbb{R} \\
    f&\longmapsto \Av_m(f).
\end{aligned}
\end{equation}
If $\tX$ carries several distinct ergodic measures, then these averages can
be used to distinguish between them.

\begin{lemma}\label{lemma__core-1-Lipschitz-distinct-averages}
Suppose that the translation surface $\tX$ has $d>1$ distinct ergodic invariant
probability measures
$$
\mu_0,\mu_1,\ldots,\mu_{d-1}.
$$
Then there exists a $1$-Lipschitz function $\mathfrak h$ on $\tX$ such that
the numbers
$$
\Av_{\mu_0}(\mathfrak h),\,
\Av_{\mu_1}(\mathfrak h),\,
\ldots,\,
\Av_{\mu_{d-1}}(\mathfrak h)
$$
are pairwise distinct.
\end{lemma}

\begin{proof}
Since $\tX$ is a compact metric space, every Lipschitz function on $\tX$ is
bounded, and consequently integrable with respect to each of the probability
measures $\mu_0,\ldots,\mu_{d-1}$.

Let $\Lip(\tX)$ and $\Lip_1(\tX)$ be the spaces of all Lipschitz functions
and all $1$-Lipschitz functions on $\tX$, respectively. We first recall that
$\Lip(\tX)$ is dense in $\mathcal C(\tX)$ in the supremum norm.

Indeed, $\Lip(\tX)$ is closed under addition and scalar multiplication, and
since $\tX$ is compact, the product of two Lipschitz functions is again
Lipschitz. Hence $\Lip(\tX)$ is a subalgebra of $\mathcal C(\tX)$. It
contains the constant functions. It also separates points: for any fixed
point $\widetilde{x}\in\tX$, the function
$$
\widetilde{y}\longmapsto d(\widetilde{x},\widetilde{y})
$$
is Lipschitz and distinguishes $\widetilde{x}$ from any other point
$\widetilde{y}\neq \widetilde{x}$. Therefore, by the Stone--Weierstrass
Theorem, $\Lip(\tX)$ is dense in $\mathcal C(\tX)$.

Now fix two distinct indices $a,b\in\{0,\ldots,d-1\}$ with $a\neq b$. Since
$\mu_a$ and $\mu_b$ are distinct finite regular Borel measures on $\tX$, the
Riesz--Markov--Kakutani Representation Theorem implies that they cannot
define the same continuous linear functional on $\mathcal C(\tX)$. Hence
there exists some $f_{ab}\in\mathcal C(\tX)$ such that
$$
\Av_{\mu_a}(f_{ab})\neq \Av_{\mu_b}(f_{ab}).
$$
Since $\Lip(\tX)$ is dense in $\mathcal C(\tX)$, there exists a Lipschitz
function $g_{ab}\in\Lip(\tX)$ such that
$$
\Av_{\mu_a}(g_{ab})\neq \Av_{\mu_b}(g_{ab}).
$$
Thus, for every pair $a\neq b$, the linear functional
$$
T_{ab}:\Lip(\tX)\longrightarrow \mathbb R,
\qquad
T_{ab}(h)=\Av_{\mu_a}(h)-\Av_{\mu_b}(h),
$$
is not identically zero.

For each pair $a<b$, define
$$
H_{ab}:=\ker T_{ab}
=
\left\{
h\in\Lip(\tX):
\Av_{\mu_a}(h)=\Av_{\mu_b}(h)
\right\}.
$$
Since $T_{ab}$ is not identically zero, each $H_{ab}$ is a proper linear
subspace of $\Lip(\tX)$.

There are only finitely many pairs $a<b$. A finite union of proper linear
subspaces of a real vector space cannot be the whole vector space. Therefore
we can choose some
$$
h\in\Lip(\tX)
$$
such that
$$
h\notin \bigcup_{1\leq a<b\leq r} H_{ab}.
$$
Equivalently, for every pair $a<b$, we have
$$
\Av_{\mu_a}(h)\neq \Av_{\mu_b}(h).
$$
Thus the numbers
$$
\Av_{\mu_0}(h),\ldots,\Av_{\mu_{d-1}}(h)
$$
are pairwise distinct.

Let $C_h$ be a Lipschitz constant for $h$. Define
$$
\mathfrak h := \frac{h}{\max\{1,C_h\}}.
$$
Then $\mathfrak h$ is $1$-Lipschitz. Moreover, rescaling by a positive
constant does not create equalities among the averages, so
$$
\Av_{\mu_a}(\mathfrak h)\neq \Av_{\mu_b}(\mathfrak h)
$$
for every pair $a\neq b$. Hence
$$
\Av_{\mu_0}(\mathfrak h),\ldots,\Av_{\mu_{d-1}}(\mathfrak h)
$$
are pairwise distinct.
\end{proof}

We now recall one of the fundamental ergodic theorems.

\begin{theorem}[Birkhoff--Khinchin Pointwise Ergodic Theorem for Flows --- Corollary~25.9, p.~517, in \cite{Kallenberg2021}]\label{thm__Birkhoff-Khinchin}
Let $(X,\,\mathcal{B},\,m)$ be a probability space and let $\{\varphi_t\}_{t\in\mathbb{R}}$ be a jointly-measurable, $m$-preserving flow (i.e. $\varphi_{t+s}=\varphi_t\circ\varphi_s$, $\varphi_0=\mrm{id}$, and $m(\varphi_t^{-1}A)=m(A)$ for all $A\in\mathcal{B}$ and $t\in\mathbb{R}$). Then, for $m$-almost every \ $x\in X$, the integrals
$$\frac{1}{T}\int_{0}^{T}f(\varphi_t(x))\,\rd t$$
converge to the conditional expectation of $f$ with respect to the $\sigma$-algebra $\mathcal{I}$ of $\{\varphi_t\}$-invariant sets
$$
\mathbb{E}\!\left(f\,\middle|\,\mathcal{I}\right)(x),
$$
and the convergence also holds in $L^1(m)$. In particular, if the flow is ergodic, then for every $f\in L^1(m)$,
$$
\lim_{T\to\infty} \frac{1}{T}\int_{0}^{T}f(\varphi_t(x))\,\rd t 
\;=\;\Av_m(f)\quad\text{for }m\text{-almost every\ }x .
$$
In our notation (if we apply this to the vertical flow), 
$$\lim_{T\to\infty} \Iup[T]{f}{\widetilde{x}} = \Av_{m}(f).$$
\end{theorem}

As an immediate consequence of Lemma \ref{lemma__reverse-flow}, we obtain the following.
\begin{corollary}\label{cor__backward-flow}
Let $m$ be an ergodic invariant measure on $\tX$. Then the conclusion of Theorem \ref{thm__Birkhoff-Khinchin} holds for both the forward and backward vertical flows on $\tX$.
\end{corollary}

We now define the relevant notions of generic behavior of points under the vertical flow.

\begin{definition}[Generic points]

Let $m$ be an ergodic invariant measure on $\tX$, and let $\widetilde{x} \in \tX$ be a non-singular point whose backward and forward vertical trajectories are defined for all times $t\in\mathbb{R}$.

We say that
\begin{enumerate}
    \item[\textbullet] $\widetilde{x}$ is \emph{forward $m$-generic} if for any $f \in \mathcal{C}(\tX)$ we have
    \begin{equation}
        \lim_{T\to\infty} \Iup[T]{f}{\widetilde{x}} = \Av_{m}(f);
    \end{equation}
    \item[\textbullet] $\widetilde{x}$ is \emph{backward $m$-generic} if for any $f \in \mathcal{C}(\tX)$ we have
    \begin{equation}
        \lim_{T\to\infty} \Idown[T]{f}{\widetilde{x}} = \Av_{m}(f);
    \end{equation}
    \item[\textbullet] 
    $\widetilde x$ is \emph{forward generic} if it is forward $m$-generic for some ergodic invariant probability measure $m$ on $\tX$;

    \item[\textbullet]   $\widetilde x$ is \emph{backward generic} if it is backward $m$-generic for some ergodic invariant probability measure $m$ on $\tX$;

    \item[\textbullet] $\widetilde{x}$ is \emph{totally $m$-generic} if it is both forward and backward $m$-generic for some ergodic invariant probability measure $m$ on $\tX$;

    \item[\textbullet] $\widetilde{x}$ is \emph{totally generic} if $\tX$ has multiple ergodic measures and the point is both backward and forward $m$-generic for at least one of them.
    
\end{enumerate}
\end{definition}

\begin{corollary} \label{cor__no-double-genericity}
If a non-critical point $\tx$ on $\tX$ is backward/forward/totally $\mu$-generic for one ergodic measure, then it cannot be backward/forward/totally $\nu$-generic for another ergodic measure, and vice versa. 
\end{corollary} 
\begin{proof}
We first prove the forward statement. Suppose, for contradiction, that $\tx$ is both forward $\mu$-generic and forward $\nu$-generic. By Lemma \ref{lemma__core-1-Lipschitz-distinct-averages}, there exists a $1$-Lipschitz function $h\colon \tX\to\mathbb R$ such that
$$
\Av_\mu(h)\neq \Av_\nu(h).
$$
Since $\tX$ is compact, every Lipschitz function on $\tX$ is bounded, hence $h\in L^1(\mu)\cap L^1(\nu)$.

Because $\tx$ is forward $\mu$-generic, we have
$$
\lim_{T\to\infty}\Iup[T]{h}{\tx}=\Av_\mu(h).
$$
Because the same point $\tx$ is also forward $\nu$-generic, we likewise have
$$
\lim_{T\to\infty}\Iup[T]{h}{\tx}=\Av_\nu(h).
$$
The left-hand sides are the same limit, so we must have
$$
\Av_\mu(h)=\Av_\nu(h),
$$
contrary to the choice of $h$. This contradiction shows that no point can be forward $\mu$-generic and forward $\nu$-generic simultaneously.

The backward statement is proved in exactly the same way, replacing the forward averages $\Iup[T]{h}{\tx}$ by the backward averages $\Idown[T]{h}{\tx}$.

Finally, if $\tx$ were totally $\mu$-generic and totally $\nu$-generic, then in particular it would be forward $\mu$-generic and forward $\nu$-generic, which we have just shown to be impossible. Hence $\tx$ cannot be totally $\mu$-generic and totally $\nu$-generic at the same time.

The statement with $\mu$ and $\nu$ interchanged is identical.
\end{proof}

\begin{corollary} \label{cor__measure-generic-symmetry}
Suppose that the translation surface $\tX$ has $d>1$ distinct ergodic invariant
probability measures
$$
\mu_0,\mu_1,\ldots,\mu_{d-1}.
$$
If a non-critical point $\widetilde{x} \in \tX$ is forward, backward, or totally $\mu_i$-generic, then $\varsigma(\widetilde{x})$ is, respectively, forward, backward, or totally $\mu_{i+1}$-generic, and vice versa: if a non-critical point $\widetilde{x} \in \tX$ is forward, backward, or totally $\mu_i$-generic, then $\varsigma^{-1}(\widetilde{x})$ is, respectively, forward, backward, or totally $\mu_{i-1}$-generic. 
\end{corollary} 

\begin{proof}
Assume that $\widetilde{x}$ is forward $\mu_i$-generic. Fix $f \in \mathcal{C}(\tX)$. Since the permutation $\varsigma$ commutes with the vertical flow, for every $T>0$ we have
\begin{equation}
    \Iup[T]{f}{\varsigma(\widetilde{x})} = \Iup[T]{f\circ\varsigma}{\widetilde{x}}
\end{equation}
Since $f\circ\varsigma\in\mathcal{C}(\tX)$ (and $\widetilde{x}$ is forward $\mu_i$-generic),
\begin{equation}
\lim_{T\to\infty} \Iup[T]{f\circ\varsigma}{\widetilde{x}} = \Av_{\mu_i}(f\circ\varsigma),
\end{equation}
and
\begin{equation}
\lim_{T\to\infty} \Iup[T]{f}{\varsigma(\widetilde{x})} = \Av_{\mu_i}(f\circ\varsigma).
\end{equation}
Since, $\varsigma_{*}\mu_i = \mu_{i+1}$,
\begin{equation}
    \Av_{\mu_i}(f\circ\varsigma) = \Av_{\mu_{i+1}}(f).
\end{equation}
So, $\varsigma(\widetilde{x})$ is forward $\mu_{i+1}$-generic.

The proof for the backward direction is identical, and the conclusion for totally generic points follows immediately. The other direction of the proof is identical.
\end{proof}

\begin{corollary}\label{cor__measure-0-no-symmetry}
Suppose that the translation surface $\tX$ has $d>1$ distinct ergodic invariant
probability measures
$$
\mu_0,\mu_1,\ldots,\mu_{d-1}.
$$
For $i\neq j$, the set of points that are simultaneously forward
$\mu_i$-generic and backward $\mu_j$-generic has
$\mathcal L_{\tX}$-measure zero. Consequently, the set of points whose
forward and backward generic measures are different has
$\mathcal L_{\tX}$-measure zero.
\end{corollary}

\begin{proof}
Let $F_i$ be the set of forward $\mu_i$-generic points, and let $B_i$
be the set of backward $\mu_i$-generic points. By the Birkhoff--Khinchin
theorem applied to the forward and backward flows,
$$
\mu_i(F_i)=\mu_i(B_i)=1.
$$
Hence
$$
\mu_i(F_i\cap B_i)=1.
$$

Set
$$
G:=\bigcup_{i=0}^{d-1}(F_i\cap B_i).
$$
Then $\mu_i(G)=1$ for every $i$. Since
$$
\mu_0+\cdots+\mu_{d-1}=d\,\mathcal L_{\tX},
$$
we obtain
$$
\mathcal L_{\tX}(G)
=
\frac1d\sum_{i=0}^{d-1}\mu_i(G)
=
1.
$$
Thus $\mathcal L_{\tX}(\tX\setminus G)=0$.

Now suppose that a point is forward $\mu_i$-generic and backward
$\mu_j$-generic with $i\neq j$. By Corollary
\ref{cor__no-double-genericity}, such a point cannot lie in any set
$F_k\cap B_k$. Therefore it belongs to $\tX\setminus G$, which has
Lebesgue measure zero. This proves the claim.
\end{proof}

\begin{corollary}
Let $m$ be an ergodic invariant measure on $\tX$. The set of totally $m$-generic points in $\tX$ has full $m$-measure.
\end{corollary}

\begin{lemma}\label{lemma__Teichmuller-preserves-genericity}
Fix $\tau\in\bbR$. Let $\tX_\tau=g_\tau\widetilde X$ be the image of $\widetilde X$ under the
\Teichmuller~flow, and let $\widetilde v^{(\tau)}_s$ denote the vertical flow on
$\tX_\tau$. Let $m$ be an ergodic $\widetilde v_s$-invariant probability
measure on $\widetilde X$, and define
$$
m_\tau := (g_\tau)_*m.
$$
Then $m_\tau$ is an ergodic $\widetilde v^{(\tau)}_s$-invariant probability measure on
$g_\tau\widetilde X$.

Moreover, if $\widetilde x\in \widetilde X$ is non-critical, then
$$
\widetilde x
\text{ is forward/backward/totally $m$-generic on }\widetilde X
$$
if and only if
$$
g_\tau\widetilde x
\text{ is forward/backward/totally $m_\tau$-generic on }\tX_\tau.
$$
\end{lemma}

\begin{proof}
Since the \Teichmuller~map
$$
g_\tau=\begin{pmatrix}e^\tau&0\\0&e^{-\tau}\end{pmatrix}
$$
contracts the vertical direction by the factor $e^{-\tau}$, the vertical flows on
$\widetilde X$ and on $\tX_\tau$ are related by
\begin{equation}\label{eq__Teich-time-change}
\widetilde v^{(\tau)}_s\circ g_\tau
=
g_\tau\circ \widetilde v_{e^\tau s}
\qquad\text{for all }s\in\bbR.
\end{equation}
In particular, $g_\tau$ sends non-critical points to non-critical points.

We first show that $m_\tau$ is invariant under $\widetilde v^{(\tau)}_s$. Let
$A\subset \tX_\tau$ be Borel. Then, using \eqref{eq__Teich-time-change},
$$
m_\tau\bigl((\widetilde v^{(\tau)}_s)^{-1}A\bigr)
=
m\bigl(g_\tau^{-1}((\widetilde v^{(\tau)}_s)^{-1}A)\bigr)
=
m\bigl((\widetilde v_{e^\tau s})^{-1}(g_\tau^{-1}A)\bigr)
=
m(g_\tau^{-1}A)
=
m_\tau(A),
$$
because $m$ is $\widetilde v_s$-invariant.

Next we show ergodicity. Let $A\subset \tX_\tau$ be a Borel set such that
$$
(\widetilde v^{(\tau)}_s)^{-1}(A)=A
\qquad\text{for all }s\in\bbR.
$$
Then \eqref{eq__Teich-time-change} implies
$$
\widetilde v_t^{-1}(g_\tau^{-1}A)=g_\tau^{-1}A
\qquad\text{for all }t\in\bbR,
$$
since every $t\in\bbR$ can be written as $t=e^\tau s$. Thus $g_\tau^{-1}A$ is
$\widetilde v_t$-invariant on $\widetilde X$. By ergodicity of $m$,
$$
m(g_\tau^{-1}A)\in\{0,1\}.
$$
Therefore
$$
m_\tau(A)=m(g_\tau^{-1}A)\in\{0,1\},
$$
so $m_\tau$ is ergodic.

Now suppose that $\widetilde x$ is forward $m$-generic, and let $f\in L^1(m_\tau)$.
Then $f\circ g_\tau\in L^1(m)$, and by \eqref{eq__Teich-time-change},
$$
\frac1T\int_0^T f\bigl(\widetilde v^{(\tau)}_s(g_\tau\widetilde x)\bigr)\,ds
=
\frac1T\int_0^T f\bigl(g_\tau(\widetilde v_{e^\tau s}\widetilde x)\bigr)\,ds.
$$
Making the change of variables $u=e^\tau s$, so that $ds=e^{-\tau}du$, we get
$$
\frac1T\int_0^T f\bigl(\widetilde v^{(\tau)}_s(g_\tau\widetilde x)\bigr)\,ds
=
\frac{1}{e^\tau T}\int_0^{e^\tau T} (f\circ g_\tau)(\widetilde v_u\widetilde x)\,du.
$$
Since $\widetilde x$ is forward $m$-generic, the right-hand side converges, as
$T\to\infty$, to
$$
\int_{\widetilde X} f\circ g_\tau \,dm
=
\int_{g_\tau\widetilde X} f\,d(g_\tau)_*m
=
\int_{g_\tau\widetilde X} f\,dm_\tau.
$$
Hence $g_\tau\widetilde x$ is forward $m_\tau$-generic.

The backward case is identical:
$$
\frac1T\int_0^T f\bigl(\widetilde v^{(\tau)}_{-s}(g_\tau\widetilde x)\bigr)\,ds
=
\frac{1}{e^\tau T}\int_0^{e^\tau T} (f\circ g_\tau)(\widetilde v_{-u}\widetilde x)\,du,
$$
so backward $m$-genericity of $\widetilde x$ implies backward $m_\tau$-genericity
of $g_\tau\widetilde x$.

Thus total $m$-genericity of $\widetilde x$ implies total $m_\tau$-genericity of
$g_\tau\widetilde x$.

Finally, the converse implications follow by applying the same argument to the inverse
\Teichmuller~map $g_{-\tau}:\tX_\tau\to \widetilde X$.
\end{proof}

Measure-generic points are clearly invariant under the action of the vertical flow. This fact is useful in many ways.

\begin{lemma}\label{lemma__int-gen-pts-ae}
Let $I$ be a horizontal interval on $\tX$. Then the set of forward
generic points in $I$ has full Lebesgue measure $\mathcal L_I$. The same
holds for backward generic points and for totally generic points.
\end{lemma}

\begin{proof}
Let
$
\mu_0,\ldots,\mu_{d-1}
$ 
be the ergodic invariant probability measures on $\tX$. For each
$i\in\mathbb Z/d\mathbb Z$, let $F_i$ be the set of forward
$\mu_i$-generic points, and let $B_i$ be the set of backward
$\mu_i$-generic points. Define
$$
F:=\bigcup_{i=0}^{d-1}F_i,
\qquad
B:=\bigcup_{i=0}^{d-1}B_i,
\qquad
G:=\bigcup_{i=0}^{d-1}(F_i\cap B_i).
$$
Thus $F$ is the set of forward generic points, $B$ is the set of
backward generic points, and $G$ is the set of totally generic points.

By the Birkhoff--Khinchin theorem applied to the forward and backward
vertical flows,
$$
\mu_i(F_i)=1,
\qquad
\mu_i(B_i)=1
$$
for every $i$. Hence
$$
\mu_i(F_i\cap B_i)=1.
$$
Therefore
$$
\mu_i(F)=\mu_i(B)=\mu_i(G)=1
$$
for every $i$.

Using Lemma \ref{lemma__n-measures}, we have
$$
\mu_0+\cdots+\mu_{d-1}=d\,\mathcal L_{\tX}.
$$
Consequently,
$$
\mathcal L_{\tX}(F)
=
\frac1d\sum_{i=0}^{d-1}\mu_i(F)
=
1.
$$
Similarly,
$$
\mathcal L_{\tX}(B)=1,
\qquad
\mathcal L_{\tX}(G)=1.
$$
Thus the complements
$$
\tX\setminus F,
\qquad
\tX\setminus B,
\qquad
\tX\setminus G
$$
all have $\mathcal L_{\tX}$-measure zero.

We now show that this full-measure statement restricts to any horizontal
interval $I$. We prove it for $F$; the arguments for $B$ and $G$ are
identical.

Suppose, for contradiction, that
$$
\mathcal L_I(I\setminus F)>0.
$$
Since the set of singularities is finite, we may choose a compact
subinterval
$$
J\subset I\setminus \widetilde\Sigma
$$
such that
$$
\mathcal L_J(J\setminus F)>0.
$$
After possibly shortening $J$, there exists $\varepsilon>0$ such that
the map
$$
\Phi:J\times[0,\varepsilon]\to \tX,
\qquad
\Phi(x,t)=\widetilde v_t x,
$$
is a flow box: it is injective and identifies $J\times[0,\varepsilon]$
with a Euclidean rectangle in flat coordinates. In particular,
$$
\mathcal L_{\tX}\bigl(\Phi(A\times[0,\varepsilon])\bigr)
=
\varepsilon\,\mathcal L_J(A)
$$
for every measurable $A\subset J$.

The set $F$ is invariant under the vertical flow. Indeed, if $x$ is
forward $\mu_i$-generic, then $\widetilde v_t x$ is also forward
$\mu_i$-generic for every time for which the orbit is defined. Therefore
$\tX\setminus F$ is also invariant under the vertical flow.

Set
$$
A:=J\setminus F.
$$
Then $\mathcal L_J(A)>0$, and by flow-invariance,
$$
\Phi(A\times[0,\varepsilon])\subset \tX\setminus F.
$$
Hence
$$
\mathcal L_{\tX}(\tX\setminus F)
\geq
\mathcal L_{\tX}\bigl(\Phi(A\times[0,\varepsilon])\bigr)
=
\varepsilon\,\mathcal L_J(A)
>
0.
$$
This contradicts $\mathcal L_{\tX}(\tX\setminus F)=0$.

Therefore $I\cap F$ has full $\mathcal L_I$-measure in $I$. The same
argument applied to $B$ and $G$ proves the backward and total
genericity statements.
\end{proof}

\begin{lemma}\label{lemma__centr-symm-gen-pts}
Let $I$ be a horizontal interval on $\tX$. Then there exist two distinct
points in $I$, symmetric with respect to the center of $I$, that are
both generic. The same holds for totally generic points.
\end{lemma}

\begin{proof}
Let $G\subset I$ be the set of generic points in $I$. By Lemma
\ref{lemma__int-gen-pts-ae}, the set $G$ has full $\mathcal L_I$-measure
in $I$.

Let
$$
\mathfrak r:I\to I
$$
be the reflection across the center of $I$. Since $\mathfrak r$ preserves
Lebesgue measure on $I$, the set $\mathfrak r^{-1}(G)$ also has full
$\mathcal L_I$-measure in $I$. Hence
$$
G\cap \mathfrak r^{-1}(G)
$$
has full $\mathcal L_I$-measure in $I$.

Let $c$ denote the center of $I$. Since the single point $\{c\}$ has
zero $\mathcal L_I$-measure, the set
$$
G\cap \mathfrak r^{-1}(G)\setminus\{c\}
$$
is still nonempty. Choose
$$
x\in G\cap \mathfrak r^{-1}(G)\setminus\{c\}.
$$
Then $x\in G$, and since $x\in\mathfrak r^{-1}(G)$, we also have
$$
\mathfrak r(x)\in G.
$$
Therefore both $x$ and $\mathfrak r(x)$ are  generic. Moreover, since
$x\neq c$, we have
$$
x\neq \mathfrak r(x).
$$
Thus $x$ and $\mathfrak r(x)$ are two distinct points in $I$, symmetric
with respect to the center of $I$, and both are  generic.

The proof for totally generic points is identical, using the full measure
set of totally generic points in $I$ instead of $G$.
\end{proof}

\subsection{Almost measure-generic points on $\tX$}\label{UE-sect2.2}

Let $0<\epsilon<1 =\LtX(\tX) $ and $C>0$ be some real constants. Define

\begin{equation}
\begin{split}
G^\uparrow_{\epsilon,\,C} &=
\left\{\, \widetilde{x}\in\tX \ \middle|\ 
\begin{multlined}
\exists\ \text{erg. inv. meas. } m \text{ on } \tX \text{ s.t. }
\forall T > C \text{ and } \forall h\in \Lip_{1}(\tX):\\
\quad \left\| \Iup[T]{h}{\widetilde{x}} - \Av_m(h) \right\| < \epsilon
\end{multlined}
\right\},\quad
\\
G^\downarrow_{\epsilon,\,C} &=
\left\{\, \widetilde{x}\in\tX \ \middle|\ 
\begin{multlined}
\exists\ \text{erg. inv. meas. } m \text{ on } \tX \text{ s.t. }
\forall T > C \text{ and } \forall h\in \Lip_{1}(\tX):\\
\quad \left\| \Idown[T]{h}{\widetilde{x}} - \Av_m(h) \right\| < \epsilon
\end{multlined}
\right\},
\\
G^{\uparrow\downarrow}_{\epsilon,\,C} &= G^\uparrow_{\epsilon,\,C}\cap G^\downarrow_{\epsilon,\,C},
\\
G_{\epsilon,\,C} &=
\left\{\, \widetilde{x}\in\tX \ \middle|\ 
\begin{multlined}
\exists\ \text{erg. inv. meas. } m \text{ on } \tX \text{ s.t. }
\forall T > C \text{ and } \forall h\in \Lip_{1}(\tX):\\
\quad \left\| \Iup[T]{h}{\widetilde{x}} - \Av_m(h) \right\| < \epsilon
\,\text{ and }\,
\left\| \Idown[T]{h}{\widetilde{x}} - \Av_m(h) \right\| < \epsilon
\end{multlined}
\right\}.
\end{split}
\end{equation}

\begin{theorem}\label{thm__almost-generic-points}
For any constant $0<\epsilon<1=\LtX(\tX),$ there exists a positive constant $T_\epsilon$ such that
\begin{equation}
\LtX(
G_{\epsilon,\,T_\epsilon}
)>
1-\epsilon.
\end{equation}
\end{theorem}

\begin{proof}
For all $1\leq i\leq d$, Let 

\begin{equation}
G_{\epsilon,\,C}^{\mu_i} =
\left\{\, \widetilde{x}\in\tX \ \middle|\ 
\begin{multlined}
\forall T > C \text{ and } \forall h\in \Lip_{1}(\tX):\\
\quad \left\| \Iup[T]{h}{\widetilde{x}} - \Av_{\mu_i}(h) \right\| < \epsilon
\,\text{ and }\,
\left\| \Idown[T]{h}{\widetilde{x}} - \Av_{\mu_i}(h) \right\| < \epsilon
\end{multlined}
\right\}.
\end{equation}

Note that 
\begin{equation}\label{union}
G_{\epsilon,\,C} =\bigcup_{i=1}^d G_{\epsilon,\,C}^{\mu_i}.
\end{equation}

First, suppose we have estimates for each $\mu_i$ individually. That is, suppose that the following holds.

\begin{lemma}\label{thm__almost-generic-points_individual_measures}
For any constant $0<\epsilon<1$ there exists positive constants $T_\epsilon^{\mu_0},\, \ldots,\, T_\epsilon^{\mu_{d-1}}$ such that for all $0\leq i\leq d-1$,
\begin{equation}
\mu_i(
G_{\epsilon,\,T_\epsilon^{\mu_i}}^{\mu_i}
)>
1-\epsilon.
\end{equation}
\end{lemma}

Let $T_{\epsilon} = \max_{0\leq i\leq d-1} T_\epsilon^{\mu_i}$. Then

\begin{equation}
  \mu(G_{\epsilon,\,T_\epsilon}^{\mu_i}) > 1 - \epsilon  \text{ for all }  0\leq i\leq d-1.
\end{equation}

By \eqref{union}, we have

\begin{equation}
  d \mathcal{L}_{\tX} (\tX \setminus G_{\epsilon,\,T_\epsilon}) =
  \sum_{i=0}^{d-1}\mu_i(\tX \setminus G_{\epsilon,\,T_\epsilon}) 
  \leq 
  \sum_{i=0}^{d-1}\mu_i(\tX \setminus G_{\epsilon,\,T_\epsilon}^{\mu_i}) 
  < d \epsilon.  
\end{equation}

So $$\mathcal{L}_{\tX} (G_{\epsilon,\,T_\epsilon}) > 1 - \epsilon.$$

Now in order to complete the proof of Theorem \ref{thm__almost-generic-points} we need to prove Lemma \ref{thm__almost-generic-points_individual_measures}. We will prove it for the measure $\mu=\mu_i$ for an arbitrary $0 \leq i\leq d-1$. 

Once again, we can reduce the proof to two twin proofs by considering the one-sided sets

\begin{equation}
G_{\epsilon,\,C}^{\mu,\, \uparrow} = \left\{\, \widetilde{x}\in\tX \ \middle|\ 
\forall T > C \text{ and } \forall h\in \Lip_{1}(\tX) : \left\| \Iup[T]{h}{\widetilde{x}} - \Av_\mu(h) \right\| < \epsilon
\right\}
\end{equation}

and 

\begin{equation}
G_{\epsilon,\,C}^{\mu,\, \downarrow} = \left\{\, \widetilde{x}\in\tX \ \middle|\ 
\forall T > C \text{ and } \forall h\in \Lip_{1}(\tX) : \left\| \Idown[T]{h}{\widetilde{x}} - \Av_\mu(h) \right\| < \epsilon
\right\}.
\end{equation}

By definition

\begin{equation}
 G_{\epsilon,\,C}^{\mu} = G_{\epsilon,\,C}^{\mu,\, \uparrow} \cap G_{\epsilon,\,C}^{\mu,\, \downarrow}.
\end{equation}

Assume that the following holds.

\begin{lemma}\label{up_down}
For any $0 < \xi < 1$ there exist constants $C_{\uparrow}$ and $C_{\downarrow}$ such that
\begin{equation}
\mu(G_{\xi,\,C_{\uparrow}}^{\mu,\, \uparrow}) > 1-\xi \text{  and  }  \mu(G_{\xi,\,C_{\downarrow}}^{\mu,\, \downarrow}) > 1-\xi.
\end{equation}
\end{lemma}

Take $\xi = \frac{\epsilon}{2}$. Let $T_\epsilon^{\mu} = \max\{C_{\uparrow}, C_{\downarrow}\}$. Then

\begin{equation}
\mu(G_{\xi,\,T_\epsilon^{\mu}}^{\mu,\, \uparrow}) > 1 - \frac{\epsilon}{2}
\end{equation}

and

\begin{equation}
\mu(G_{\xi,\,T_\epsilon^{\mu}}^{\mu,\, \downarrow}) > 1 - \frac{\epsilon}{2}.
\end{equation}

Hence 

\begin{equation}
\mu(G_{\xi,\,T_\epsilon^{\mu}}^{\mu,\, \uparrow} \cap G_{\xi,\,T_\epsilon^{\mu}}^{\mu,\, \downarrow}) > 1 - \frac{\epsilon}{2}.
\end{equation}.

Finally, this intersection is contained in $G_{\epsilon,\,T_\epsilon^{\mu}}^{\mu}$, since $\frac{\epsilon}{2} < \epsilon$. We have now reduced the proof of Theorem \ref{thm__almost-generic-points} to proving Lemma \ref{up_down}. We will now give a proof for the upward direction, and the downward direction will follow by a similar argument once more.

\begin{proof}
Let $\widetilde{x}_{0} \in \tX$ be some non-singular point on $\tX$. Let us consider all $1$-Lipschitz functions on $\tX$ that vanish at $\widetilde{x}_{0}$:
\begin{equation*}
    \mathcal{F}_{0} = \bigg\{ f \in \Lip_{1}(\tX) \text{ }\bigg|\text{ } f(\widetilde{x}_{0}) = 0 \bigg\}.
\end{equation*}

We will discuss a couple of properties of the set $\mathcal{F}_{0}$ that will be useful in the proof.

\begin{proposition}\label{F_zero_uniformly_bounded}
There exists a constant $D$ such that every $f\in\mcl{F}_{0}$ is uniformly bounded by $D$ (note that the constant here is the same for all the functions in $\mathcal{F}_{0}$).  
\end{proposition}

\begin{proof}
For all $f \in \mathcal{F}_{0}$ and for all $\widetilde{x} \in \tX$ we have 
\begin{equation}
    |f(\widetilde{x})| = |f(\widetilde{x}) - 0| = |f(\widetilde{x}) - f(\widetilde{x}_{0})|\leq d(\widetilde{x}, \widetilde{x}_{0}) \leq \textrm{diam}(\tX)=:D.
\end{equation}
\end{proof}

\begin{proposition}\label{compact F 0}
The set $\mcl{F}_{0} \subset \mathcal{C}(\tX)$ is compact in the supremum norm $\|\cdot\|_{\infty}$.
\end{proposition}

\begin{proof}
First, we will show that $\mcl{F}_{0}$ is closed in the supremum norm topology. Suppose that a sequence of functions $\{f_{n}\}$ from $\mathcal{F}_{0}$ converges uniformly on $\tX$ to some other function $f$. We will prove that $f \in \mathcal{F}_{0}$.

Since convergence is uniform, pointwise convergence is present at each point. So 

\begin{equation}
    f(\widetilde{x}_{0}) = \lim_{n \rightarrow \infty} f_{n}(\widetilde{x}_{0}) = \lim_{n \rightarrow \infty} 0 = 0,
\end{equation}
\\
since every $f_{n}(\widetilde{x}_{0}) = 0$.
\\
We now check that $f$ is $1$-Lipschitz. Fix any pair of points $\widetilde{x}, \widetilde{y} \in \tX$. We have for any fixed $f_{n}$ 

\begin{equation}
|f_{n}(\widetilde{x}) - f_{n}(\widetilde{y})| \leq d_{\tX}(\widetilde{x}, \widetilde{y}).
\end{equation}
\\
Since $f_{n}(\widetilde{x}) \rightarrow f(\widetilde{x})$ and $f_{n}(\widetilde{y}) \rightarrow f(\widetilde{y})$, we have

\begin{equation}
|f(\widetilde{x}) - f(\widetilde{y})| = \lim_{n \rightarrow\infty}|f_{n}(\widetilde{x}) - f_{n}(\widetilde{y})| \leq d_{\tX}(\widetilde{x}, \widetilde{y}).
\end{equation}
\\

Now we will show that the closure of $\mcl{F}_{0}$ is compact by applying the famous Arzelà–Ascoli Theorem.

\begin{remark}[Arzelà–Ascoli Theorem]
Recall that if
        \begin{enumerate}
            \item[\textbullet] $K$ is a compact metric space;
            \item[\textbullet] $\mathcal{F} \subset \mathcal{C}(K)$ is a subset of continuous functions from $K$ to $\mathbb{R}$.
        \end{enumerate}
        Then $\mathcal{F}$ is has compact closure in $\mathcal{C}(K)$ with the uniform topology if the two conditions below are met:
        \begin{enumerate}
            \item $\mathcal{F}$ is equicontinuous (that is,  for all $x, y\in K$ and for any $\varepsilon >0$ there exists a $\delta > 0$ such that for any $f \in \mathcal{F}$
        $$\text{ if }  d_{K}(x,y) < \delta \Rightarrow |f(x) - f(y)| < \varepsilon);$$
            \item For every $x \in K$ the set $\big\{ f(x) \text{ }\big|\text{ } f \in \mathcal{F}\big\}$ has compact closure in $\mathbb{R}$.
        \end{enumerate}
\end{remark}

We do a straightforward check of all the conditions in our case. First, note that $\tX$ with its standard flat metric is compact, since it is a finite-area translation surface. Equicontinuity is a consequence of all the functions being $1$-Lipschitz for the same constant $1$. We can simply take $\delta = \varepsilon$, and then if $d_{\tX}(\widetilde{x}, \widetilde{y}) < \delta$ we have by the definition of a $1$-Lipschitz function:
$$|f(\widetilde{x}) - f(\widetilde{y})| \leq d_{\tX}(\widetilde{x}, \widetilde{y}) < \delta = \varepsilon.$$
Note that the same $\delta$ works for all functions in $\mathcal{F}_{0}$.

By Proposition \ref{F_zero_uniformly_bounded}, for every $\widetilde{x} \in \tX$ we have the set 
$$\big\{f(\widetilde{x}) \text{ }\big|\text{ } f \in \mathcal{F}_{0} \big\}$$
bounded from above by the diameter $D$ of $\tX$. It is then a bounded set in $\mathbb{R}$, and bounded sets in $\mathbb{R}$ have compact closures. So $\mathcal{F}_{0}$ is closed and has compact closure, which means it is compact (in the sup-norm $\|\cdot\|_{\infty}$).
\end{proof}

After establishing compactness of $\mathcal{F}_{0}$, we can build a finite subcover. Fix any positive number $\delta > 0$. For every $f \in \mathcal{F}_{0}$ consider its $\delta$-neighborhood $U_{\delta}(f)$ in the uniform metric. The collection of all such neighborhoods is an open cover of the set $\mathcal{F}_{0}$. Since $\mathcal{F}_{0}$ is compact, it has a finite subcover, so there is a finite subset of functions $\{f_{1}, f_{2}, \ldots, f_{M_{\delta}}\} \subset \mathcal{F}_{0}$, that are centers of this finite subset of neighborhoods $U_{\delta}(f_{1}), U_{\delta}(f_{2}), \ldots, U_{\delta}(f_{M_{\delta}})$ that cover the entire $\mathcal{F}_{0}$. In other words, for any $f \in \mathcal{F}_{0}$ there is at least one $f_{j}\in\{f_{1}, f_{2}, \ldots, f_{M_{\delta}}\}$ such that

\begin{equation}
    ||f-f_{j}||_{\infty} \leq \delta.
\end{equation}

By the Birkhoff--Khinchin Pointwise Ergodic Theorem for flows, for $\mu$-almost every point in $\tX$ the discrete sequence of the time-average integrals at integer points in time converges to the space average for any fixed $f_{j}$ :

\begin{equation}
 \lim_{N \rightarrow \infty} \Iup[N]{f_{j}}{\widetilde{x}} = \Av_\mu(f_{j}).
\end{equation}

 Fix one of the neighborhood-center $f_{j}$-s. Then fix $N$ -- the integer moment in time. Only the point $\widetilde{x}$ is now varying in the expsession $\Iup[N]{f_{j}}{\widetilde{x}}$, so view that average of $f_{j}$ at time $N$ as a function of $\widetilde{x}$.

 \begin{proposition}\label{up_int_mu_measurable}
    This function $\Iup[N]{f_{j}}{\cdot}$ is $\mu$-measurable on $\tX$.
 \end{proposition}

 \begin{proof}
 Since the function $(s, \widetilde{x}) \mapsto f_{j}(\widetilde{v}_{s}\widetilde{x})$ is continuous on $[0, N] \times \tX$, its integral over the compact interval $[0, N]$ is continuous, hence is $\mu$-measurable. 
 \end{proof}

The space-average $\Av_\mu(f_{j})$ does not depend on any point, but can be viewed as a constant function on $\tX$ that just takes the value $\Av_\mu(f_{j})$ at all points $\widetilde{x}\in\tX$ ($j$ is fixed). Being a constant function, it is obviously $\mu$-measurable on $\tX$. 

We will now prove this.

\begin{proposition}\label{main_egorov_step}
For any $0 < \zeta < 1$ there exists a $\mu$-measurable set $E(\zeta) \subset \tX$ such that $\mu(E(\zeta)) \geq 1 - \zeta$ and such that for any $j \in \{1, \ldots, M_{\delta}\}$

\begin{equation}
    \lim_{N \rightarrow \infty} \Iup[N]{f_{j}}{\widetilde{x}} = \Av_\mu(f_{j})
\end{equation}

uniformly on $E(\zeta)$.
\end{proposition}

\begin{proof}
This proof amounts to an application of another classical theorem.

\begin{remark}[Egorov-Severini Theorem]
Suppose that $(Y, m)$ is a finite measure space. Let $\psi_{n}:Y \rightarrow \mathbb{R}$ be measurable functions with $\lim_{n \rightarrow\infty}\psi_{n}(y) = \psi(y)$ for $m$-almost every $y \in Y$. Then for every $\eta > 0$ there exists a measurable set $E \subset Y$ such that $m(Y\setminus E) < \eta$ and $\psi_{n} \rightarrow \psi$ uniformly on $E$.
\end{remark}

Fix one $f_{j}$ from the finite set $\{f_{1}, f_{2}, \ldots, f_{M_{\delta}}\}$ obtained above. Take $Y = \tX$, $m = \mu$ and $\psi_{N}(\widetilde{x}) = \Iup[N]{f_{j}}{\widetilde{x}}$ (for a fixed $f_{j}$). Note that $\Iup[N]{f_{j}}{\widetilde{x}}$ is $\mu$-measurable by Proposition \ref{up_int_mu_measurable}, and 

\begin{equation}
     \lim_{N \rightarrow \infty} \Iup[N]{f_{j}}{\widetilde{x}} = \Av_\mu(f_{j})
\end{equation}
\\
for $\mu$-almost every $\widetilde{x} \in \tX$ by the Birkhoff--Khinchin Pointwise Ergodic Theorem for flows. 

Egorov-Severini Theorem tells us that (in particular) for $\eta = \frac{\zeta}{M_{\delta}}$ there exists a measurable set $E_{j}(\frac{\zeta}{M_{\delta}})\subset \tX$ with $\mu(E_{j}(\frac{\zeta}{M_{\delta}})) \geq 1 - \frac{\zeta}{M_{\delta}}$ and with 
$$\Iup[N]{f_{j}}{\widetilde{x}} \rightarrow  \Av_\mu(f_{j})$$
uniformly on $E_{j}(\frac{\zeta}{M_{\delta}})$.

Since this can be applied to each $f_{j}$ individually, we can consider the intersection

\begin{equation}
    E(\zeta) = \cap_{j =1}^{M_{\delta}}E_{j}\bigg(\frac{\zeta}{M_{\delta}}\bigg).
\end{equation}

Then 

\begin{equation}
    \mu(\tX \setminus E(\zeta)) \leq \sum_{j = 1}^{M_{\delta}}\mu\bigg(\tX\setminus E_{j}\bigg(\frac{\zeta}{M_{\delta}}\bigg)\bigg) < M_{\delta}\cdot\frac{\zeta}{M_{\delta}}  = \zeta,
\end{equation}

and 

\begin{equation}
    \mu(E(\zeta)) \geq 1 - \zeta.
\end{equation}
        
Since for each $j$ we have $E(\zeta) \subset E_{j}(\frac{\zeta}{M_{\delta}})$,
we have $$\Iup[N]{f_{j}}{\widetilde{x}} \rightarrow  \Av_\mu(f_{j})$$ uniformly on $E(\zeta)$ for every $j \in \{1,....,M_{\delta}\}$. So we have proved Proposition \ref{main_egorov_step}.
\end{proof}

We will now combine these $j$ separate uniform convergence statements into one using that there are only finitely many $j$-s. Expanding the main statement of Proposition \ref{main_egorov_step}, we can write that for every $\rho > 0$ and for each $f_{j}$ we have some natural $N_{j}(\delta, \zeta, \rho)$ such that for any $n > N_{j}(\delta, \zeta, \rho)$ we have

\begin{equation}
\sup_{\widetilde{x} \in E(\zeta)} |\Iup[n]{f_{j}}{\widetilde{x}} - \Av_\mu(f_{j})| < \rho.
\end{equation}

Let $N(\delta, \zeta, \rho) = \max_{j \in \{1, \ldots, M_{\delta}\}}N_{j}(\delta, \zeta, \rho)$. This will work for all $f_j$-s simultaneously. To sum this up, we have shown the following.

\begin{lemma}\label{lemma_at_middle_stage}
For any $0 < \zeta < 1$ there exists a $\mu$-measurable set $E(\zeta) \subset \tX$ with $\mu(E(\zeta)) \geq 1 - \zeta$ such that for any $\rho > 0$ there exists a natural number $N(\delta, \zeta, \rho)$ such that for any $n > N(\delta, \zeta, \rho)$, for any $\widetilde{x} \in E(\zeta)$ and for any $j \in \{1, \ldots, M_{\delta}\}$ 

\begin{equation}
    |\Iup[n]{f_{j}}{\widetilde{x}} - \Av_\mu(f_{j})| < \rho.
\end{equation}
\end{lemma}

We will now show that the statement of Lemma \ref{lemma_at_middle_stage} holds for all functions in $\mathcal{F}_{0}$. To be more precise, we will prove the following.

\begin{lemma}\label{lemma_upgrade_from_f_j_s_to_F_0}
For any $0 < \zeta < 1$ and for any $\rho > 0$ there exists a $\mu$-measurable set $E(\zeta) \subset \tX$ with $\mu(E(\zeta)) \geq 1 - \zeta$ such that there exists a natural number $N(\zeta, \rho)$ such that for any $n > N(\zeta, \rho)$, for any $\widetilde{x} \in E(\zeta)$ and for any $f \in \mathcal{F}_{0}$ 

\begin{equation}
    |\Iup[n]{f}{\widetilde{x}} - \Av_\mu(f)| < \rho.
\end{equation}
\end{lemma}

\begin{proof}
Let $f \in \mathcal{F}_{0}$. Take $\delta = \frac{\rho}{3}$ and build the corresponding finite cover and the functions $f_{1}, \ldots, f_{M_{\delta}}$. There is some $f_{j} \in \{f_{1}, \ldots, f_{M_{\delta}}\}$ such that

\begin{equation}
    ||f - f_j||_{\infty} < \delta.
\end{equation}

Let us compare the time and space averages of $f$ and $f_{j}$.  Using elementary properties of norms we get:
    \begin{multline*}
        |\Iup[n]{f}{\widetilde{x}} - \Av_\mu(f)| < |\Iup[n]{f}{\widetilde{x}} - \Iup[n]{f_{j}}{\widetilde{x}}|
        + |\Iup[n]{f_{j}}{\widetilde{x}} - \Av_\mu(f_{j})|  + |\Av_\mu(f_{j}) - \Av_\mu(f)|.
    \end{multline*}

We will now estimate each term in the expression above individually.  The first term
         $$|\Iup[n]{f}{\widetilde{x}} - \Iup[n]{f_{j}}{\widetilde{x}}| = |\frac{1}{n}\int_{0}^{n} (f-f_j)(\widetilde{v}_{t} \widetilde{x}) dt | \leq |\frac{1}{n}\int_{0}^{n} ||f-f_j||_{\infty} dt | \leq \frac{1}{n}\int_{0}^{n} \delta dt = \delta$$
For the second term, since $\widetilde{x} \in E(\epsilon)$, and $n > N(\delta, \zeta, \rho)$ we have

$$|\Iup[n]{f_{j}}{\widetilde{x}} - \Av_\mu(f_{j})| < \delta.$$
\\
Finally, the third term

 $$|\Av_\mu(f_{j}) - \Av_\mu(f)| = |\int_{\tX}(f_j - f) d\mu| \leq \int_{\tX}d\mu \cdot ||f-f_j||_{\infty} = 1 \cdot ||f-f_j||_{\infty} < \delta.$$

Combining these we get 
\begin{equation}
    |\Iup[n]{f}{\widetilde{x}} - \Av_\mu(f)| < 3 \delta = \rho.
\end{equation}

\end{proof}

Now that we have proved Lemma \ref{lemma_upgrade_from_f_j_s_to_F_0}, we will upgrade it again and show that a similar statement holds for all moments in time, not only for integer ones.

\begin{lemma}\label{lemma_upgrade_from_N_to_any_T}
For any $0 < \zeta < 1$ and for any $\eta > 0$ there exists a $\mu$-measurable set $E(\zeta) \subset \tX$ with $\mu(E(\zeta)) \geq 1 - \zeta$ such that there exists a real number $T(\zeta, \eta) > 0$ such that for any $t > T(\zeta, \eta)$, for any $\widetilde{x} \in E(\zeta)$ and for any $f \in \mathcal{F}_{0}$ 

\begin{equation}
    |\Iup[t]{f}{\widetilde{x}} - \Av_\mu(f)| < \eta.
\end{equation}
\end{lemma}

\begin{proof}

Fix $1> \zeta > 0$ and $\eta > 0$. Take $\rho = \frac{\eta}{2}$. By Lemma \ref{lemma_upgrade_from_f_j_s_to_F_0} there is a set $E(\zeta)$ with $\mu(E(\zeta))\geq 1- \zeta$ 
and an $N(\zeta, \rho) > 0$ such that for any $\widetilde{x} \in E(\zeta)$, for every integer $n > N(\zeta, \rho)$ and for any $f \in \mathcal{F}_{0}$ we have

\begin{equation}
    |\Iup[n]{f}{\widetilde{x}} - \Av_\mu(f)| < \rho.
\end{equation}

Now fix $t>0$, $\widetilde{x} \in \tX$ and $f \in \mathcal{F}_{0}$. Let $n := \lfloor t \rfloor$. We want to estimate $|\Iup[t]{f}{\widetilde{x}} - \Av_\mu(f)|.$  We can write 

\begin{equation}
    |\Iup[t]{f}{\widetilde{x}} - \Av_\mu(f)| \leq |\Iup[t]{f}{\widetilde{x}} - \Iup[n]{f}{\widetilde{x}}| + |\Iup[n]{f}{\widetilde{x}} - \Av_\mu(f)|.
\end{equation}

If $n > N(\zeta, \rho)$, then the second term is smaller than $\rho = \frac{\eta}{2}$.

To estimate the first term, we write:

\begin{multline}
    |\Iup[t]{f}{\widetilde{x}} - \Iup[n]{f}{\widetilde{x}}| = \bigg|\frac{1}{t} \int_{0}^{t} f(\widetilde{v}_s(\widetilde{x}))ds - \frac{1}{n} \int_{0}^{n} f(\widetilde{v}_s(\widetilde{x}))ds\bigg|  = \bigg|\frac{1}{t} \int_{0}^{n} f(\widetilde{v}_s(\widetilde{x}))ds + \\ + \frac{1}{t} \int_{n}^{t} f(\widetilde{v}_s(\widetilde{x}))ds - \frac{1}{n} \int_{0}^{n} f(\widetilde{v}_s(\widetilde{x}))ds\bigg| = \bigg|\bigg(\frac{1}{t} - \frac{1}{n}\bigg) \int_{0}^{n} f(\widetilde{v}_s(\widetilde{x}))ds + \frac{1}{t} \int_{n}^{t} f(\widetilde{v}_s(\widetilde{x}))ds \bigg| \leq\\
    \leq \bigg|\bigg(\frac{1}{t} - \frac{1}{n}\bigg)\bigg|nD + \frac{1}{t} (t-n)D =\\= \bigg|\bigg(\frac{n}{t} - 1\bigg)\bigg|D + \frac{1}{t} (t-n)D = 2 \frac{t-n}{t}D
\end{multline}

by Proposition \ref{F_zero_uniformly_bounded}. Since $t\in[n,n+1]$, we have $0 \leq t -n \leq 1 $. So 

$$2 \frac{t-n}{t}D \leq \frac{2D}{t} \leq \frac{2D}{n}$$

So combining all of that we get

$$|\Iup[t]{f}{\widetilde{x}} - \Av_\mu(f)| \leq \frac{2D}{n} + \frac{\eta}{2}$$

Now choose $T(\zeta, \eta)$ large enough so that whenever $t > T(\zeta, \eta)$ and $n = \lfloor t \rfloor$ we have both $n > N(\zeta, \rho)$ and $\frac{2D}{n}< \frac{\eta}{2}$.

 Note that this threshold only depends on $\zeta$, $\eta$ and the constant $D$. This immediately leads to 

$$|\Iup[t]{f}{\widetilde{x}} - \Av_\mu(f)| \leq \frac{\eta}{2} + \frac{\eta}{2} = \eta.$$

This proves Lemma \ref{lemma_upgrade_from_N_to_any_T}.
\end{proof}

Finally, to prove Lemma \ref{up_down}, we note that for any constant $c$ 

\begin{equation}
    \Iup[t]{f+c}{\widetilde{x}} = \Iup[t]{f}{\widetilde{x}} + c\text{ and  }
      \Idown[t]{f+c}{\widetilde{x}} = \Idown[t]{f}{\widetilde{x}} + c,
\end{equation}

and also
\begin{equation}
    \Av_\mu(f + c)  = \Av_\mu(f) + c.
\end{equation}

This means that $$\Iup[t]{f+c}{\widetilde{x}} - \Av_\mu(f + c) = \Iup[t]{f}{\widetilde{x}} - \Av_\mu(f),$$ and  $$\Idown[t]{f+c}{\widetilde{x}} - \Av_\mu(f + c) = \Idown[t]{f}{\widetilde{x}} - \Av_\mu(f).$$ Thus the statement of Lemma \ref{up_down} holds for all $f \in \Lip_{1}(\tX)$. This completes the proof of Theorem \ref{thm__almost-generic-points}.

\end{proof}
\end{proof}


\begin{definition}[Almost generic points]
We call a point $\tx\in \tX$ \emph{$\epsilon$-almost totally generic} if $\tx\in G_{\epsilon,\,T_\epsilon}$. 
\end{definition}

\begin{lemma}
\label{lemma__measure-eps-generic-symmetry}
Let $0<\epsilon<1$ and $C>0$. If a non-critical point $\tx\in\tX$ is
forward/backward/totally $(\epsilon,C)$-almost $\mu_i$-generic, then $\varsigma(\tx)$ is
forward/backward/totally $(\epsilon,C)$-almost $\mu_{i+1}$-generic, respectively. 
\end{lemma}

\begin{proof}
We prove the forward statement. The backward statement is identical, and the totally
$(\epsilon,C)$-generic statement follows by combining the forward and backward cases.

Assume that $\tx$ is forward $(\epsilon,C)$-almost $\mu_i$-generic. We must show that
$\varsigma(\tx)$ is forward $(\epsilon,C)$-almost $\mu_{i+1}$-generic.

Let $h\in \Lip_1(\tX)$. Since $\varsigma$ is an isometry of $\tX$, the function
$h\circ\varsigma$ also belongs to $\Lip_1(\tX)$.

By Lemma~\ref{lemma__n-measures}, the permutation $\varsigma$ pushes $\mu_i$ forward to
$\mu_{i+1}$, i.e.
$$
\varsigma_*\mu_i=\mu_{i+1}.
$$
Therefore
$$
\Av_{\mu_i}(h\circ\varsigma)
=
\int_{\tX} h(\varsigma(y))\,d\mu_i(y)
=
\int_{\tX} h(z)\,d\mu_{i+1}(z)
=
\Av_{\mu_{i+1}}(h).
$$

Next, by symmetry of the construction, $\varsigma$ commutes with the vertical flow:
$$
\tv_t\circ\varsigma=\varsigma\circ\tv_t
\qquad\text{for all }t\in\bbR.
$$
Hence for every $T>0$,
$$
\Iup[T]{h}{\varsigma(\tx)}
=
\frac1T\int_0^T h(\tv_t(\varsigma(\tx)))\,dt
=
\frac1T\int_0^T h(\varsigma(\tv_t(\tx)))\,dt
=
\frac1T\int_0^T (h\circ\varsigma)(\tv_t(\tx))\,dt
=
\Iup[T]{h\circ\varsigma}{\tx}.
$$

Since $\tx$ is forward $(\epsilon,C)$-almost $\mu_i$-generic, it follows that for every
$T>C$,
$$
\left|
\Iup[T]{h\circ\varsigma}{\tx}
-
\Av_{\mu_i}(h\circ\varsigma)
\right|
<\epsilon.
$$
Using the two identities above, we obtain
$$
\left|
\Iup[T]{h}{\varsigma(\tx)}
-
\Av_{\mu_{i+1}}(h)
\right|
<\epsilon
\qquad\text{for all }T>C.
$$
Since this holds for every $h\in\Lip_1(\tX)$, the point $\varsigma(\tx)$ is forward
$(\epsilon,C)$-almost $\mu_{i+1}$-generic.

The backward case is proved in exactly the same way, using
$$
\Idown[T]{h}{\varsigma(\tx)}
=
\Idown[T]{h\circ\varsigma}{\tx},
$$
which again follows from $\tv_{-t}\circ\varsigma=\varsigma\circ\tv_{-t}$.

If $\tx$ is totally $(\epsilon,C)$-almost $\mu_i$-generic, then it is both forward and
backward $(\epsilon,C)$-almost $\mu_i$-generic, so the argument above shows that
$\varsigma(\tx)$ is both forward and backward $(\epsilon,C)$-almost $\mu_{i+1}$-generic. Hence
$\varsigma(\tx)$ is totally $(\epsilon,C)$-almost $\mu_{i+1}$-generic.
\end{proof}

\begin{lemma}\label{lemma__int-eps-gen-pts}
Let $I$ be a horizontal interval on $\tX$, and let $0<\epsilon<1$. Let $F_I:I\to I$ be the first return map of the vertical flow to $I$.

Then there exists a constant 
$C=C(I,\epsilon)>0$ such that
$$
\mcl{L}_I\bigl(G^{I}_{\epsilon,\,C}\bigr)
>(1-\epsilon)\mcl{L}_I(I).
$$
\end{lemma}

\begin{proof}
Consider the first return map $F_I:I\to I$ of the vertical flow to $I$.
This is an interval exchange transformation preserving the Lebesgue measure
$\mcl{L}_I$ on $I$.

Apply the quantitative almost-genericity Theorem \ref{thm__almost-generic-points} to the measure-preserving system
$(I,F_I,\mcl{L}_I)$. It yields a constant $C=C(I,\epsilon)$ such that the set
of $\epsilon$-almost totally generic points for $F_I$ has $\mcl{L}_I$-measure greater than
$(1-\epsilon)\mcl{L}_I(I)$. 
\end{proof}

\begin{lemma}\label{lemma__centr-symm-eps-gen-pts}
Let $I$ be a horizontal interval on $\tX$, and let $0<\epsilon<\frac12$.
Then there exists a constant $C=C(I,\epsilon)>0$ and a pair of points in $I$,
symmetric with respect to the center of $I$, that are both $\epsilon$-almost totally generic
for the first return map on $I$.
\end{lemma}

\begin{proof}
Let $R:I\to I$ be reflection across the center of $I$. Since $R$ preserves
$\mcl{L}_I$, for every measurable subset $A\subset I$ one has
$$
\mcl{L}_I(R(A))=\mcl{L}_I(A).
$$

By Lemma~\ref{lemma__int-eps-gen-pts}, there exists $C=C(I,\epsilon)$ such that
$$
\mcl{L}_I\bigl(G^I_{\epsilon,\,C}\bigr)>(1-\epsilon)\mcl{L}_I(I).
$$
Since $\epsilon<\frac12$, this implies
$$
\mcl{L}_I\bigl(G^I_{\epsilon,\,C}\bigr)>\frac12\,\mcl{L}_I(I).
$$
Therefore the two sets $G^I_{\epsilon,\,C}$ and $R(G^I_{\epsilon,\,C})$ cannot be disjoint,
because otherwise
$$
\mcl{L}_I\bigl(G^I_{\epsilon,\,C}\bigr)+
\mcl{L}_I\bigl(R(G^I_{\epsilon,\,C}))\bigr)
\le \mcl{L}_I(I),
$$
contradicting the fact that each of these two sets has measure
$\mcl{L}_I(G^I_{\epsilon,\,C})>\frac12\,\mcl{L}_I(I)$.

Hence there exists $\tx\in G^I_{\epsilon,\,C}\cap R(G^I_{\epsilon,\,C})$. Then both
$\tx$ and $R(\tx)$ belong to $G^I_{\epsilon,\,C}$, and these two points are symmetric
with respect to the center of $I$. 
\end{proof}

\section{The circle argument}\label{UE-sect3}

As before, let $\pi: \tX\to X$ be the slit $\mathrm{N}$-cover with ramification points $\widetilde{P}=\pi^{-1} P$ and $\widetilde{Q}=\pi^{-1} Q$. For an $r>0$, let $\mathrm{U}(x, r)$ denote the $r$-neighborhood of the point $x$ in the flat metric.  Denote the \Teichmuller~flow by $g_t$, and write $X_t=g_t(X)$, and $\tX_t=g_t(\tX)$. Similarly, we set $P_{t} = g_{t}P$, $Q_{t} = g_{t}Q$.

Our main theorem in this section is the following geometric criterion for unique ergodicity.

\begin{theorem}\label{thm__The-CIRCLE_CRITERION}
Consider one of the slit endpoints on $X$ (without loss of generality, call it $Q$). Suppose that the following key assumption holds: there exist a number $R>0$ and a sequence of moments in time $\{t_k\}_{k\in\bbN_0}$, with $\lim_{k\to\infty}t_k=+\infty$, such that the $R$-neighborhood $\mathrm{U}_{t_k}(Q_{t_{k}}, R)\subset X_{t_k}$ of $Q_{t_k}$ is an embedded disk that does not contain $P_{t_k}=g_{t_{k}}P$.
Then the surface $\widetilde{X}$ is uniquely ergodic.
\end{theorem}

\begin{figure}[h]
\centering
\includegraphics[width=.5\linewidth]{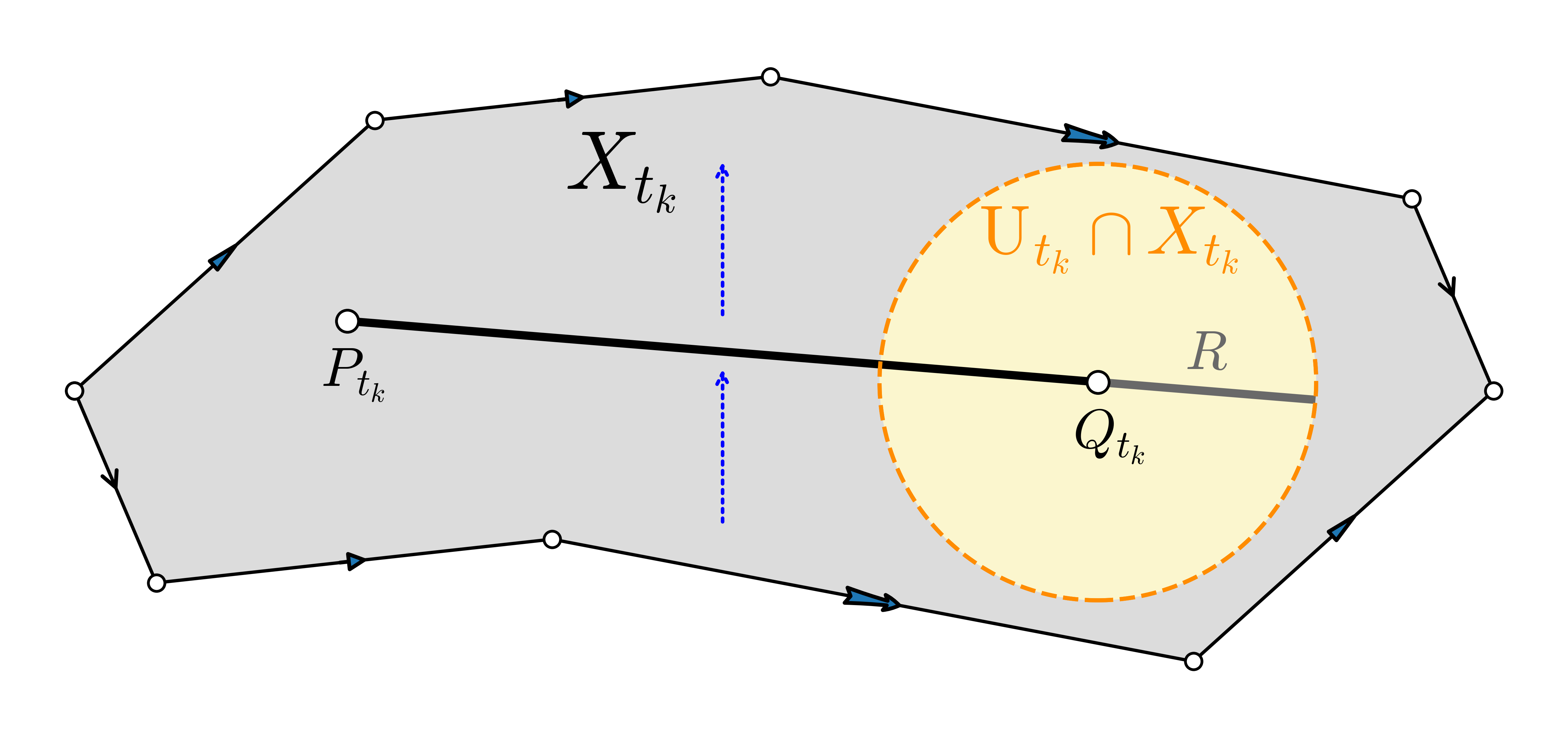}
\caption{Example of the key assumption of Theorem \ref{thm__The-CIRCLE_CRITERION}.}
\label{fig__X}
\end{figure}

Without loss of generality, we may assume that $t_0=0$ and choose the local coordinates around the slit endpoint $\widetilde{Q}$ so that it is the origin for the purposes of applying the \Teichmuller~flow. Suppose that $\tX$ is non-uniquely ergodic, and denote the normalized ergodic measures on it by $\mu_0,\mu_1,\,\ldots,\,\mu_{d-1}$, where $d\geq 2$ and $d\, |\,\mathrm{N}$.  We will prove the theorem by contradiction.

\begin{figure}[h]
\centering
\begin{subfigure}{.49\textwidth}
  \centering
  \includegraphics[width=.9\linewidth]{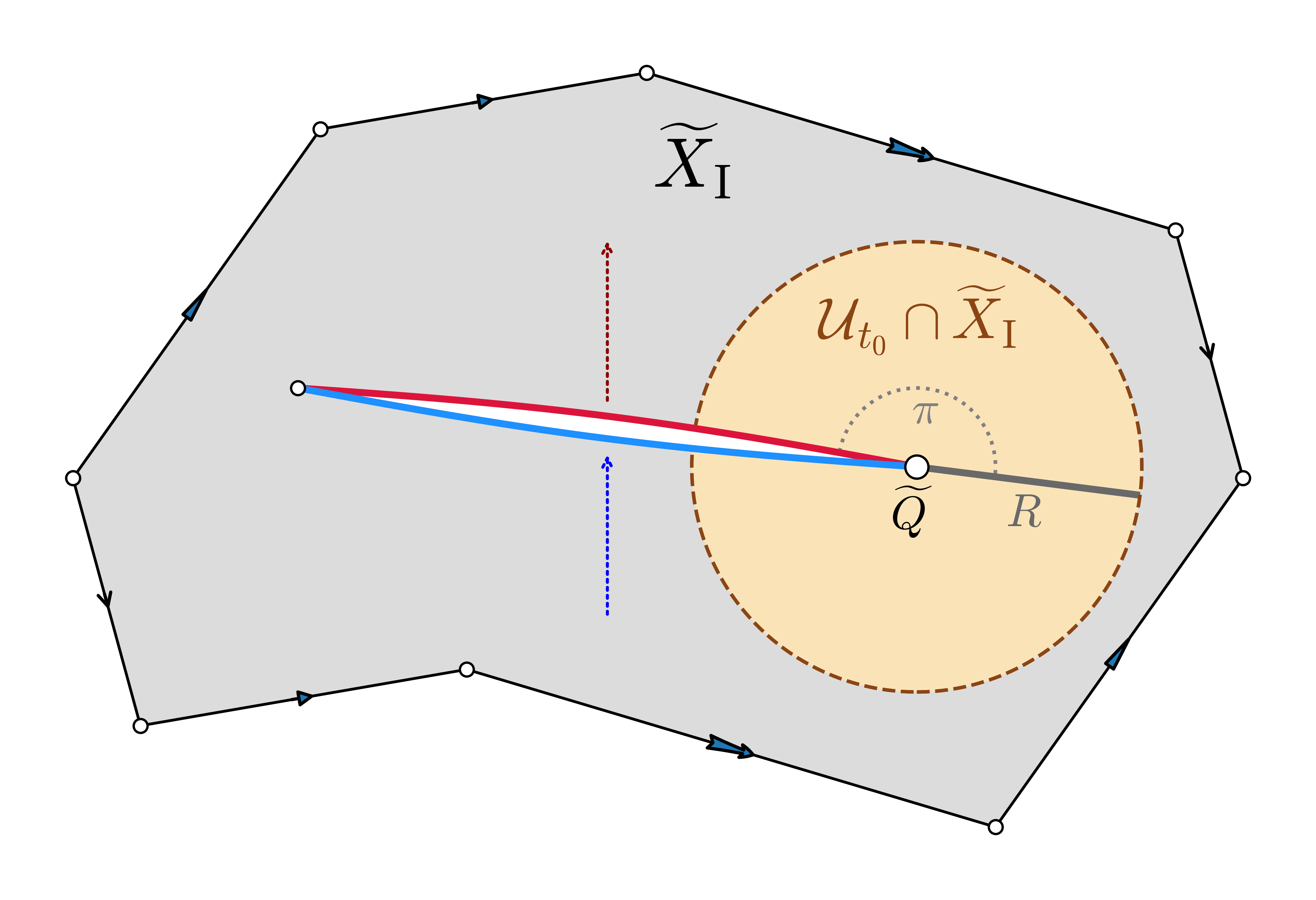}
  \label{}
\end{subfigure}%
\begin{subfigure}{.49\textwidth}
  \centering
  \includegraphics[width=.9\linewidth]{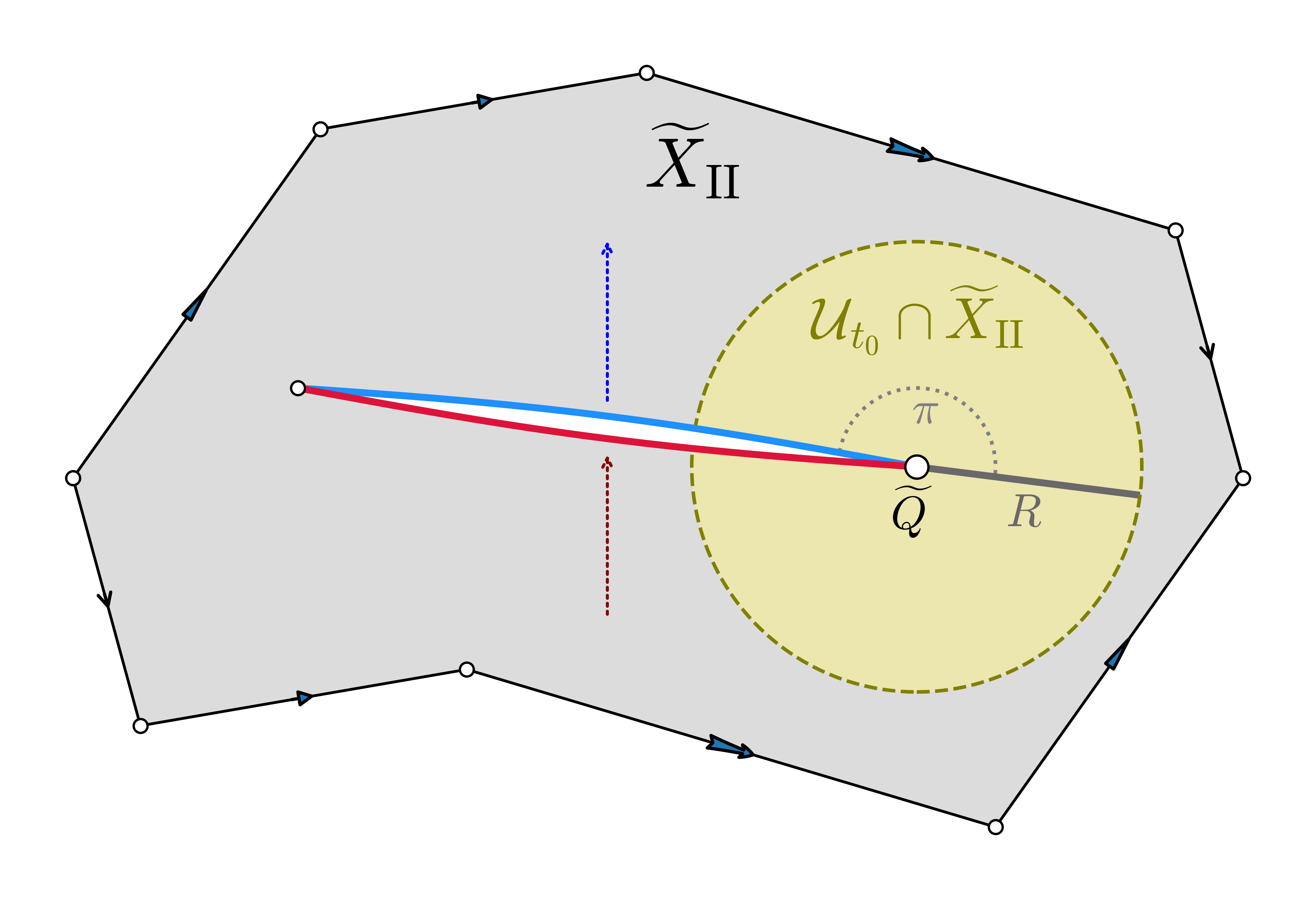}
  \label{}
\end{subfigure}
\caption{Example of the $R$-neighborhood $\mcl{U}_{t_0}$  of the endpoint $\widetilde{Q}$ of a slit on the surface $\tX_t$ for $\mathrm{N}=2$. }
\label{fig__choosing-neighbourhoods}
\end{figure}

\begin{proof}

Denote by $\mcl{U}_{t_k}$ the biggest neighborhood of $\widetilde{Q}_{t_k}=g_{t_k}\widetilde{Q}$ such that $\pi(\mcl{U}_{t_k})=\mathrm{U}_{t_k}$ (see Figure \ref{fig__choosing-neighbourhoods} for the visual representation).

\emph{I. Passing to a subsequence of times $\{t_{k_i}\}_{i\in\bbN_0}$.}

First, we pick a certain subsequence $\{t_{k_{i}}\}_{i\in\bbN_0}$ of the sequence $\{t_{k_i}\}_{i\in\bbN_0}$. Let us consider the sequence $\{\epsilon_{i} \}_{i\in\bbN}$, $\epsilon_i=1/(i+2)$ of reciprocals of natural numbers. The first element we pick is simply $t_0=0$. Starting from $i=1$, for each $\epsilon_{i}$, we pick a moment in time $t_{k_{i}}$ such that 
\begin{equation}\label{eq__condition-t_k_i}
\frac{\sqrt{3}}{2}e^{t_{k_{i}}}R\; > \; \max\{T_{\epsilon_{i}},\, T_{\epsilon_{i+1}}\},
\end{equation}
where $T_{\epsilon_{i}}$ and $T_{\epsilon_{i+1}}$ are the constants from Theorem \ref{thm__almost-generic-points}
corresponding to $\epsilon = \epsilon_{i}$ and $\epsilon = \epsilon_{i+1}$ respectively (the motivation from this, admittedly somewhat obscure, choice will become apparent in Part V of the proof; more specifically, in Lemma \ref{lemma__A-and-B-limits}). Moreover, we pick our new elements $t_{k_{i}}$ in such a way that the subsequence remains strictly increasing. 

For the sake of simplicity, we will abuse our notation and denote the resulting subsequence of times $t_{k_i}$ by $\{t_{i}\}_{i\in\bbN_0}$. We will also sometimes write $\tX_{t_i} = \tX_{i}$.

\begin{figure}[h]
\centering
\begin{subfigure}{.49\textwidth}
  \centering
  \includegraphics[width=.99\linewidth]{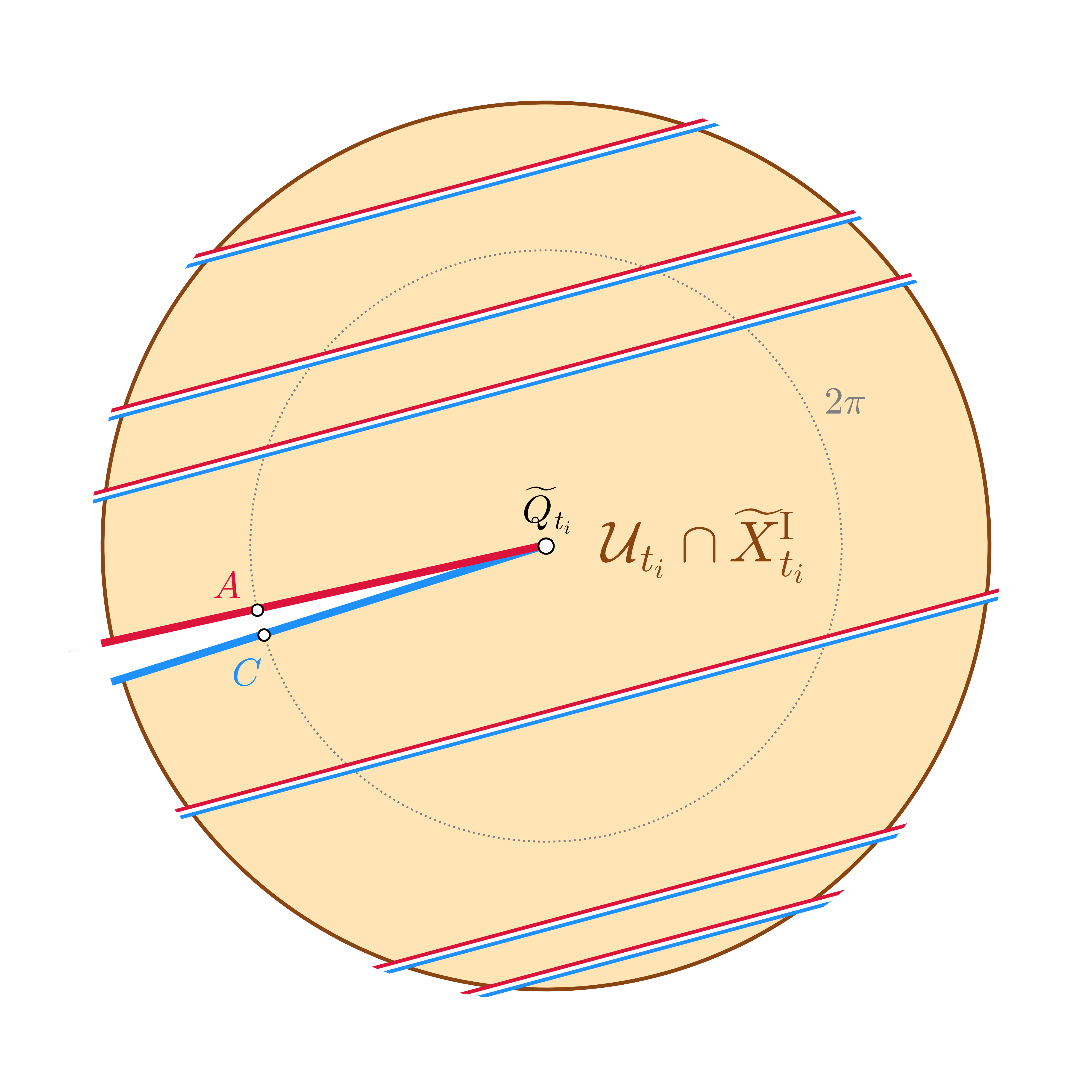}
  \label{}
\end{subfigure}%
\begin{subfigure}{.49\textwidth}
  \centering
  \includegraphics[width=.99\linewidth]{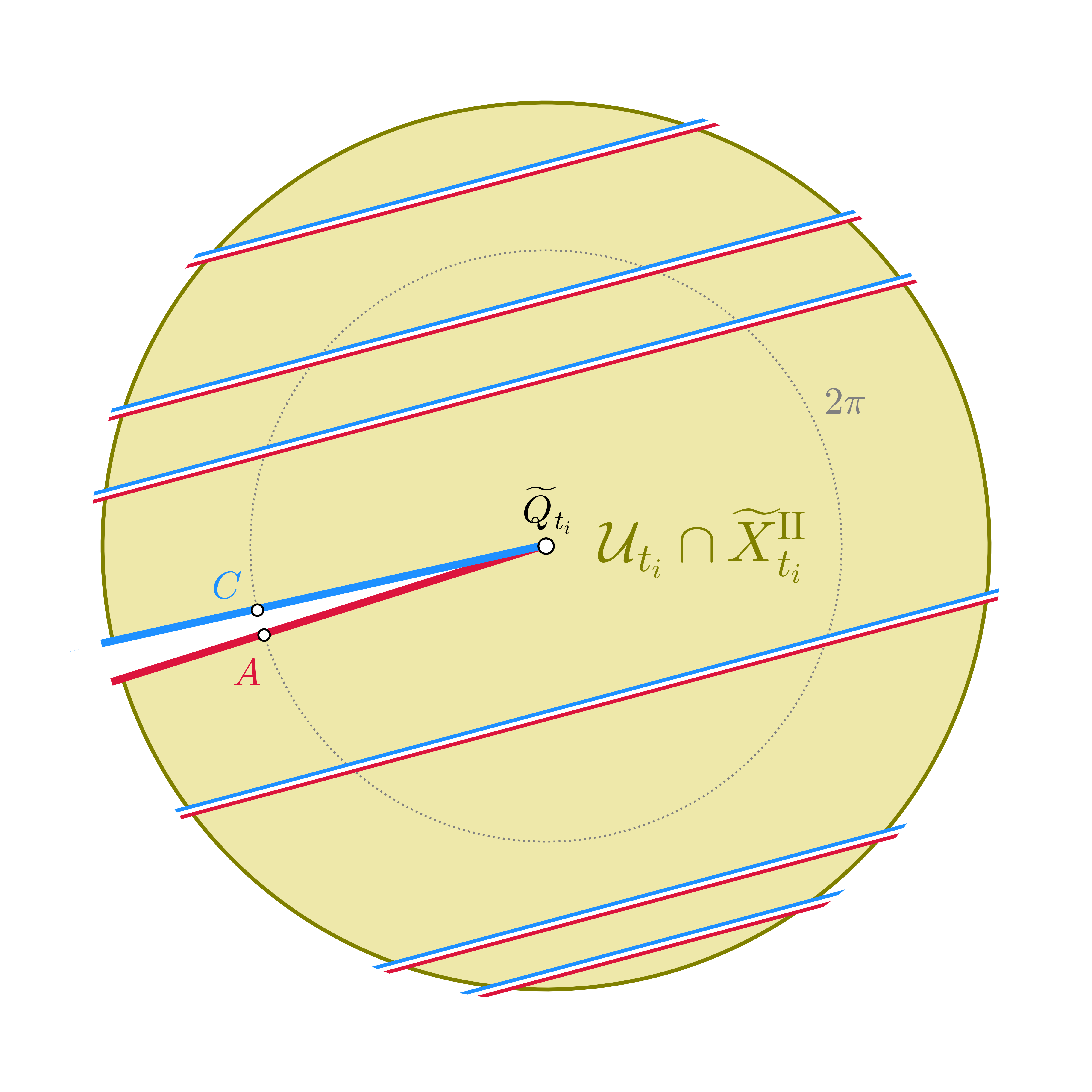}
  \label{}
\end{subfigure}
\caption{Neighborhood of the endpoint $O$ of radius $r$ inside $\tX_{t_i}$ for $\mathrm{N}=2$.}
\label{fig__neighorhood}
\end{figure}

\emph{II. The neighborhood $\mcl{U}_{t_i}$.}

Take a point $A\neq \widetilde{Q}_{t_i}$ on the intersection of $\mcl{U}_{t_i}$ with the slit. Consider the clockwise path around $\widetilde{Q}_{t_i}$ on $\tX_{t_i}$ starting from $A$. It is clear that, after rotating by $2\pi$, we will arrive at $C=\varsigma^{-1}(A)$, and after rotating by the angle $2M\pi$, we will return back to $A$. In this process, we may cross the slit again some number of times. Every time we do so, we jump between the sheets of $\tX_{t_i}$. However, since the neighbourhoods $U_{t_i}\cap \tX_{t_i}^j$ are clearly symmetric for each copy $j$, inside $\mcl{U}_{t_i}$, we can reglue the stripes formed by the slit. Therefore, $\mcl{U}_{t_i}$ is simply a glueing of $\mathrm{N}$ radius $r_i$ circles centered at $\widetilde{Q}_{t_i}$ and cut along their parallel radii. We will denote those  circles by $\mcl{U}^j_{t_i}$, $0\leq j\leq \mathrm{N}-1$. Take two adjacent ones, which we denote by $\mcl{U}^+_{t_i}$ and $\mcl{U}^-_{t_i}$ (indexing can be chosen, for example, by fixing a side of the slit and stating that $\mcl{U}^+_{t_i}$ corresponds to counterclockwise rotation by at most $2\pi$ and $\mcl{U}^-_{t_i}$ corresponds to clockwise rotation by at most $2\pi$). See Figure \ref{fig__neighorhood} for the visual representation.

\emph{III. Geometry inside $\mcl{U}^+_{t_i}$.}


In this part, we aim to build a special sequence of pairs of vertical intervals on the initial surface $\tX$. 

For a fixed $i\in\bbN_0$, consider the neighborhood $\mcl{U}_{t_i}$ of the slit endpoint $\widetilde{Q}_{t_i}$. As before, we will represent this neighborhood as $\mathrm{N}$ disks, cut and glued along a symmetric radius segment determined by the slit. Let us first describe all our geometric constructions for the forward vertical flow on $\mcl{U}^+_{t_i}$; from symmetry, geometric constructions on $\mcl{U}^-_{t_i}$, as well as those for the backward vertical flow, will repeat these. See Figure \ref{fig__X_i-circle} for the visual reference.

Prolong the slit to a diameter of $\mcl{U}^+_{t_i}$. More specifically, starting from the ``top'' part of the slit, we rotate it counterclockwise inside $\mcl{U}^+_{t_i}$ around the endpoint $\widetilde{Q}_{t_i}$ in and consider the separatrix segment that goes in the direction of the slit and is contained in $\mcl{U}^+_{t_i}$. 

By Lemma \ref{lemma__centr-symm-gen-pts}, Corollary \ref{cor__measure-generic-symmetry}, as well as Theorem \ref{thm__almost-generic-points} and Lemmas \ref{lemma__int-eps-gen-pts}, \ref{lemma__centr-symm-eps-gen-pts}, we can find points $\hhat{A}_{i}$ on the top side of the slit, $\hhat{C}_i=\varsigma(\hhat{A}_{i})$ on the bottom side of the slit, and $\hhat{B}_{i}$ on the separatrix such that they are all totally generic (possibly for different measures), $\epsilon_i$-almost totally generic (from Theorem \ref{thm__almost-generic-points}, since the area of the neighborhood $\mcl{U}_{t_i}$ is strictly bigger than $2\epsilon_i$ by assumption), and satisfy  $0<\hhat{r}_i:=|\widetilde{Q}_{t_i}\hhat{A}_{i}|=|\widetilde{Q}_{t_i}\hhat{B}_{i}|< R/2$. 

Let $\hhat{A}_{i}'$ and $\hhat{B}_{i}'$, respectively, be the intersections of the vertical projections on the horizontal diameter of $\mcl{U}_{t_i}$ (note that  $\hhat{A}_{i}'$ may belong to $\mcl{U}^-_{t_i}$ dependent on the angle between the slit and the vertical flow). We will denote by $\hhat{H}_i$ the signed vertical distance from $\hhat{B}_{i}$ to $\hhat{B}_{i}'$ (i.e. $\hhat{H}_i$ is positive if $\hhat{B}_{i}'=v_{|\hhat{H}_i|}\hhat{B}_{i}$ and negative if $\hhat{B}_{i}=v_{|\hhat{H}_i|}\hhat{B}_{i}'$). Clearly, the signed vertical distance from $\hhat{A}_{i}$ to $\hhat{A}_{i}'$ is $-\hhat{H}_i$.

Let $\hhat{M}_{i}^{+,\,top}$ and $\hhat{N}_{i}^{+,\,top}$, respectively, be the points of first intersection with the boundary of $\mcl{U}^+_{i}$ of the forward vertical trajectories starting at the points $\hhat{A}_{i}$ and $\hhat{B}_{i}$. We will denote by $\hhat{L}_i$ the vertical distance from $\hhat{B}_{i}'$ to $\hhat{M}_{i}^{+,\,top}$. Note that $\frac{\sqrt{3}}{2}r_i\leq \hhat{L}_i<r_i$.

Evidently, $\hhat{L}_i=|\hhat{B}_{i}'\hhat{N}_{i}^{+,\,top}|=|\hhat{A}_{i}'\hhat{M}_{i}^{+,\,top}|$. Moreover, $|\hhat{B}_{i}\hhat{N}_{i}^{+,\,top}|=\hhat{L}_i+\hhat{H}_i$ and $|\hhat{A}_{i}\hhat{M}_{i}^{+,\,top}|=\hhat{L}_i-\hhat{H}_i$.

\begin{figure}[h]
\centering
\includegraphics[width=.5\linewidth]{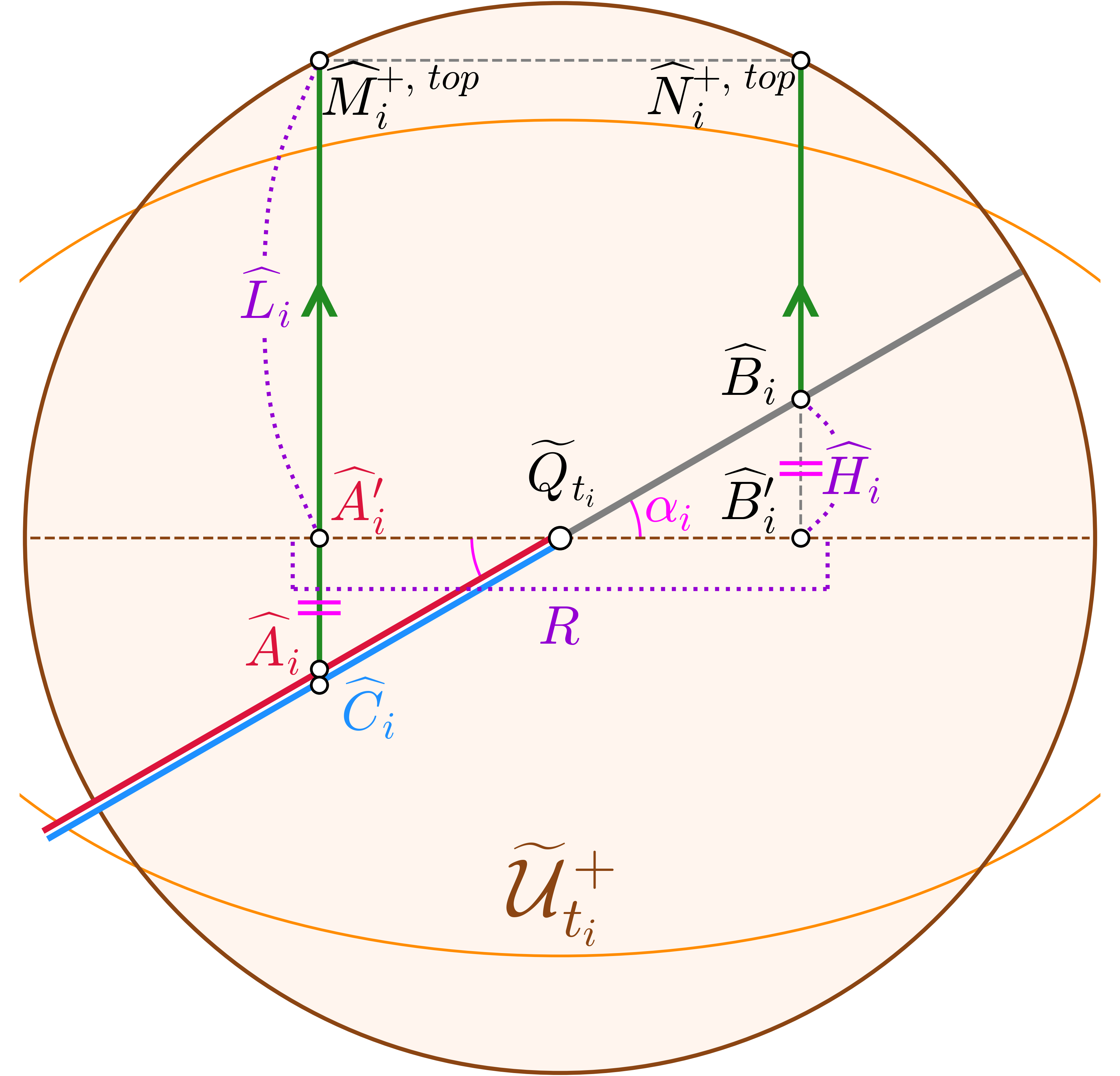}
\caption{Vertical interval sequence construction on $\mcl{U}^+_{t_i}$.}
\label{fig__X_i-circle}
\end{figure}

Next, we describe the preimage of this picture on $\mcl{U}^+_{t_0}\subset\tX$. See Figure \ref{fig__X-circle} for visual reference.

Let $A_{i} = g_{t_{i}}^{-1}\hhat{A}_{i}$, $B_{i} = g_{t_{i}}^{-1}\hhat{B}_{i}$, $A_{i}' = g_{t_{i}}^{-1}\hhat{A}_{i}'$, $B_{i}' = g_{t_{i}}^{-1}\hhat{B}_{i}'$, $M_{i}^{+,\,top} = g_{t_{i}}^{-1}\hhat{M}_{i}^{+,\,top}$, and $N_{i}^{+,\,top} = g_{t_{i}}^{-1}\hhat{N}_{i}^{+,\,top}$ be the respective preimages of the points $\hhat{A}_{i}$, $\hhat{B}_{i}$, $\hhat{A}_{i}'$, $\hhat{B}_{i}'$, $\hhat{M}_{i}^{+,\,top}$, and $\hhat{N}_{i}^{+,\,top}$. 
Note that $A_i'$ and $B_i'$ remain equidistant from $\widetilde{Q}$ and are still totally generic points on $\tX$ (as the \Teichmuller~flow preserves measure-generality of non-critical points, see Lemma \ref{lemma__Teichmuller-preserves-genericity}). 
The same is true for $A_i'$ and $B_i'$, as well as $M_i^{+,\,top}$ and $N_i^{+,\,top}$. By construction, the distance $2R_i=|A_iB_i|=2e^{-t_i}\hhat{r}_i<e^{-t_{i}}r_i$. We can find a subsequence $\{t_{i_{j}}\}_{j\in\bbN}$ of the sequence of times $\{t_{i}\}_{i\in\bbN}$ such that the distances $r_{i_{j}}$ are strictly decreasing. For simplicity, we will abuse our notation and denote this subsequence of times by $\{t_{i}\}_{i\in\bbN}$ too. 

\begin{figure}[h]
\centering
\includegraphics[width=.5\linewidth]{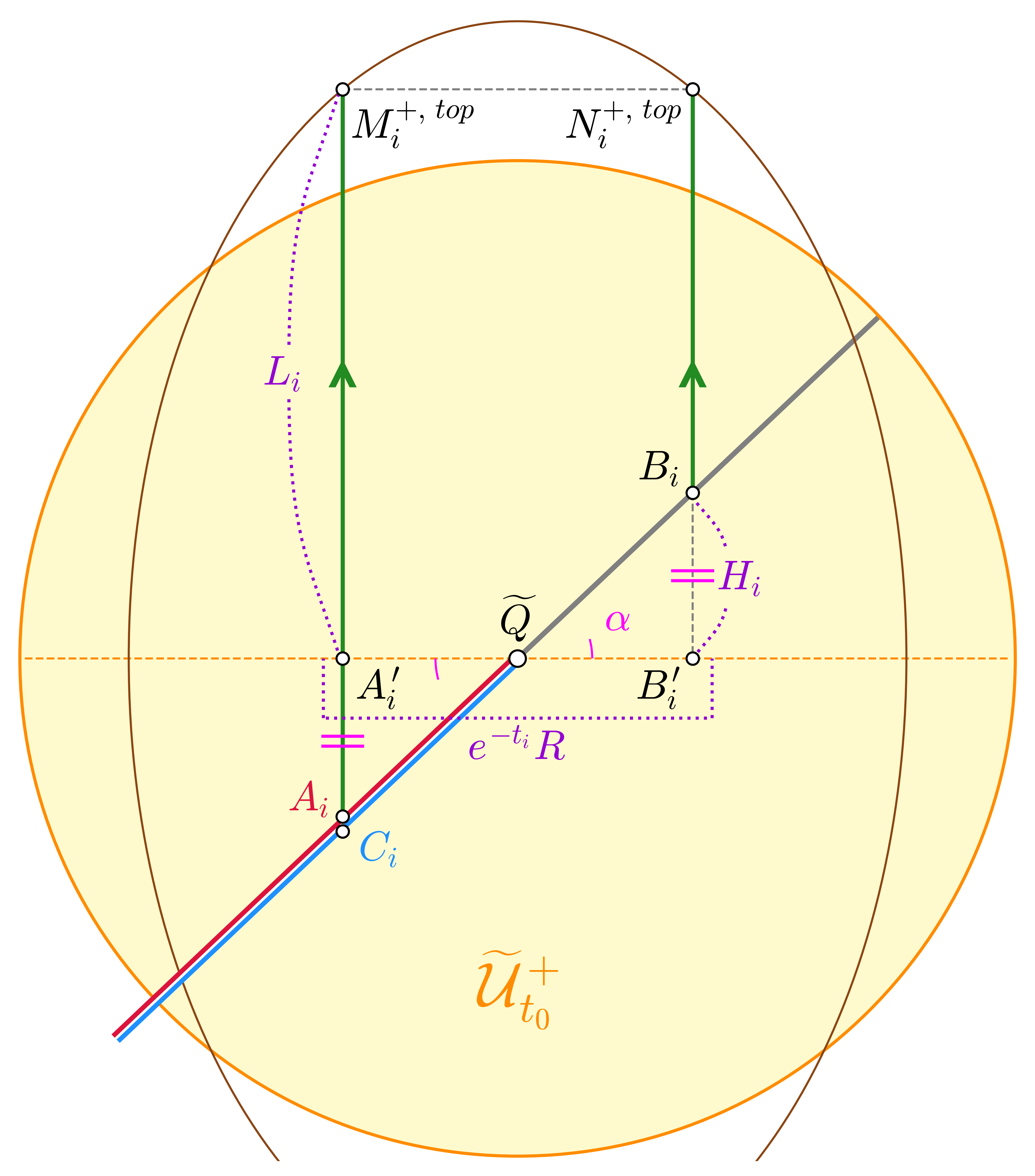}
\caption{Construction of two sequences of vertical intervals, $[A_i,\,M_{i}^{+,\,top}]$ and $[B_i,\,N_{i}^{+,\,top}]$, on $\mcl{U}_{t_0}$.}
\label{fig__X-circle}
\end{figure}

\begin{figure}[h]
\centering
\includegraphics[width=.5\linewidth]{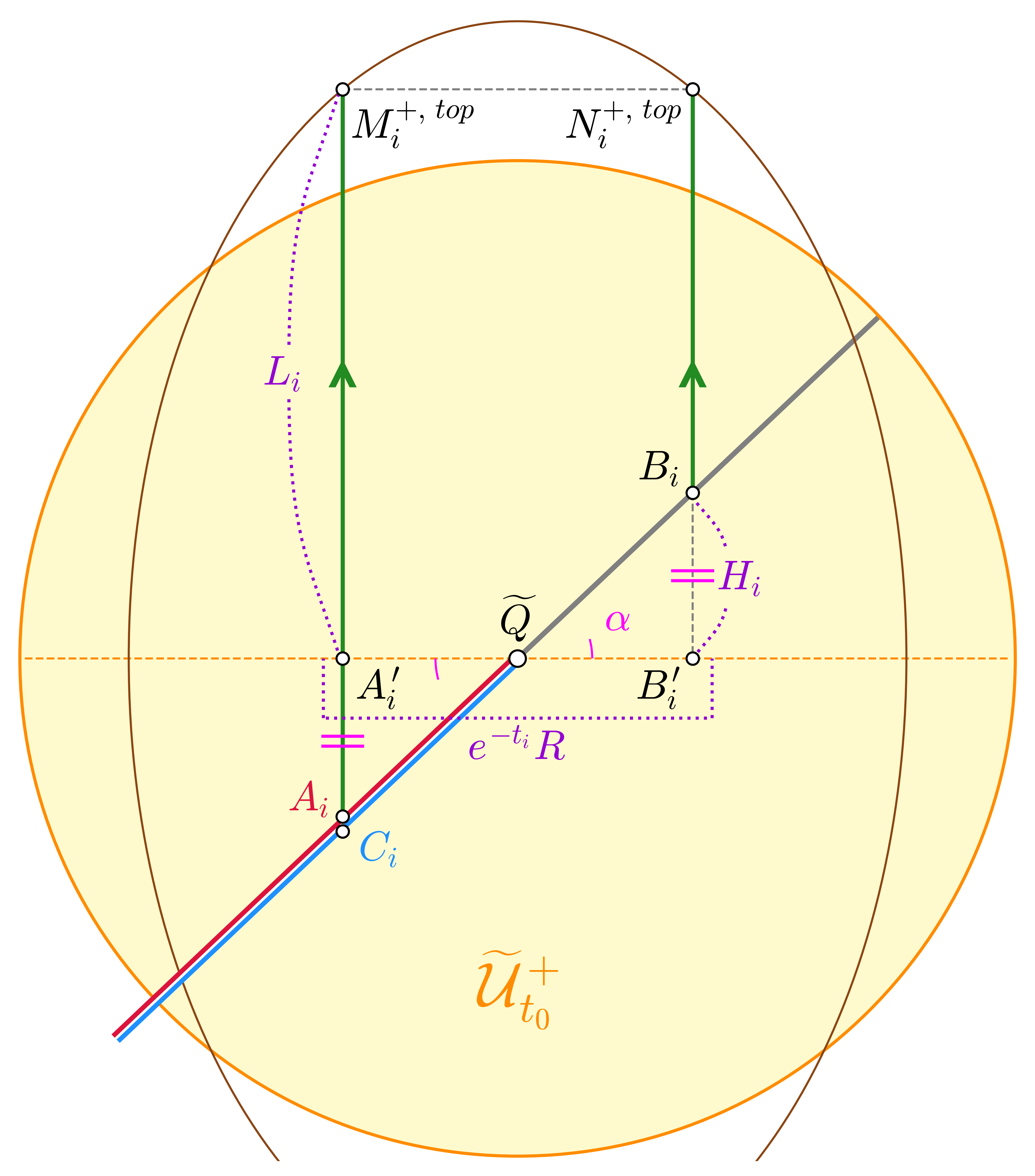}
\caption{Construction of two sequences of vertical intervals, $[A_i,\,M_{i}^{+,\,top}]$ and $[B_i,\,N_{i}^{+,\,top}]$, on $\mcl{U}_{t_0}$.}
\label{fig__X-circle}
\end{figure}

For each $t_{i}$ in the subsequence, consider the parallel intervals $[A_i',\,M_{i}^{+,\,top}]$ and $[B_i',\,N_{i}^{+,\,top}]$ on the surface $\tX$. Clearly, they are both of length $L_i=e^{t_i}\hhat{L}_i$, and 
\begin{equation}
T_{\epsilon_i} \leq \frac{\sqrt{3}}{2}e^{t_i}r_i\leq\, L_i \,<e^{t_i}r_i, \quad
\text{ therefore}\quad \lim_{i\to\infty}L_i\geq \lim_{i\to\infty} {\frac{\sqrt{3}}{2}e^{t_i}r_i}=+\infty.
\end{equation} 
Recall that the value of $T_{\epsilon_i}$ comes from Theorem \ref{thm__almost-generic-points}. 

Moreover, the signed distance $H_i=e^{t}\hhat{H}_i$ from $B_i$ to $B_i'$ clearly tends to $0$ as $i$ goes to infinity (thus so does the distance $|\widetilde{Q}B_i'|$).

\emph{IV. Geometry inside the entire $\mcl{U}_{t_i}$.}

Since $\mcl{U}^+_{t_i}=\varsigma(\mcl{U}^-_{t_i})$, so the exact same geometric construction can be repeated in $\mcl{U}^-_{t_i}$. Note that $\hhat{A}_i,\,\hhat{C}_i\in \mcl{U}^+_{t_i}\cap \mcl{U}^-_{t_i}$. The point $\hhat{D}_i=\varsigma(\hhat{B}_i)$ has to be totally generic since $\hhat{B}_i$ is totally generic and not critical. Moreover, $\hhat{D}_i$ has to be $\epsilon_i$-almost totally generic from symmetry (see Lemma \ref{lemma__measure-eps-generic-symmetry}). The exact same extension construction as above can be done in the backward direction of the vertical flow on $\tX$. The only slight issue is keeping track of the relative positions of rays and points, as demonstrated by Figure \ref{fig__full-U_0}.

    \begin{figure}[h]
\centering
\begin{subfigure}{.49\textwidth}
  \centering
  \includegraphics[width=.99\linewidth]{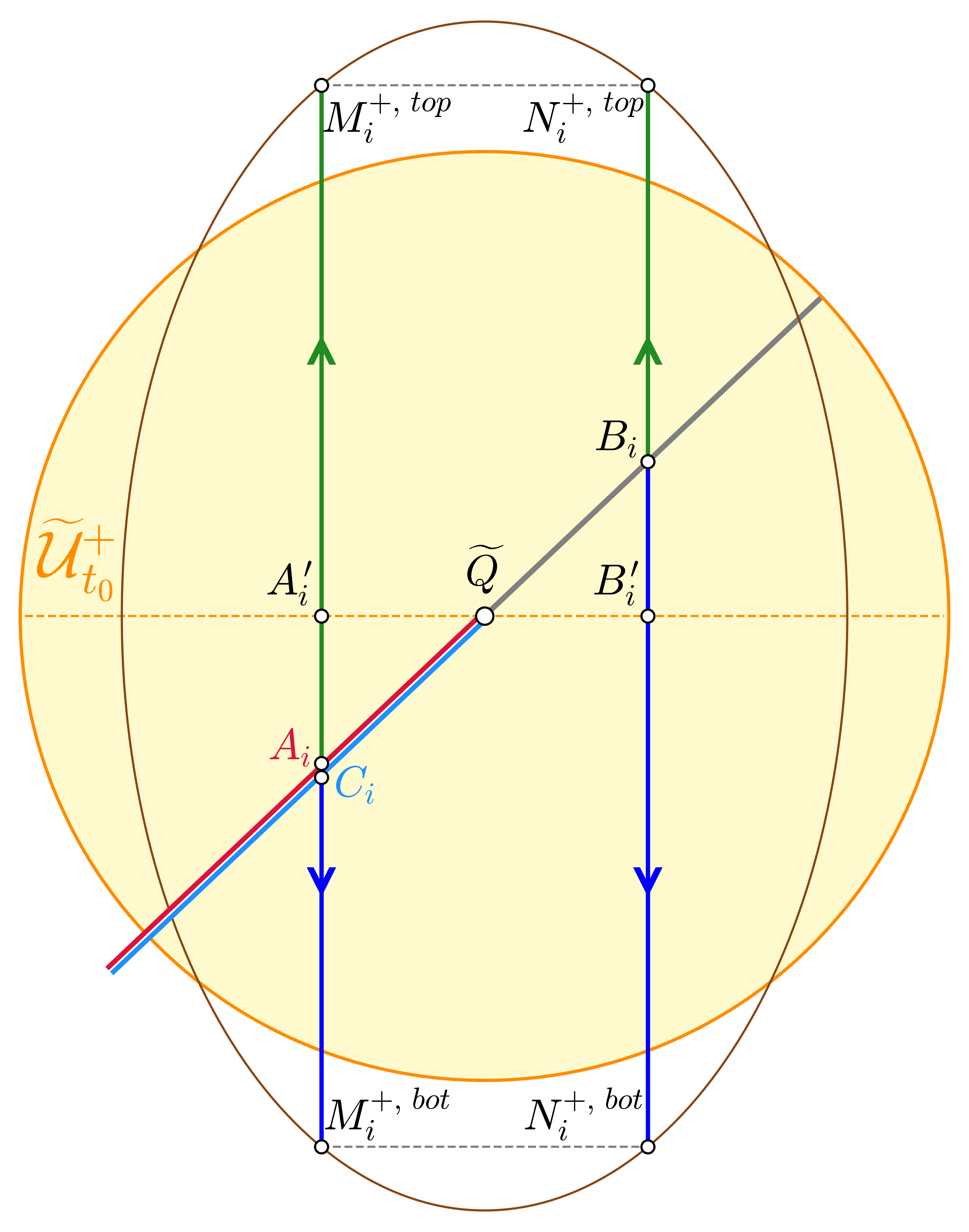}
  \label{}
\end{subfigure}%
\begin{subfigure}{.49\textwidth}
  \centering
  \includegraphics[width=.99\linewidth]{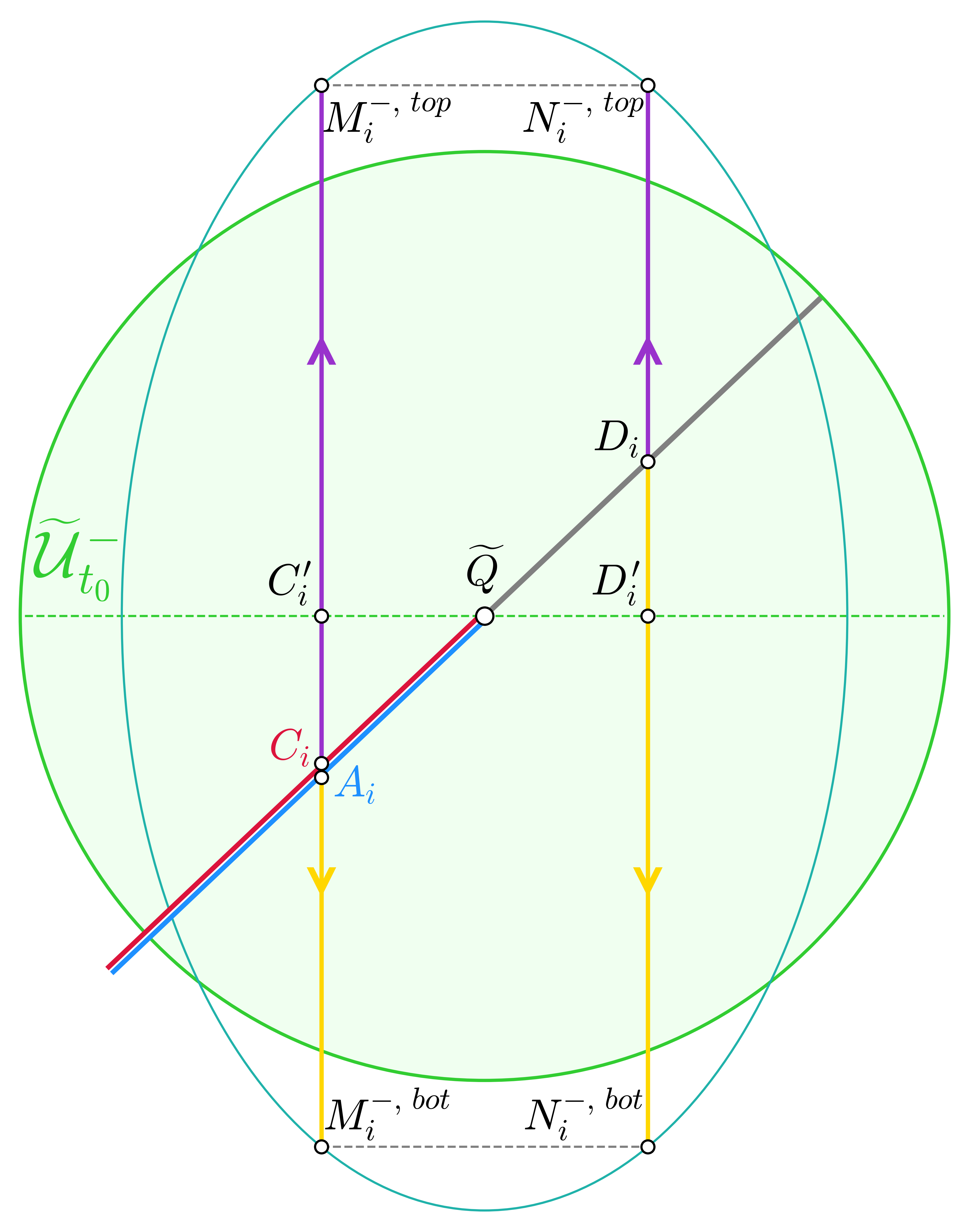}
  \label{}
\end{subfigure}
\caption{The neighborhood $\mcl{U}_{t_0}$.}
\label{fig__full-U_0}
\end{figure}

\emph{V. Forward flow limits.}

Let $f$ be a function on $\tX$, and let $t\geq 0$. For simplicity, slightly change the notation of the integrals from Equation \eqref{eq__int-def}:
\begin{equation}\label{eq__special-int-def}
    \Iup{f}{A_{i}} := \Iup[L_i-H_i]{f}{A_{i}} = \frac{1}{L_{i}-H_i}\int_{0}^{L_{i}-H_i} f(v_{s}A_{i})\,\rd s,
\end{equation}
\begin{equation}\label{eq__special-int-def_3}
\Iup{f}{A_{i}'} := \Iup[L_i]{f}{A_{i}'} = \frac{1}{L_{i}}\int_{0}^{L_{i}} f(v_{s}A_{i}')\,\rd s,
\end{equation}
\begin{equation}\label{eq__special-int-def_4}
\Iup{f}{B_{i}} := \Iup[L_i+H_i]{f}{B_{i}} = \frac{1}{L_{i}+H_i}\int_{0}^{L_{i}+H_i} f(v_{s}B_{i})\,\rd s,
\end{equation}
\begin{equation}\label{eq__special-int-def_5}
\Iup{f}{B_{i}'} := \Iup[L_i]{f}{B_{i}'} = \frac{1}{L_{i}}\int_{0}^{L_{i}} f(v_{-s}B_{i}')\,\rd s.
\end{equation}

\begin{lemma}\label{lemma__A-and-A-prime-limits}
    If $\lim_{i\to\infty}\Iup{f}{B_{i}'}$ exists and is finite, then so does $\lim_{i\to\infty}\Iup{f}{B_{i}}$, and vice versa. Moreover, these two limits are equal. The same holds for $\lim_{i\to\infty}\Iup{f}{A_{i}'}$ and $\lim_{i\to\infty}\Iup{f}{A_{i}}$. 
\end{lemma}

\begin{proof}It is enough to prove the statement for $B_i'$ and $B_i$. There are two cases to consider. 

First, assume that $B_i'$ is an image of $B_i$ under the forward vertical flow: $B_i'=g_{|H_i|}B_i$. Then $H_i$ is positive, and 
\begin{multline}
\left\|\Iup{f}{B_{i}} - \Iup{f}{B_{i}'}\right\|
=
\left\|
\frac{1}{L_{i}+H_i}\int_{0}^{L_{i}+H_i} f(v_{s}B_{i})\,\rd s
-
\frac{1}{L_{i}}\int_{0}^{L_{i}} f(v_{s}B_{i}')\,\rd s
\right\|
=\\=
\left\|
\frac{1}{L_{i}+H_i}\int_{0}^{H_{i}} f(v_{s}B_{i})\,\rd s+
\left(\frac{1}{L_{i}+H_i}-\frac{1}{L_i}\right)\int_{0}^{L_{i}} f(v_{s}B_{i}')\,\rd s
\right\|
= \\ =
\frac{H_i}{L_i+H_i}\left\|
\Iup[H_i]{f}{B_i}-\Iup{f}{B_i'}
\right\|.
\end{multline}
As $i$ goes to infinity, $L_i$ goes to infinity and $H_i$ goes to $0$. Therefore $\Iup[H_i]{f}{B_i}$ goes to $0$. Therefore, if $\lim_{i\to\infty}\|\Iup{f}{B_i'}\|$ exists and is finite, then $\lim_{i\to\infty}\|\Iup{f}{B_i}\|$ exists and is finite, and these two limits are equal. The other direction (from existence of limits for $B_i$ to existence of limits for $B_i'$) follows analogously from considering:
\begin{multline}
\left\|\Iup{f}{B_{i}} - \Iup{f}{B_{i}'}\right\|
=
\left\|
\frac{1}{L_{i}+H_i}\int_{0}^{L_{i}+H_i} f(v_{s}B_{i})\,\rd s
-
\frac{1}{L_{i}}\int_{0}^{L_{i}} f(v_{s}B_{i}')\,\rd s
\right\|
=\\=
\left\|
\left(\frac{1}{L_{i}+H_i}-\frac{1}{L_i}\right)\int_{0}^{L_{i}+H_i} f(v_{s}B_{i})\,
+
\frac{1}{L_{i}}
\int_{0}^{H_i} f(v_{s}B_{i})\,\rd s
\right\|
=
\frac{H_i}{L_i}\left\|
\Iup[H_i]{f}{B_i}-\Iup{f}{B_i}.\right\|
\end{multline}

Calculations for the case when $B_i'$ is a preimage of $B_i$ under the forward vertical flow ($B_i=g_{|H_i|}B_i'$) are done analogously.

\if
Then $H_i$ is negative, and 
\begin{multline}
\left\|\Iup{f}{B_{i}} - \Iup{f}{B_{i}'}\right\|
=
\left\|
\frac{1}{L_{i}-|H_i|}\int_{0}^{L_{i}-|H_i|} f(v_{s}B_{i})\,\rd s
-
\frac{1}{L_{i}}\int_{0}^{L_{i}} f(v_{s}B_{i}')\,\rd s
\right\|
=\\=
\left\|
\left(\frac{1}{L_{i}-|H_i|}-\frac{1}{L_i}\right)\int_{0}^{L_{i}-|H_i|} f(v_{s}B_{i})\,\rd s
-\frac{1}{L_{i}}\int_{0}^{|H_{i}|} f(v_{s}B_{i}')\,\rd s
\right\|
=
\\
=
\frac{|H_i|}{L_i-|H_i|}\left\|
\Iup{f}{B_i}-\Iup[|H_i|]{f}{B_i'}
\right\|.
\end{multline}
As $i$ tends to infinity, $L_i$ tends to infinity and $H_i$ tends to $0$. Therefore $\Iup[H_i]{f}{B_i}$ tends to $0$. Therefore, if $\lim_{i\to\infty}\|\Iup{f}{B_i}\|$ exists and is finite, then $\lim_{i\to\infty}\|\Iup{f}{B_i'}\|$ exists and is finite, and these two limits are equal. The other direction (from the existence of limits for $B_i'$ to the existence of limits for $B_i$) follows analogously. 
\fi
\end{proof}

Consider an arbitrary $1$-Lipschitz function $h$ on $\tX$. Since $\tX$ is a compact metric space, and since all measures $\mu_i$ are absolutely continuous with respect to the Lebesgue measure $\mcl{L}_{\tX}$ on $\tX$ (see Lemma \ref{lemma__n-measures}), the function $h$ is going to be both $\mu_i$-integrable on $\tX$ for any $0\leq i\leq d-1$.

\begin{lemma}\label{lemma__A-and-B-limits}
The sequence $\{\Iup{h}{A_i}\}_{n\in\bbN_+}$ converges to a finite limit, and so does the sequence $\{\Iup{h}{B_i}\}_{n\in\bbN_+}$. Moreover, these two limits are equal to each other and to $\Av_{m}(h)$ for some $m=\mu_j$, $0\leq j\leq d-1$.
\end{lemma}

\begin{proof}\, We first prove that the sequence $\{\Iup{h}{A_i}\}_{n\in\bbN_+}=\{\Iup[L_i-H_i]{h}{A_i}\}_{n\in\bbN_+}$ converges (the proof for the sequence $\{\Iup{h}{B_i}\}_{n\in\bbN_+}=\{\Iup[L_i+H_i]{h}{B_i}\}_{n\in\bbN_+}$ is analogous). Then we will prove that the limits are equal. 

\emph{Step 1:}  
 From the previous lemma, we can instead prove convergence of the sequence $\{\Iup[L_i]{h}{A_i'}\}_{n\in\bbN_+}$. Recall that each point $A_{i}$, $i\geq 1$, is $\epsilon_{i}$-almost totally generic w.r.t. an ergodic invariant measure $m$ on $\tX$. While the same is not necessarily true for points $A_{i}'$, $i\geq 1$, we can denote by $d_{A_i}$ the difference $\left\|\Iup{f}{A_{i}} - \Iup{f}{A_{i}'}\right\|$, implying that a point $A_i'$ is at most  $d_{A_i}$-away from being $\epsilon_{i}$-almost totally generic. Keep in mind that $d_{A_i}$ tends to $0$ as $i$ tends to $\infty$.

 We want to show that there exists $N \in \mathbb{N}_+$ such that the measure $m$ is the same for all $A_{i}'$ with $i > N$. 

Suppose not. Then, for any $N \in \mathbb{N}_+$, there exist two consecutive points $A_{i}$ and $A_{i+1}$, $i\geq N$, such that $A_{i}$ is $\epsilon_{i}$-almost totally $\mu$-generic for some $\mu=\mu_{j_1}$ and $A_{i+1}$ is $\epsilon_{i+1}$-almost totally $\nu$-generic for some $\mu=\mu_{j_2}$, $j_1\neq j_2$. By the construction from the main proof, we have $L_{i}\geq T_{\epsilon_{i}}$, $L_{i}\geq T_{\epsilon_{i+1}}$, and $L_{i+1}\geq T_{\epsilon_{i+1}}$. Therefore
\begin{equation}\label{eq__Ai_Ai+1_ineqs}
\left\|\Iup[L_i]{h}{A_i'} - \Av_{\mu}(h)\right\| < \epsilon_{i}+d_{A_i}
\quad  
\text{ and } 
\quad 
\left\|\Iup[L_{i+1}]{h}{A_{i+1}'} - \Av_{\nu}(h)\right\|  < \epsilon_{i+1}+d_{A_{i+1}}.
\end{equation}
Moreover, given Condition \eqref{eq__condition-t_k_i} on the sequence of times $t_i$, we have
\begin{equation}\label{eq__one_more_Ai+1_ineq}
\left\| \Iup[L_i]{h}{A_{i+1}'} - \Av_{\nu}(h)\right\| < \epsilon_{i+1}+d_{A_{i+1}}.
\end{equation}

Since $h$ is a $1$-Lipschitz function, $A_{i+1}'$ is positioned between $O$ and $A_i'$ by construction, and $|OA_i'|\leq \frac{1}{2}R$, we have 
\begin{multline}\label{eq__A_i_Ai+1-R/2-ineq}
\left\|\Iup[L_i]{h}{A_i'} - \Iup[L_i]{h}{A_{i+1}'} \right\| =  
\frac{1}{L_{i}}\left\|\int_{0}^{L_{i}} h(v_{s}A_{i}')\,\rd s -\int_{0}^{L_{i}} h(v_{s}A_{i+1}')\,\rd s \right\| 
\leq 
\\
\leq
\frac{1}{L_{i}} \int_{0}^{L_{i}} \left\|h(v_{s}A_{i}')-h(v_{s}A_{i+1}') \right\| \,\rd s 
\leq 
\frac{1}{L_{i}} \int_{0}^{L_{i}} e^{-t_{i}}\frac{R}{2}  \ \,\rd s = e^{-t_{i}}\frac{R}{2}.
\end{multline} 

Combining \eqref{eq__Ai_Ai+1_ineqs}, \eqref{eq__one_more_Ai+1_ineq} and \eqref{eq__A_i_Ai+1-R/2-ineq}, we get
\begin{multline*}
\left\|\Av_{\nu}(h) - \Av_{\mu}(h)\right\| 
\leq
\left\|\Iup[L_i]{h}{A_{i}'} - \Av_{\mu}(h)\right\|  + d_{A_{i}}
+ 
\left\|\Iup[L_i]{h}{A_{i+1}'} - \Iup[L_i]{h}{A_{i}'}\right\| + d_{A_{i}} + d_{A_{i+1}}
+
\\
+ 
\left\|\Av_{\nu}(h) - \Iup[L_i]{h}{A_{i+1}'}\right\| + d_{A_{i+1}}
< \epsilon_{i} + e^{-t_{i}}\frac{R}{2} + \epsilon_{i+1} + 2d_{A_{i}} + 2d_{A_{i+1}}.
\end{multline*} 

Since the sum on the right clearly converges to $0$ as $i$ tends to $\infty$, so, even though the measures are different, their averages $\Av_{\mu}(h)$ and $\Av_{\nu}(h)$ are the same. Since we did not specify the function $h$ beyond the fact that it is 1-Lipschitz, we can take the function $\mathfrak{h}$ from Lemma \ref{lemma__core-1-Lipschitz-distinct-averages}. We have obtained a contradiction.

A completely analogous proof shows that the sequence $\{\Iup{h}{B_i}\}_{n\in\bbN_+}$ also converges to $\Av_{\bar{m}}(h)$ for some $\bar{m}=\mu_k$, $0\leq k\leq d-1$.

\emph{Step 2:}  If $\bar{m} = m$, the proof is complete. It remains to prove that we cannot have $\bar{m} \neq m$.

Recall that $h$ is $1$-Lipschitz, $|A_{i}'B_{i}'| = |v_{s}A_{i}'v_{s}B_{i}'|$ for any $s \in [0,\,L_{i}]$, and $|A_{i}'B_{i}'|\leq e^{-t_{i}}R$ for all $i>0$. Therefore, for any $i>0$,
\begin{multline*}
\left\|\Iup{h}{B_i'} - \Iup{h}{A_i'}\right\| 
= 
\left\|\frac{1}{L_{i}}\int_{0}^{L_{i}} h(v_{s}B_{i}')\,\rd s - \frac{1}{L_{i}}\int_{0}^{L_{i}} h(v_{s}A_{i}')\,\rd s \right\| 
\leq
\\
\leq
\frac{1}{L_{i}} \int_{0}^{L_{i}}
\left\|h(v_{s}B_{i}')-h(v_{s}A_{i}')\right\| \,\rd s 
\leq 
\frac{1}{L_{i}} \int_{0}^{L_{i}} e^{-t_{i}}R  \,\rd s 
= 
e^{-t_{i}}R.
\end{multline*} 
Therefore $\{\Iup{h}{A_i}'\}_{n\in\bbN_+}$ and $\{\Iup{h}{B_i'}\}_{n\in\bbN_+}$ have to converge to the same limit. Then the same is true for $\{\Iup{h}{A_i}\}_{n\in\bbN_+}$ and $\{\Iup{h}{B_i}\}_{n\in\bbN_+}$ from Lemma \ref{lemma__A-and-A-prime-limits}.
\end{proof}

\emph{Part VI. Backward flow limits.}

Note that the exact same construction as above can be repeated for the backward vertical flow and the sequence of pairs of points $\{(C_i,\,B_i)\}_{i\in\bbN_0}$, where $C_i=\varsigma^{-1}(A_i)$ for all $i\geq 0$. Analogously to Lemma \ref{lemma__A-and-B-limits}, we can obtain

\begin{lemma}\label{lemma__C-and-B-limits}
The sequence $\{\Idown{h}{C_i}\}_{n\in\bbN_+}$ converges to a finite limit, and so does the sequence $\{\Idown{h}{B_i}\}_{n\in\bbN_+}$. Moreover, these two limits are equal to each other and to $\Av_{\bar{m}}(h)$ for some $\bar{m}=\mu_k$, $0\leq k\leq d-1$.
\end{lemma}

Moreover, since $A_i,\,B_i$, and $C_i=\varsigma^{-1}(A_i)$ are, by construction, both totally generic and $\epsilon_i$-generic for every $i\geq 1$, we obtain
\begin{lemma}\label{lemma__up-and-down-limits}
The sequences $\{\Iup{h}{B_i}\}_{n\in\bbN_+}$ and $\{\Idown{h}{B_i}\}_{n\in\bbN_+}$ converge to the same limit. The same statement holds for the sequences $\{\Iup{h}{A_i}\}_{n\in\bbN_+}$ and $\{\Idown{h}{A_i}\}_{n\in\bbN_+}$, as well as the sequences $\{\Iup{h}{C_i}\}_{n\in\bbN_+}$ and $\{\Idown{h}{C_i}\}_{n\in\bbN_+}$.
\end{lemma}

From Lemmas \ref{lemma__A-and-B-limits}, \ref{lemma__C-and-B-limits}, and \ref{lemma__up-and-down-limits}, we conclude that, for every $h\in\Lip_1(\tX)$,
\begin{equation}
    \lim_{i\to\infty} \Iup{h}{A_i}
    \; = \;    
    \lim_{i\to\infty} \Idown{h}{A_i}
    \; = \;    
    \lim_{i\to\infty} \Iup{h}{\varsigma^{-1}(A_i)}
    \; = \;
    \lim_{i\to\infty} \Idown{h}{\varsigma^{-1}(A_i)} \; = \; \Av_m(h),
\end{equation}
for some $m=\mu_j$, $0\leq j\leq d-1$.

Since  $A_i$, and $\varsigma^{-1}(A_i)$ are, by construction, totally generic and  $\epsilon_i$-generic for every $i\geq 1$, we obtain a contradiction with Corollary \ref{cor__measure-generic-symmetry}.  
\end{proof}

\newpage

\section{The proof of the main conjecture}\label{UE-sect4}

\subsection{Embedded disks}\label{UE-sect4.1}

\begin{definition}
The \emph{embedded radius} of $X$ is
    $$
    r_{\mathrm{emb}}(X)
    :=
    \sup\Bigl\{ r>0 : \text{there exists an isometric embedding } B_r(0)\hookrightarrow X \Bigr\},
    $$
    where $B_r(0)\subset \mathbb{R}^2$ is the open Euclidean disk of radius $r$.
\end{definition}

\begin{theorem}\label{thm__embedded-radius-ae-points}
Let $X$ be a finite-area translation surface whose vertical flow is uniquely ergodic, and fix a marked point $P\in X$. Suppose that there exist $r>0$ and a sequence $t_k\to\infty$ such that
$$
r_{\mathrm{emb}}(X_{t_k})\ge r,
\qquad X_{t_k}:=g_{t_k}X.
$$
Then for Lebesgue-almost every point $Q\in X$ there exist a number $R(Q)>0$ and an infinite subsequence $t_{k_j}\to\infty$ such that, for every $j$, the surface $X_{t_{k_j}}$ contains an embedded Euclidean disk of radius at least $R(Q)$ centered at $g_{t_{k_j}}(Q)$, and this disk does not contain $g_{t_{k_j}}(P)$.
\end{theorem}

\begin{proof}
Fix $R$ with $0<R<r$.

For each $k$, choose an isometric embedding
$$
\iota_k:B_R(0)\hookrightarrow X_{t_k},
$$
and write
$$
D_k^R:=\iota_k(B_R(0)).
$$
We also consider the concentric disk
$$
D_k^{R/2}:=\iota_k(B_{R/2}(0))\subset D_k^R.
$$

We now define a measurable subset $A_k\subset D_k^{R/2}$ of uniformly positive area, all of whose points stay a definite Euclidean distance away from $g_{t_k}(P)$ inside the chosen chart.

If $g_{t_k}(P)\notin D_k^R$, set
$$
A_k:=D_k^{R/2}.
$$
If $g_{t_k}(P)\in D_k^R$, let
$$
p_k:=\iota_k^{-1}(g_{t_k}(P))\in B_R(0),
$$
and set
$$
A_k:=\iota_k\bigl(B_{R/2}(0)\setminus B_{R/4}(p_k)\bigr).
$$

In either case, $A_k\subset D_k^{R/2}$ is measurable. Moreover,
$$
\mathcal L(A_k)\ge \pi(R/2)^2-\pi(R/4)^2=\frac{3\pi R^2}{16}.
$$
Indeed, if $g_{t_k}(P)\notin D_k^R$, then $\mathcal L(A_k)=\pi(R/2)^2$. If $g_{t_k}(P)\in D_k^R$, then removing the set
$$
B_{R/2}(0)\cap B_{R/4}(p_k)
$$
can decrease the area by at most $\pi(R/4)^2$.

Now define
$$
F_k:=g_{-t_k}(A_k)\subset X.
$$
Since the \Teichmuller~flow preserves Lebesgue measure,
$$
\mathcal L_X(F_k)=\mathcal L(A_k)\ge \frac{3\pi R^2}{16}.
$$
Let
$$
U_N:=\bigcup_{k\ge N}F_k,
\qquad
S:=\bigcap_{N=1}^\infty U_N=\limsup_{k\to\infty}F_k.
$$
The sets $U_N$ decrease with $N$, and $U_N\supset F_N$ for every $N$. Therefore
$$
\mathcal L_X(S)
=
\lim_{N\to\infty}\mathcal L_X(U_N)
\ge
\frac{3\pi R^2}{16}>0.
$$

Now apply the Birkhoff--Khinchin pointwise ergodic Theorem \ref{thm__Birkhoff-Khinchin} for the vertical flow to the indicator function $\mathbf{1}_S$. Since the vertical flow is uniquely ergodic, it is ergodic for Lebesgue measure. Hence for Lebesgue-almost every point $Q\in X$,
$$
\lim_{T\to\infty}\frac1T\int_0^T \mathbf{1}_S(v_tQ)\,dt
=
\mathcal L_X(S)>0.
$$
In particular, for Lebesgue-almost every $Q$, the forward vertical orbit of $Q$ meets $S$ infinitely often. Fix such a point $Q$, and choose one time $s(Q)\ge0$ such that
$$
v_{s(Q)}Q\in S.
$$

Since $v_{s(Q)}Q\in S=\limsup_{k\to\infty}F_k$, there exists an infinite subsequence $t_{k_j}$ such that
$$
v_{s(Q)}Q\in \bigcap_{j=1}^{\infty} F_{k_j}.
$$
Equivalently,
$$
x_j:=g_{t_{k_j}}(v_{s(Q)}Q)\in A_{k_j}.
$$

Because $x_j\in D_{k_j}^{R/2}$, its distance to the boundary of $D_{k_j}^R$ is at least $R/2$. Also, if $g_{t_{k_j}}(P)\in D_{k_j}^R$, then by construction of $A_{k_j}$,
$$
d_{D_{k_j}^R}\bigl(x_j,\,g_{t_{k_j}}(P)\bigr)\ge \frac{R}{4},
$$
where $d_{D_{k_j}^R}$ denotes Euclidean distance in the chart $D_{k_j}^R$. If $g_{t_{k_j}}(P)\notin D_{k_j}^R$, then every subset of $D_{k_j}^R$ automatically avoids $g_{t_{k_j}}(P)$.

Now use the relation
$$
g_t\circ v_s=v_{e^{-t}s}\circ g_t.
$$
Set
$$
q_j:=g_{t_{k_j}}(Q).
$$
Then $q_j$ and $x_j$ lie on the same vertical orbit segment in $X_{t_{k_j}}$, and
$$
d(q_j,x_j)=e^{-t_{k_j}}\,s(Q)\longrightarrow 0.
$$
After discarding finitely many terms, we may assume that
$$
d(q_j,x_j)<\frac{R}{8}
\qquad\text{for every }j.
$$

We claim that the open metric disk
$$
B_{R/8}(q_j)
$$
is contained in $D_{k_j}^R$ and does not contain $g_{t_{k_j}}(P)$.

Let $y\in B_{R/8}(q_j)$. Then
$$
d(y,x_j)\le d(y,q_j)+d(q_j,x_j)<\frac{R}{8}+\frac{R}{8}=\frac{R}{4}.
$$
Since $x_j$ is at distance at least $R/2$ from $\partial D_{k_j}^R$, it follows that $y\in D_{k_j}^R$. Thus
$$
B_{R/8}(q_j)\subset D_{k_j}^R.
$$

If $g_{t_{k_j}}(P)\notin D_{k_j}^R$, then this disk certainly avoids $g_{t_{k_j}}(P)$. So suppose instead that
$$
g_{t_{k_j}}(P)\in D_{k_j}^R.
$$
Since both $x_j$ and $y$ lie in the embedded chart $D_{k_j}^R$, all distances between them are Euclidean chart distances. Therefore
$$
d_{D_{k_j}^R}\bigl(y,g_{t_{k_j}}(P)\bigr)
\ge
d_{D_{k_j}^R}\bigl(x_j,g_{t_{k_j}}(P)\bigr)-d_{D_{k_j}^R}(y,x_j)
>
\frac{R}{4}-\frac{R}{4}
=
0.
$$
Hence $y\neq g_{t_{k_j}}(P)$. Therefore
$$
g_{t_{k_j}}(P)\notin B_{R/8}(q_j).
$$

We have shown that $B_{R/8}(q_j)$ is contained in the embedded Euclidean disk $D_{k_j}^R$, so it is itself an embedded Euclidean disk in $X_{t_{k_j}}$. This proves the theorem, with
$$
R(Q):=\frac{R}{8}.
$$
\end{proof}

\begin{remark}
If $P$ is nonsingular and $Q$ lies on the same vertical trajectory as $P$, then the conclusion generally fails. Indeed, if
$$
Q=v_{-a}(P)
$$
for some $a\in\mathbb R$, then
$$
g_t(Q)=v_{-e^{-t}a}(g_t(P)),
$$
so
$$
d(g_t(Q),g_t(P))\to 0
\qquad\text{as }t\to\infty.
$$
Hence no fixed positive-radius disk centered at $g_t(Q)$ can avoid $g_t(P)$ for all large $t$. This does not affect the theorem, since a single vertical trajectory has Lebesgue measure zero.
\end{remark}

\begin{corollary}\label{cor__ae-slits-fixed-basepoint}
Let $X$ be a finite-area translation surface whose vertical flow is uniquely ergodic, and fix a marked point $P\in X$. Suppose that there exist $r>0$ and a sequence $t_k\to\infty$ such that
$$
r_{\mathrm{emb}}(X_{t_k})\ge r,
\qquad X_{t_k}:=g_{t_k}X.
$$
Then for Lebesgue-almost every point $Q\in X$, the $\mathrm{N}$-cover construction determined by the slit from $P$ to $Q$ is uniquely ergodic.
\end{corollary}

\begin{proof}
By Theorem~\ref{thm__embedded-radius-ae-points}, for Lebesgue-almost every point $Q\in X$ there exist a number $R(Q)>0$ and an infinite subsequence $t_{k_j}\to\infty$ such that, for every $j$, the surface $X_{t_{k_j}}$ contains an embedded Euclidean disk of radius at least $R(Q)$ centered at $g_{t_{k_j}}(Q)$, and this disk does not contain $g_{t_{k_j}}(P)$.

Therefore, the conditions of Theorem \ref{thm__The-CIRCLE_CRITERION} are satisfied. 
\end{proof}
Note that, apart from finite area and unique ergodicity of the vertical flow, the only additional requirement is the existence of a sequence $t_k\to\infty$ along which the embedded radius of $X_{t_k}$ is bounded below.

\subsection{Cylinders}\label{UE-sect4.2}

\begin{definition}[Flat cylinder]
A \emph{flat cylinder} (in direction $\theta$) is an open subset
$$
C\subset X\setminus\Sigma
$$
isometric, in translation coordinates, to
$$
(\mathbb R/c\mathbb Z)\times(0,h)
$$
with the Euclidean metric, for some $c,h>0$, in such a way that the curves $\{y=\mathrm{const}\}$ map to simple closed geodesics in direction $\theta$ on $X$.

We write:
$$
c(C):=\text{circumference},\qquad
h(C):=\text{height},\qquad
a(C):=c(C)h(C),\qquad
m(C):=\frac{h(C)}{c(C)}.
$$
\end{definition}

\begin{figure}[h]
\centering
\includegraphics[width=.55\linewidth]{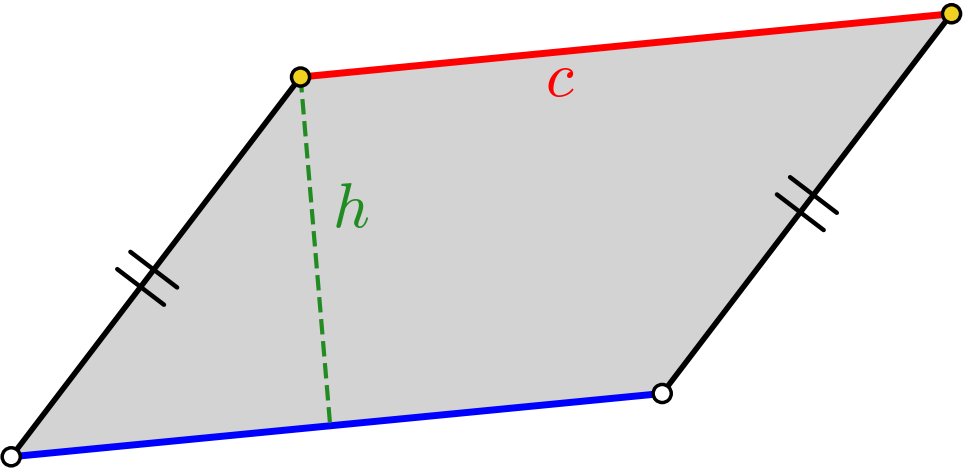}
\caption{A maximal cylinder.}
\label{fig__cylinder}
\end{figure}

\begin{definition}[Maximal flat cylinder]
A flat cylinder $C$ is \emph{maximal} if it is not properly contained in any larger flat cylinder in the same direction. Equivalently, it is a connected component of the set of points whose geodesic in direction $\theta$ is periodic. We call two maximal cylinders \emph{disjoint} if they have disjoint interiors.

By \cite[Lemma 1.6]{MasTab2002}, the closure $\overline C$ has two boundary components, each of which is a finite union of saddle connections in direction $\theta$.
\end{definition}

\begin{definition}
A maximal flat cylinder $C$ is called a \emph{pipe cylinder} if
$$
c(C)<h(C),
$$
and a \emph{perfect cylinder} if
$$
c(C)=h(C).
$$
\end{definition}

\begin{lemma}\label{lemma__disjoint-cylinders}
For a fixed translation surface $X$, any two distinct maximal pipe cylinders are disjoint.
\end{lemma}

\begin{proof}
Let $C_1$ and $C_2$ be two distinct maximal pipe cylinders. If they are parallel, then maximality implies that either they coincide or their interiors are disjoint. So only the non-parallel case needs to be studied.

Assume, for contradiction, that $C_1$ and $C_2$ intersect and are not parallel. Then a core curve of $C_1$ crosses $C_2$ from one boundary component of $C_2$ to the other, so its length is at least the height of $C_2$:
$$
c(C_1)\ge h(C_2).
$$
Similarly,
$$
c(C_2)\ge h(C_1).
$$
But since both cylinders are pipes,
$$
h(C_1)>c(C_1),
\qquad
h(C_2)>c(C_2).
$$
Combining these inequalities gives
$$
c(C_1)\ge h(C_2)>c(C_2)\ge h(C_1)>c(C_1),
$$
a contradiction. Hence the interiors of $C_1$ and $C_2$ are disjoint.
\end{proof}

\begin{lemma}\label{lemma__number-of-pipes}
For a fixed stratum, the number of maximal pipe cylinders on any surface in that stratum is bounded above by a constant depending only on the stratum.
\end{lemma}

\begin{proof}
Choose one core curve from each maximal pipe cylinder. By Lemma~\ref{lemma__disjoint-cylinders}, these core curves are pairwise disjoint.

Moreover, two such core curves cannot be homotopic in $X\setminus\Sigma$. Indeed, any two simple closed geodesics on a translation surface that avoid the singularities and are freely homotopic lie in the closure of the same maximal cylinder. Hence distinct maximal cylinders give non-homotopic core curves.

Thus the core curves form a collection of pairwise disjoint, pairwise non-homotopic, essential simple closed curves on the punctured surface $X\setminus\Sigma$. The size of such a collection is bounded above by the usual topological constant
$$
M=3g-3+n,
$$
where $g$ is the genus and $n=|\Sigma|$ is the number of singularities. This bound depends only on the stratum.
\end{proof}

\begin{proposition}[Masur--Smillie, \cite{MasurSmillie1991HD}, Proposition 5.4]\label{prop__delaunay}
The Delaunay triangulation of an area-$1$ translation surface $X$ consists of edges which either have length at most $2\sqrt{2/\pi}$ or cross a pipe cylinder $C$. If an edge crosses $C$, then its length lies in the interval
$$
\big[h(C),\,\sqrt{h(C)^2+c(C)^2}\big].
$$
\end{proposition}

\begin{lemma}\label{lemma__inscribed-triangles}
Suppose there exists a sequence $t_k\to\infty$ such that $X_{t_k}=g_{t_k}X$ contains no pipe cylinders. Then there exists a constant $R>0$ such that every $X_{t_k}$ contains an embedded disk of radius at least $R$.
\end{lemma}

\begin{proof}
Fix the stratum of $X$, and let $T$ be the number of triangles in a Delaunay triangulation of a surface in that stratum. This number depends only on the stratum.

Since the total area is $1$, at least one Delaunay triangle $\Delta$ has area
$$
A_\Delta\ge \frac1T.
$$
If $X_{t_k}$ contains no pipe cylinders, then Proposition~\ref{prop__delaunay} implies that every Delaunay edge has length at most $2\sqrt{2/\pi}$. Therefore the perimeter of $\Delta$ satisfies
$$
P_\Delta\le 6\sqrt{2/\pi}.
$$
The inradius of a Euclidean triangle is
$$
r_\Delta=\frac{2A_\Delta}{P_\Delta},
$$
so
$$
r_\Delta\ge \frac{2/T}{6\sqrt{2/\pi}}
= \frac{\sqrt{\pi}}{3\sqrt2\,T}.
$$
Since the interior of each Delaunay triangle is isometrically embedded in the surface, the inscribed Euclidean disk in $\Delta$ is an embedded disk in $X_{t_k}$. Thus one may take
$$
R:=\frac{\sqrt{\pi}}{3\sqrt2\,T}.
$$
\end{proof}

Let $\mathfrak{P}_t$ denote the union of the closures of all pipe cylinders on $X_t=g_tX$, and let $A(P_t)$ be its area.

\begin{lemma}\label{lemma__big-area-pipes}
Assume that the condition of Lemma~\ref{lemma__inscribed-triangles} fails. Suppose there exist a constant $K>0$ and a sequence $t_k\to\infty$ such that
$$
A(\mathfrak{P}_{t_k})>K
\qquad\text{for all }k.
$$
Then there exist a constant $R>0$ and a sequence $t_k'\to\infty$ such that each $X_{t_k'}$ contains an embedded disk of radius at least $R$.
\end{lemma}

\begin{proof}
By Lemmas~\ref{lemma__disjoint-cylinders} and \ref{lemma__number-of-pipes}, for each $k$ the pipe cylinders on $X_{t_k}$ are pairwise disjoint and there are at most $\mathrm{N}$ of them, where $\mathrm{N}$ depends only on the stratum. Hence at least one pipe cylinder $C_k\subset X_{t_k}$ has area
$$
a(C_k)\ge \frac{K}{M}.
$$

Let $c_k(s)$ and $h_k(s)$ be the circumference and height of the image of $C_k$ after flowing forward for time $s\ge 0$. The area is preserved, so
$$
c_k(s)\,h_k(s)=a(C_k)\ge \frac{K}{M}.
$$
At time $s=0$, $C_k$ is a pipe cylinder, so $c_k(0)<h_k(0)$. As $s\to+\infty$, the circumference in its periodic direction tends to $+\infty$, so by continuity there exists $s_k\ge 0$ such that
$$
c_k(s_k)=h_k(s_k)=\sqrt{a(C_k)}\ge \sqrt{\frac{K}{M}}.
$$
Thus, on the later surface
$$
X_{t_k'}:=X_{t_k+s_k},
$$
the image of $C_k$ is a perfect cylinder of height and circumference at least $\sqrt{K/M}$. Any flat cylinder of height $L$ and circumference $L$ contains an embedded Euclidean disk of radius $L/2$, so $X_{t_k'}$ contains an embedded disk of radius at least
$$
R:=\frac12\sqrt{\frac{K}{M}}.
$$
Since $t_k'\ge t_k\to\infty$, we have $t_k'\to\infty$.
\end{proof}

\begin{lemma}\label{lemma__small-area-pipes}
Assume that the conditions of Lemmas~\ref{lemma__inscribed-triangles} and \ref{lemma__big-area-pipes} both fail. In other words, assume that for all sufficiently large $t$ the surface $X_t=g_tX$ contains at least one pipe cylinder, and that
$$
A(\mathfrak{P}_t)\longrightarrow 0
\qquad\text{as }t\to\infty.
$$
Then there exist $R>0$ and $\tau>0$ such that for every $t>\tau$, the surface $X_t$ contains an embedded disk of radius at least $R$.
\end{lemma}

\begin{proof}
Let $\Sigma_t=\{g_t(s_1),\dots,g_t(s_n)\}$ be the singular set of $X_t$, and let
$$
\Sigma_t(\delta):=\{x\in X_t:d(x,\Sigma_t)\le \delta\}.
$$
If the cone angle at $g_t(s_i)$ is $2\pi(k_i+1)$, then the Euclidean area of the radius-$\delta$ neighborhood of $g_t(s_i)$ is $\pi(k_i+1)\delta^2$. Hence
$$
\mathcal L(\Sigma_t(\delta))
\le
\pi\Bigl(\sum_{i=1}^n (k_i+1)\Bigr)\delta^2
=
\pi(2g-2+n)\delta^2.
$$
Let
$$
C_{\mathrm{str}}:=\pi(2g-2+n).
$$

Choose $\varepsilon>0$ and $\delta>0$ so that
$$
1-\varepsilon-C_{\mathrm{str}}\delta^2>0.
$$
Since $A(\mathfrak{P}_t)\to 0$, there exists $\tau>0$ such that
$$
A(\mathfrak{P}_t)<\varepsilon
\qquad\text{for all }t>\tau.
$$
For such $t$, the set
$$
F_\delta(t):=X_t\setminus \bigl(\mathfrak{P}_t\cup \Sigma_t(\delta)\bigr)
$$
has positive area, hence is nonempty. Fix $x\in F_\delta(t)$.

Because $d(x,\Sigma_t)>\delta$, there is a well-defined local developing map
$$
\iota_x:B_\delta(0)\to X_t
$$
which is a local isometry and satisfies $\iota_x(0)=x$.

We claim that the restriction of $\iota_x$ to $B_{\delta/4}(0)$ is injective. Suppose not. Then there exist distinct points $u,v\in B_{\delta/4}(0)$ such that
$$
\iota_x(u)=\iota_x(v).
$$
In translation coordinates, the two local branches differ by a translation
$$
T(z)=z+w,
\qquad w:=v-u.
$$
Since
$$
|w|=|v-u|<\frac{\delta}{2},
$$
we have $w\in B_\delta(0)$, and therefore
$$
\iota_x(w)=\iota_x(0)=x.
$$
So the straight segment from $0$ to $w$ projects to a simple closed geodesic $\gamma$ based at $x$ of length
$$
c:=|w|<\frac{\delta}{2}.
$$

Let $C$ be the maximal cylinder containing $\gamma$. Let $a$ and $b$ be the distances from $\gamma$ to the two boundary components of $C$, so that
$$
h(C)=a+b.
$$
Each boundary component has total length $c$. Therefore on each boundary component one can find a singularity whose tangential displacement from the perpendicular through $x$ is at most $c/2$. It follows that there are singularities at Euclidean distance at most
$$
\sqrt{a^2+(c/2)^2}
\qquad\text{and}\qquad
\sqrt{b^2+(c/2)^2}
$$
from $x$. Since $x\notin \Sigma_t(\delta)$, every singularity is at distance bigger than $\delta$ from $x$, hence
$$
a>\sqrt{\delta^2-(c/2)^2},
\qquad
b>\sqrt{\delta^2-(c/2)^2}.
$$
Therefore
$$
h(C)=a+b>2\sqrt{\delta^2-(c/2)^2}.
$$
Because $c<\delta/2$, the right-hand side is strictly larger than $c$, so
$$
h(C)>c.
$$
Thus $C$ is a pipe cylinder. But $x\in \gamma\subset C\subset \mathfrak{P}_t$, contradicting the choice $x\in X_t\setminus \mathfrak{P}_t$.

This contradiction proves that $\iota_x|_{B_{\delta/4}(0)}$ is injective. Hence $X_t$ contains an embedded disk of radius $\delta/4$ centered at $x$.

Since $t>\tau$ was arbitrary, the conclusion holds with
$$
R=\frac{\delta}{4}.
$$
\end{proof}

\printbibliography


\end{document}